\documentclass[11pt]{article}%
\usepackage{amsfonts}
\usepackage{amsmath}
\usepackage{amssymb}
\usepackage{extarrows}
\usepackage{color}
\usepackage{times}
\usepackage{bbold}
\usepackage{palatino}
\usepackage[colorlinks]{hyperref}
\usepackage[colorinlistoftodos]{todonotes}
\usepackage{graphicx}%
\providecommand{\U}[1]{\protect\rule{.1in}{.1in}}
\providecommand{\U}[1]{\protect\rule{.1in}{.1in}}
\hypersetup{colorlinks=blue, linkcolor=cyan, citecolor=green,
filecolor=black, urlcolor=blue }
\newtheorem{theorem}{Theorem}

\newtheorem{definition}[theorem]{Definition}

\newtheorem{remark}[theorem]{Remark}

\newenvironment{proof}[1][Proof]{\noindent\textbf{#1.} }{\ \rule{0.5em}{0.5em}}

\renewcommand{\thefootnote}{\fnsymbol{footnote}}
\begin{document}

\title{A stochastic approach to path--dependent nonlinear Kolmogorov equations via
BSDEs with time--delayed generators and applications to finance}
\author{Francesco Cordoni$^{a}$, Luca Di Persio$^{b}$, Lucian Maticiuc$^{c}$, Adrian Z\u{a}linescu$^{d}$
\bigskip\\
{\small $^{a}$ Department of Mathematics, University of Trento,}\\{\small Via Sommarive, no. 14, Trento, 38123, Italy}\\{\small $^{b}$ Department of Computer Science, University of Verona,}\\{\small Strada le Grazie, no. 15, Verona, 37134, Italy}\\{\small $^{c}$ Faculty of Mathematics, \textquotedblleft Alexandru Ioan
Cuza\textquotedblright\ University,}\\{\small Carol I Blvd., no. 11, Ia\c{s}i, 700506, Romania}\\{\small $^{d}$ \textquotedblleft Octav Mayer \textquotedblright Mathematics
Institute of the Romanian Academy,}\\{\small Carol I Blvd., no. 8, Ia\c{s}i, 700506, Romania}}
\maketitle

\begin{abstract}
We prove the existence of a viscosity solution of the following path dependent
nonlinear Kolmogorov equation:
\[%
\begin{cases}
\partial_{t}u(t,\phi)+\mathcal{L}u(t,\phi)+f(t,\phi,u(t,\phi),\partial
_{x}u(t,\phi) \sigma(t,\phi),(u(\cdot,\phi)) _{t})=0,\;t\in[0,T),\;\phi
\in\mathbb{\Lambda}\, ,\\
u(T,\phi)=h(\phi),\;\phi\in\mathbb{\Lambda},
\end{cases}
\]
where $\mathbb{\Lambda}=\mathcal{C}([0,T];\mathbb{R}^{d})$, $(u(\cdot
,\phi))_{t}:=(u(t+\theta,\phi))_{\theta\in[-\delta,0]}$ and
\[
\mathcal{L}u(t,\phi):=\langle b(t,\phi),\partial_{x}u(t,\phi)\rangle+\dfrac
{1}{2}\mathrm{Tr}\big[\sigma(t,\phi)\sigma^{\ast}(t,\phi)\partial_{xx}
^{2}u(t,\phi)\big].
\]
The result is obtained by a stochastic approach. In particular we prove a new
type of nonlinear Feynman--Kac representation formula associated to a backward
stochastic differential equation with time--delayed generator which is of
non--Markovian type. Applications to the large investor problem and risk
measures via $g$--expectations are also provided.

\end{abstract}


\renewcommand{\thefootnote}{\fnsymbol{footnote}}
\footnotetext{{\scriptsize E-mail addresses: francesco.cordoni@unitn.it
(Francesco Cordoni), luca.dipersio@univr.it (Luca Di Persio),
lucian.maticiuc@tuiasi.ro (Lucian Maticiuc), adrian.zalinescu@gmail.com
(Adrian Z\u{a}linescu)}}

\textbf{AMS Classification subjects:} 35D40, 35K10, 60H10, 60H30\medskip

\textbf{Keywords or phrases: }Path--dependent partial differential equations;
viscosity solutions; Feynman--Kac formula; backward stochastic differential
equations; time--delayed generators.

\section{Introduction}

\label{Section 1}We aim at providing a probabilistic representation of a
viscosity solution to the following path--dependent nonlinear Kolmogorov
equation (PDKE)
\begin{equation}
\left\{
\begin{array}{r}
-\partial _{t}u(t,\phi )-\mathcal{L}u(t,\phi )-f(t,\phi ,u(t,\phi ),\partial
_{x}u\left( t,\phi \right) \sigma (t,\phi ),\left( u(\cdot ,\phi )\right)
_{t})=0,\medskip \\
\multicolumn{1}{l}{u(T,\phi )=h(\phi ),}%
\end{array}%
\right.  \label{PDKE}
\end{equation}%
for $T<\infty $ a fixed time horizon, $t\in \lbrack 0,T)$, $\phi \in \mathbb{%
\Lambda }:=\mathcal{C}([0,T];\mathbb{R}^{d})$, the space of continuous $%
\mathbb{R}^{d}$--valued functions defined on the interval $[0,T]$. For a
fixed delay $\delta >0,$ we have set%
\begin{equation*}
\left( u(\cdot ,\phi )\right) _{t}:=\left( u(t+\theta ,\phi )\right)
_{\theta \in \left[ -\delta ,0\right] }\,.
\end{equation*}%
In equation \eqref{PDKE} we have denoted by $\mathcal{L}$ the second order
differential operator given by
\begin{equation*}
\mathcal{L}u(t,\phi ):=\dfrac{1}{2}\mathrm{Tr}\big[\sigma (t,\phi )\sigma
^{\ast }(t,\phi )\partial _{xx}^{2}u(t,\phi )\big]+\langle b(t,\phi
),\partial _{x}u(t,\phi )\rangle ,
\end{equation*}%
with $b:\left[ 0,T\right] \times \mathbb{\Lambda }\rightarrow \mathbb{R}^{d}$
and $\sigma :\left[ 0,T\right] \times \mathbb{\Lambda }\rightarrow \mathbb{R}%
^{d\times d^{\prime }}$ being two non--anticipative functionals to be better
introduced in subsequent section.

We will prove that, under appropriate assumptions on the coefficients, the
deterministic non--anticipative functional $u:\left[ 0,T\right] \times
\mathbb{\Lambda }\rightarrow \mathbb{R}$ given by the representation formula%
\begin{equation*}
u\left( t,\phi \right) :=Y^{t,\phi }\left( t\right)
\end{equation*}%
is a viscosity solution, in the sense given in \cite{ek-ke-to-zh/14}, to
equation (\ref{PDKE}), where%
\begin{equation*}
(X^{t,\phi }\left( s\right) ,Y^{t,\phi }\left( s\right) ,Z^{t,\phi }\left(
s\right) )_{s\in \left[ t,T\right] }
\end{equation*}%
is the unique solution to the decoupled forward--backward stochastic
differential system on $\left[ t,T\right] $%
\begin{equation}
\left\{
\begin{array}{l}
\displaystyle X^{t,\phi }\left( s\right) =\phi \left( t\right)
+\int_{t}^{s}b(r,X^{t,\phi })dr+\int_{t}^{s}\sigma (r,X^{t,\phi })dW\left(
r\right) ,\medskip \\
\displaystyle Y^{t,\phi }\left( s\right) =h(X^{t,\phi
})+\int_{s}^{T}f(r,X^{t,\phi },Y^{t,\phi }\left( r\right) ,Z^{t,\phi }\left(
r\right) ,Y_{r}^{t,\phi })dr-\int_{s}^{T}Z^{t,\phi }\left( r\right) dW\left(
r\right) ,%
\end{array}%
\right.  \label{FBSDE}
\end{equation}%
with $\left( t,\phi \right) \in \left[ 0,T\right] \times \mathbb{\Lambda },$
where $W$ is a standard Brownian motion.

Here, the notation $Y_{r}^{t,\phi }$ appearing in the generator $f$ of the
backward component in system \eqref{FBSDE} stands for the path of the
process $Y^{t,\phi }$ restricted to $[r-\delta ,r]$, namely%
\begin{equation*}
Y_{r}^{t,\phi }:=(Y^{t,\phi }(r+\theta ))_{\theta \in \left[ -\delta ,0%
\right] }\,.
\end{equation*}%
In system (\ref{FBSDE}) the forward equation is a functional stochastic
differential equation, while the backward equation has time--delayed
generator, that is the generator $f$ can depend, unlike the classical
backward stochastic differential equations (BSDEs), on the past values of $%
Y^{t,\phi }$. %

Let us stress that if we do not consider delay neither in the forward nor in
the backward component, we retrieve standard results of Markovian
forward--backward system, hence in this case $u(t,\phi )=u(t,\phi \left(
t\right) )$ and equation (\ref{PDKE}) becomes%
\begin{equation*}
\left\{
\begin{array}{r}
-\partial _{t}u(t,x)-\mathcal{L}u(t,x)-f(t,x,u(t,x),\partial _{x}u\left(
t,x\right) \sigma (t,x))=0,\medskip \\
\multicolumn{1}{l}{u(T,x)=h(x),\quad t\in \lbrack 0,T),\;x\in \mathbb{R}^{d},
}%
\end{array}%
\right.
\end{equation*}%
with
\begin{equation*}
\mathcal{L}u(t,x):=\dfrac{1}{2}\mathrm{Tr}\big[\sigma (t,x)\sigma ^{\ast
}(t,x)\partial _{xx}^{2}u(t,x)\big]+\langle b(t,x),\partial
_{x}u(t,x)\rangle \,.
\end{equation*}

Let us recall that BSDEs were first introduced by Bismut \cite{bi/73}, in
the linear case, whereas the nonlinear case was considered by Pardoux and
Peng in \cite{pa-pe/90}. Later, in \cite{pa-pe/92,pe/91}, the connection
between BSDEs and semilinear parabolic partial differential equations (PDEs)
was established, proving the nonlinear Feynman--Kac formula for Markovian
equations. Also, similar deterministic representations associated with
suitable PDEs can be proved by taking into account different types of BSDEs,
such as BSDEs with random terminal time, see, e.g. \cite{da-pa/97},
reflected BSDEs, see, e.g. \cite{ka-ka/97}, or also backward stochastic
variational inequalities, see, e.g. \cite{ma-ra/10, ma-ra/15}.

When considering the non-Markovian case, the associated PDE becomes
path-dependent; this is shown for the first time in \cite{pe/10}.
Subsequently, in \cite{pe-wa/11} the authors proved, in the case of smooth
coefficients, the existence and uniqueness of a classical solution for the
path--de\-pen\-dent Kolmogorov equation%
\begin{equation}
\left\{
\begin{array}{r}
-\partial _{t}u(t,\phi )-\dfrac{1}{2}\partial _{xx}^{2}u(t,\phi )-f(t,\phi
,u(t,\phi ),\partial _{x}u\left( t,\phi \right) )=0,\medskip \\
\multicolumn{1}{l}{u(T,\phi )=h(\phi ),\quad t\in \lbrack 0,T),\;\phi \in
\mathbb{\Lambda }.}%
\end{array}%
\right.  \label{PDKE 2}
\end{equation}%
In particular, the authors appealed to a representation formula using the
non--Markovian BSDE:%
\begin{equation}
Y^{t,\phi }\left( s\right) =h(W^{t,\phi })+\int_{s}^{T}f(W^{t,\phi
},Y^{t,\phi }\left( r\right) ,Z^{t,\phi }\left( r\right)
)dr-\int_{s}^{T}Z^{t,\phi }\left( r\right) dW\left( r\right) ,\;~s\in \left[
t,T\right] ,  \label{BSDE time delay_introd 3}
\end{equation}%
with the generator and the final condition depending on the Brownian paths%
\begin{equation*}
\begin{array}{l}
\displaystyle W^{t,\phi }\left( s\right) =\phi \left( t\right) +W\left(
s\right) -W\left( t\right) ,\quad \text{if }s\in \left[ t,T\right] ,\quad
\text{and}\medskip \\
\displaystyle W^{t,\phi }\left( s\right) =\phi \left( s\right) ,\quad \text{%
if }s\in \lbrack 0,t).%
\end{array}%
\end{equation*}%
After that, in order to deal with such PDEs, in \cite{pe/12} a new type of
viscosity solution is introduced.

The definition of viscosity solution we will adopt in the present work is
that introduced in \cite{ek-ke-to-zh/14,ek-to-zh/14, ek-to-zh/14x} for
semilinear and fully non linear path-dependent PDE, by using the framework
of functional It\^{o} calculus first set by Dupire \cite{du/09} and Cont \&
Fourni\'{e} \cite{co-fo/13}.

In the present work we will generalize the results in \cite{ek-ke-to-zh/14}
along two directions. First we will consider a BSDE whose generator depends
not only on past values assumed by a standard Brownian motion $W$, but the
BSDE may depend on a general diffusion process $X$. Secondly, and perhaps
the most important result, we will prove the connection between
path--dependent PDEs and BSDEs with time--delayed generators. We recall that
time--delayed BSDEs were first introduced in \cite{de-im/10} and \cite%
{de-im/10x} where the authors obtained the existence and uniqueness of the
solution of the time--delayed BSDE
\begin{equation}
Y\left( t\right) =\xi
+\int_{t}^{T}f(s,Y_{s},Z_{s})ds-\int_{t}^{T}Z(s)dW(s),\quad 0\leq t\leq T,
\label{BSDE time delay_introd 2}
\end{equation}%
where%
\begin{equation*}
Y_{s}:=(Y(r))_{r\in \left[ 0,s\right] }\quad \text{and}\quad
Z_{s}:=(Z(r))_{r\in \left[ 0,s\right] }\,.
\end{equation*}%
In particular, the aforementioned existence and uniqueness result holds true
if the time horizon $T$ or the Lipschitz constant for the generator $f$ are
sufficiently small. A generalization of the existence and uniqueness theorem
for these kind of BSDEs was made in \cite{ch-na/17} by considering the
equation%
\begin{equation*}
Y\left( t\right) =\xi +\int_{t}^{T}f(s,Y_{s},M_{s})ds+M(T)-M\left( t\right)
,\;0\leq t\leq T,
\end{equation*}%
where the solution is the pair $\left( Y,M\right) $\ such that $Y$\ is an
adapted process and $M$\ is a martingale.$\smallskip $

To our best knowledge, the link between time--delayed BSDEs and
path--dependent PDEs has never been addressed in literature.

We also emphasize that our framework, since the BSDE under study is
time--delayed, requires that the backward equation contains a supplementary
initial condition to be satisfied, namely%
\begin{equation*}
Y^{t,\phi }\left( s\right) =Y^{s,\phi }\left( s\right) ,\quad \text{if }s\in
\lbrack 0,t).
\end{equation*}%
Let us further stress that the Feynman--Kac formula would fail with standard
prolongation $Y^{t,\phi }\left( s\right) =Y^{t,\phi }\left( t\right) $, for $%
s\in \lbrack 0,t)\,$. Although the existence results for equation (\ref{BSDE
time delay_introd 2}) have already been treated in \cite{de-im/10, de-im/10x}%
, our new initial condition imposes a more elaborated proof.$\medskip $

The last part of the paper presents two financial models based on our
theoretical results. In recent years delay equations have been of growing
interest, mainly motivated by many concrete applications where the effect of
delay cannot be neglected, see, e.g. \cite{mo/96, uv-pl/07}. On the
contrary, BSDEs with time delayed generator have been first introduced as a
pure mathematical tool, with no application of interest. Only later in \cite%
{de/11, de/12} the author proposed some financial applications to pricing,
hedging and investment portfolio management, where backward equations with
delayed generator provide a fundamental tool.

Based on the recently introduced path--dependent calculus, together with the
mild assumptions of differentiability required, the probabilistic
representation for a viscosity solution of a non-linear parabolic equations
proved in the present paper finds perhaps its best application in finance.
In fact, a wide variety of financial derivatives can be formally treated
under the theory developed in what follows, from the more standard European
options, to the more exotic path--dependent options, such as Asian options
or look-back options.

We propose here two possible applications of forward--backward stochastic
differential system (\ref{FBSDE}), where the delay in the backward component
arises from two different motivations. The first example we will deal with
is a generalization of a well-known model in finance, where we will consider
the case of a non standard investor acting on a financial market. We will
assume, following \cite{cv-ma/96, ka-pe-qu/01}, that a so called \textit{%
large investor} wishes to invest on a given market, buying or selling a
stock. This investor has the peculiarity that his actions on the market can
affect the stock price. In particular, we will assume that the stock price $%
S $ and the bond $B$ are a function of the large investor's portfolio $%
(X,\pi ) $, $X$ being the value of the portfolio and $\pi $ the amount
invested in the risky asset $S$.

This case has already been treated in the financial literature, see, e.g.
\cite{ka-pe-qu/01}. We further generalize the aforementioned results
assuming a second market imperfection, that is we assume that it might be a
small time delay between the action of the large investor and the reaction
of the market, so that we are led to consider the financial system with the
presence of the past of $X$ in the coefficients $r,\mu ,\sigma :$%
\begin{equation*}
\left\{
\begin{array}{l}
\displaystyle\frac{dB(t)}{B(t)}=r(t,X(t),\pi \left( t\right)
,X_{t})dt\;,\quad B(0)=1\;,\medskip \\
\displaystyle\frac{dS(t)}{S(t)}=\mu \left( t,X(t),\pi \left( t\right)
,X_{t}\right) dt+\sigma (t,X(t),X_{t})dW(t),\quad S(0)=s_{0}>0,%
\end{array}%
\right.
\end{equation*}%
where the notation $X_{t}$ stands for the path $\left( X(t+\theta )\right)
_{\theta \in \lbrack -\delta ,0]}$, $\delta $ being a small enough delay.

The second example we deal with arises from a different situation. Recent
literature in financial mathematics has been focused in how to measure the
riskiness of a given financial investment. To this extent \textit{dynamic
risk measures} have been introduced in \cite{ar-de-eb/02}. In particular,
BSDEs have been shown to be perhaps the best mathematical tool for modelling
dynamic risk measures, via the so called $g$\textit{-expectations}. In \cite%
{de/12} the author proposed a risk measure that takes also into account the
past values assumed by the investment, that is we will suppose that, in
making his future choices, the investor will consider not only the present
value of the investment, but also the values assumed in a sufficiently small
past interval. This has been motivated by empirical studies that show how
the memory effect has a fundamental importance in an investor's choices,
see, e.g. \cite{de/12} and references therein for financial studies on the
memory effect in financial investment. We therefore consider an investor
that tries to quantify the riskiness of a given investment, with $Y$ being
his investment, and will assume that the investor looks at the average value
of his investment in a sufficiently recent past, that is we consider a
generator of the form $\frac{1}{\delta }\,g_{1}\big(\int_{-\delta
}^{0}Y(t+\theta )d\theta \big)\,g_{2}\left( Z\left( t\right) \right) ,$ with
$\delta >0$ being a sufficiently small delay.\medskip

The paper is organized as follows: in Section \ref{SEC:Pre} we introduce
notions which are needed for the functional It\^{o}'s calculus and also the
notion of viscosity solution for path--dependent PDEs. In Section \ref%
{Section 2} we prove the existence and uniqueness of a solution for the
time-delayed BSDE, whereas Section \ref{Section 3} is devoted to the main
results of the present work, that is the proof of the continuity of the
function $u\left( t,\phi \right) :=Y^{t,\phi }\left( t\right) $ as well as
the nonlinear Feynman--Kac formula, that is Theorem \ref{theorem 1}, the
core of the present work. In Section \ref{SEC:FA} we eventually present the
financial applications.

\section{Preliminaries}

\label{SEC:Pre}

\subsection{Pathwise derivatives and functional It\^{o}'s formula}

Let us first introduce the framework on which we shall construct the
solutions of PDKE (\ref{PDKE}). For a deep treatment of functional It\^{o}
calculus we refer the reader to Dupire \cite{du/09} and Cont \& Fourni\'{e}
\cite{co-fo/13}.

Let $\mathbb{\hat{\Lambda}}:=\mathbb{D}\left( \left[ 0,T\right] ;\mathbb{R}%
^{d}\right) $ be the set of c\`{a}dl\`{a}g (i.e., right continuous, with
finite left--hand limits) $\mathbb{R}^{d}$--valued functions, $\hat{B}$ the
canonical process on $\mathbb{\hat{\Lambda}}$, i.e. $\hat{B}(t,\hat{\phi}):=%
\hat{\phi}\left( t\right) $ and $\mathbb{\hat{F}}:=(\mathcal{\hat{F}}%
_{s})_{s\in \left[ 0,T\right] }$ the filtration generated by $\hat{B}$. On $%
\mathbb{\hat{\Lambda}}$ and $\left[ 0,T\right] \times \mathbb{\hat{\Lambda}}$%
, we introduce the following norm and respectively pseudometric, with
respect to whom it becomes a Banach space, respectively a complete
pseudometric space. In this regard, we define, for any $(t,\hat{\phi}%
),(t^{\prime },\hat{\phi}{^{\prime }})\in \left[ 0,T\right] \times \mathbb{%
\hat{\Lambda}}$,
\begin{equation*}
\begin{array}{l}
||\hat{\phi}||_{T}:=\sup_{r\in \left[ 0,T\right] }|\hat{\phi}\left( r\right)
|,\medskip \\
d((t,\hat{\phi}),(t^{\prime },\hat{\phi}{^{\prime }})):=|t-t^{\prime
}|+\sup_{r\in \left[ 0,T\right] }|\hat{\phi}\left( r\wedge t\right) -\hat{%
\phi}{^{\prime }}\left( r\wedge t^{\prime }\right) |\,.%
\end{array}%
\end{equation*}

Let $\hat{u}:\left[ 0,T\right] \times \mathbb{\hat{\Lambda}}\rightarrow
\mathbb{R}$ be an $\mathbb{\hat{F}}$--progressively measurable
non-anticipative process, that is $\hat{u}(t,\hat{\phi})$ depends only on
the restriction of $\hat{\phi}$ on $\left[ 0,t\right] $, i.e. $\hat{u}(t,%
\hat{\phi})=\hat{u}(t,\hat{\phi}\left( \cdot \wedge t\right) )$, for any $(t,%
\hat{\phi})\in \left[ 0,T\right] \times \mathbb{\hat{\Lambda}}$. We say that
$\hat{u}$ is vertically differentiable at $(t,\hat{\phi})\in \left[ 0,T%
\right] \times \mathbb{\hat{\Lambda}}$ if there exists%
\begin{equation*}
\partial _{x_{i}}\hat{u}(t,\hat{\phi}):=\lim_{h\rightarrow 0}\displaystyle%
\frac{\hat{u}(t,\hat{\phi}+h\mathbb{1}_{\left[ t,T\right] }e_{i})-\hat{u}(t,%
\hat{\phi})}{h}
\end{equation*}%
for any $i=\overline{1,d}$, where we have denoted by $\left\{ e_{i}\right\}
_{i=\overline{1,d}}$ the canonical basis of $\mathbb{R}^{d}$. The second
order derivatives, when they exist, are denoted by $\partial
_{x_{i}x_{j}}^{2}\hat{u}(t,\hat{\phi}):=\partial _{x_{i}}(\partial _{x_{j}}%
\hat{u})$, for any $i,j=\overline{1,d}$. Let us further denote by $\partial
_{x}\hat{u}(t,\hat{\phi})$ the gradient vector, that is we have $\partial
_{x}\hat{u}(t,\hat{\phi})=\left( \partial _{x_{1}}\hat{u}(t,\hat{\phi}%
),\dots ,\partial _{x_{d}}\hat{u}(t,\hat{\phi})\right) ,$ and by $\partial
_{xx}^{2}\hat{u}(t,\hat{\phi})$ the $d\times d$--Hessian matrix, that is $%
\partial _{xx}^{2}\hat{u}(t,\hat{\phi})=\left( \partial _{x_{i}x_{j}}^{2}%
\hat{u}(t,\hat{\phi})\right) _{i,j=\overline{1,d}}\,.$

For $t\in \lbrack 0,T]$ and a path $\phi \in \mathbb{\hat{\Lambda}}$, we
denote%
\begin{equation}
\phi _{(t)}:=\phi (\cdot \wedge t)\in \mathbb{\hat{\Lambda}}.
\label{definition delayed term 1}
\end{equation}%
We say that $\hat{u}$ is horizontally differentiable at $(t,\hat{\phi})\in %
\left[ 0,T\right] \times \mathbb{\hat{\Lambda}}$ if there exists%
\begin{equation*}
\partial _{t}\hat{u}(t,\hat{\phi}):=\lim_{h\rightarrow 0_{+}}\displaystyle%
\frac{\hat{u}(t+h,\hat{\phi}_{\left( t\right) })-\hat{u}(t,\hat{\phi})}{h},
\end{equation*}%
for $t\in \lbrack 0,T)$ and $\partial _{t}\hat{u}(T,\hat{\phi}%
):=\lim_{t\rightarrow T_{-}}\partial _{t}\hat{u}(t,\hat{\phi})$.

If $\hat{u}:\left[ 0,T\right] \times \mathbb{\hat{\Lambda}}\rightarrow
\mathbb{R}$ is non-anticipative, we write $\hat{u}\in \mathcal{C}(\left[ 0,T%
\right] \times \mathbb{\hat{\Lambda}})$ if $\hat{u}$ is continuous on $\left[
0,T\right] \times \mathbb{\hat{\Lambda}}$ under the pseudometric $d;$ we
write that $\hat{u}\in \mathcal{C}_{b}(\left[ 0,T\right] \times \mathbb{\hat{%
\Lambda}})$ if $\hat{u}\in \mathcal{C}(\left[ 0,T\right] \times \mathbb{\hat{%
\Lambda}})$ and $\hat{u}$ is bounded on $\left[ 0,T\right] \times \mathbb{%
\hat{\Lambda}}$. Eventually we write $\hat{u}\in \mathcal{C}_{b}^{1,2}(\left[
0,T\right] \times \mathbb{\hat{\Lambda}})$ if $\hat{u}\in \mathcal{C}(\left[
0,T\right] \times \mathbb{\hat{\Lambda}})$ and the derivatives $\partial _{x}%
\hat{u}$, $\partial _{xx}^{2}\hat{u}$, $\partial _{t}\hat{u}$ exist and they
are continuous and bounded.

Having introduced the above notations by following \cite{du/09}, we will now
work with processes $u:\left[ 0,T\right] \times \mathbb{\Lambda }\rightarrow
\mathbb{R},$ with $\mathbb{\Lambda }$ being the space of continuous paths,
i.e.%
\begin{equation*}
\mathbb{\Lambda }=\mathcal{C}([0,T];\mathbb{R}^{d}).
\end{equation*}%
Let $B$ be the canonical process on $\Lambda $, i.e. $B(t,\phi ):=\phi
\left( t\right) $ and $\mathbb{F}:=(\mathcal{F}_{s})_{s\in \left[ 0,T\right]
}$ the filtration generated by $B$.

From the fact that $\mathbb{\Lambda }$ is a closed subspace of $\mathbb{\hat{%
\Lambda}}$, we have $(\Lambda ,||\cdot ||_{T})$ is also a Banach space;
similarly, we claim that $(\left[ 0,T\right] \times \mathbb{\Lambda },d)$ is
a complete pseudometric space. As above, if $u:\left[ 0,T\right] \times
\mathbb{\Lambda }\rightarrow \mathbb{R}$ is a non-anticipative process, we
write that $u\in \mathcal{C}(\left[ 0,T\right] \times \mathbb{\Lambda })$ if
$u$ is continuous on $\left[ 0,T\right] \times \mathbb{\Lambda }$ under the
pseudometric $d$; if, moreover, $u$ is continuous and bounded on $\left[ 0,T%
\right] \times \mathbb{\Lambda }$, we write $u\in \mathcal{C}_{b}(\left[ 0,T%
\right] \times \mathbb{\Lambda })$. Eventually, following \cite%
{ek-ke-to-zh/14}, we write that $u\in \mathcal{C}_{b}^{1,2}(\left[ 0,T\right]
\times \mathbb{\Lambda })$ if there exists $\hat{u}\in \mathcal{C}_{b}^{1,2}(%
\left[ 0,T\right] \times \mathbb{\hat{\Lambda}})$ such that $\hat{u}\big|_{%
\left[ 0,T\right] \times \mathbb{\Lambda }}=u$ and by definition we take $%
\partial _{t_{i}}u:=\partial _{t}\hat{u}$, $\partial _{x}u:=\partial _{x}%
\hat{u}$, $\partial _{xx}^{2}u:=\partial _{xx}^{2}\hat{u}$; notice that
definitions are independent of the choice of $\hat{u}$.

We introduce now the shifted spaces of c\`{a}dl\`{a}g and continuous paths.
If $t\in \left[ 0,T\right] $, $\hat{B}^{t}$ is the shifted canonical process
on $\mathbb{\hat{\Lambda}}^{t}:=\mathbb{D}\left( \left[ t,T\right] ;\mathbb{R%
}^{d}\right) $, $\mathbb{\hat{F}}^{t}:=(\mathcal{\hat{F}}_{s}^{t})_{s\in %
\left[ t,T\right] }$ is the shifted filtration generated by $\hat{B}^{t}$,%
\begin{equation*}
\begin{array}{l}
||\hat{\phi}||_{T}^{t}:=\sup_{r\in \left[ t,T\right] }|\hat{\phi}\left(
r\right) |,\medskip \\
d^{t}((s,\hat{\phi}),(s^{\prime },\hat{\phi}{^{\prime }})):=|s-s^{\prime
}|+\sup_{r\in \left[ t,T\right] }|\hat{\phi}\left( r\wedge s\right) -\hat{%
\phi}{^{\prime }}\left( r\wedge s^{\prime }\right) |,%
\end{array}%
\end{equation*}%
for any $(s,\hat{\phi}),(s^{\prime },\hat{\phi}{^{\prime }})\in \left[ t,T%
\right] \times \mathbb{\hat{\Lambda}}^{t}$. Analogously we define the spaces
$\mathcal{C}(\left[ t,T\right] \times \mathbb{\hat{\Lambda}}^{t})$, $%
\mathcal{C}_{b}(\left[ t,T\right] \times \mathbb{\hat{\Lambda}}^{t})$ and $%
\mathcal{C}_{b}^{1,2}(\left[ t,T\right] \times \mathbb{\hat{\Lambda}}^{t}).$
Similarly, we denote $\mathbb{\Lambda }^{t}:=\mathcal{C}([t,T];\mathbb{R}%
^{d})$, $B^{t}$ the shifted canonical process on $\mathbb{\Lambda }^{t}$, $%
\mathbb{F}^{t}:=(\mathcal{F}_{s}^{t})_{s\in \left[ t,T\right] }$ the shifted
filtration generated by $B^{t}$ and we introduce the spaces $\mathcal{C}(%
\left[ t,T\right] \times \mathbb{\Lambda }^{t})$, $\mathcal{C}_{b}(\left[ t,T%
\right] \times \mathbb{\Lambda }^{t})$ and $\mathcal{C}_{b}^{1,2}(\left[ t,T%
\right] \times \mathbb{\Lambda }^{t})$.

Let us denote by $\mathcal{T}$ the set of all $\mathbb{F}$--stopping times $%
\tau $ such that for all $t\in \lbrack 0,T)$, the set $\left\{ \phi \in
\mathbb{\Lambda }:\tau \left( \phi \right) >t\right\} $ is an open subset of
$\left( \mathbb{\Lambda },||\cdot ||_{T}\right) $ and $\mathcal{T}^{t}$ the
be the set of all $\mathbb{F}$--stopping times $\tau $ such that for all $%
s\in \lbrack t,T)$, the set $\left\{ \phi \in \mathbb{\Lambda }^{t}:\tau
\left( \phi \right) >s\right\} $ is an open subset of $\left( \mathbb{%
\Lambda }^{t},||\cdot ||_{T}^{t}\right) $.

For a c\`{a}dl\`{a}g function $\phi \in \mathbb{D}\left( \left[ -\delta ,T%
\right] ;\mathbb{R}^{d}\right) $, we denote%
\begin{equation}
\phi _{t}:=\left( \phi (t+\theta )\right) _{\theta \in \lbrack -\delta
,0]}\,.  \label{definition delayed term 2}
\end{equation}%
We conclude this subsection by recalling the functional version of It\^{o}'s
formula (see Cont \& Fourni\'{e} \cite[Theorem 4.1]{co-fo/13}).

\begin{theorem}[Functional It\^{o}'s formula]
\label{functional Ito formula}Let $A$ be a $d$--dimensional It\^{o} process,
i.e. $A:\left[ 0,T\right] \times \mathbb{\Lambda }\rightarrow \mathbb{R}^{d}$
is a continuous $\mathbb{R}^{d}$--valued semimartingale defined on the
probability space $\left( \mathbb{\Lambda },\mathbb{F},\mathbb{P}\right) $
which admits the representation%
\begin{equation*}
A\left( t\right) =A\left( 0\right) +\int_{0}^{t}b\left( r\right)
dr+\int_{0}^{t}\sigma (r)dB\left( r\right) ,\quad \text{for all }t\in \left[
0,T\right] ,\quad \mathbb{P}\text{--a.s.,}
\end{equation*}%
where $b,\sigma $ are progressively measurable stochastic processes such
that $\int_{0}^{T}\big(\left\vert b\left( r\right) \right\vert +\left\vert
\sigma \left( r\right) \right\vert ^{2}\big)dr<\infty ,$ $\mathbb{P}$--a.s.

If $F\in \mathcal{C}_{b}^{1,2}(\left[ 0,T\right] \times \mathbb{\hat{\Lambda}%
})$ then, for any $t\in \lbrack 0,T)$, the following change of variable
formula holds true:%
\begin{equation*}
\begin{array}{l}
\displaystyle F\left( t,A_{\left( t\right) }\right) =F\left( 0,A_{\left(
0\right) }\right) +\int_{0}^{t}\partial _{t}F\left( r,A_{\left( r\right)
}\right) dr+\int_{0}^{t}\left\langle \partial _{x}F\left( r,A_{\left(
r\right) }\right) ,b\left( r\right) \right\rangle dr\medskip \\
\displaystyle\quad \quad \quad \quad +\frac{1}{2}\int_{0}^{t}\mathrm{Tr}\big[%
\sigma \left( r\right) \sigma ^{\ast }\left( r\right) \,\partial
_{xx}^{2}F(r,A_{\left( r\right) })\big]dr+\int_{0}^{t}\left\langle \partial
_{x}F\left( r,A_{\left( r\right) }\right) ,\sigma \left( r\right) dB\left(
r\right) \right\rangle ,\quad \mathbb{P}\text{--a.s..}%
\end{array}%
\end{equation*}
\end{theorem}

\subsection{Path--dependent PDEs}

We give now the notion of viscosity solution to equation \eqref{PDKE} as it
was first introduced in \cite{ek-ke-to-zh/14} (see also \cite%
{ek-to-zh/14,ek-to-zh/14x}).

Let $\left( t,\phi \right) \in \left[ 0,T\right] \times \mathbb{\Lambda }$
be fixed and $\left( W\left( t\right) \right) _{t\geq 0}$ be a $d^{\prime }$%
--dimensional standard Brownian motion defined on some complete probability
space $(\Omega ,\mathcal{G},\mathbb{P})$. We denote by $\mathbb{G}%
^{t}=\left( \mathcal{G}_{s}^{t}\right) _{s\in \left[ 0,T\right] }$ the
natural filtration generated by $\left( (W\left( s\right) -W\left( t\right) )%
\mathbb{1}_{\{s\geq t\}}\right) _{s\in \left[ 0,T\right] }\,,$ augmented by $%
\mathcal{N}$, the set of$\;\mathbb{P}$--null events of $\mathcal{G}.$

Let us take $L\geq 0$ and $t<T$. We denote by $\mathcal{U}_{t}^{L}$ the
space of $\mathbb{G}^{t}$--progressively measurable $\mathbb{R}^{d}$--valued
processes $\lambda $ such that $|\lambda |\leq L$. We define a new
probability measure $\mathbb{P}^{t,\lambda }$ by $d\mathbb{P}^{t,\lambda
}:=M^{t,\lambda }\left( T\right) d\mathbb{P}$, where
\begin{equation*}
M^{t,\lambda }\left( s\right) :=\exp \Big(\displaystyle\int_{t}^{s}\lambda
\left( r\right) dW\left( r\right) -\frac{1}{2}\int_{t}^{s}|\lambda \left(
r\right) |^{2}dr\Big),\quad \mathbb{P}\text{--a.s..}
\end{equation*}

Under some suitable assumptions on the coefficients, to be specified later
on (see Theorem \ref{theorem 0}), there exists and is unique a continuous
and adapted stochastic process $\left( X^{t,\phi }\left( s\right) \right)
_{s\in \left[ 0,T\right] }$ such that, for given $\left( t,\phi \right) \in %
\left[ 0,T\right] \times \mathbb{\Lambda },$%
\begin{equation*}
\left\{
\begin{array}{l}
\displaystyle X^{t,\phi }\left( s\right) =\phi \left( t\right)
+\int_{t}^{s}b(r,X^{t,\phi })dr+\int_{t}^{s}\sigma (r,X^{t,\phi })dW\left(
r\right) ,\quad s\in \left[ t,T\right] ,\medskip \\
\displaystyle X^{t,\phi }\left( s\right) =\phi \left( s\right) ,\quad s\in
\lbrack 0,t).%
\end{array}%
\right.
\end{equation*}%
We are now ready to define the space of the test functions,%
\begin{equation*}
\begin{array}{l}
\displaystyle\underline{\mathcal{A}}^{L}u\left( t,\phi \right) :=\Big\{%
\varphi \in \mathcal{C}_{b}^{1,2}(\left[ 0,T\right] \times \mathbb{\Lambda }%
):\exists \tau _{0}\in \mathcal{T}_{+}^{t}\,,\medskip \\
\displaystyle\quad \quad \quad \quad \quad \quad \quad \quad \quad \quad
\quad \quad \varphi \left( t,\phi \right) -u\left( t,\phi \right)
=\min\nolimits_{\tau \in \mathcal{T}^{t}}\underline{\mathcal{E}}_{t}^{L}\big[%
(\varphi -u)\big(\tau \wedge \tau _{0},X^{t,\phi }\big)\big]\Big\}%
\end{array}%
\end{equation*}%
and%
\begin{equation*}
\begin{array}{l}
\displaystyle\overline{\mathcal{A}}^{L}u\left( t,\phi \right) :=\Big\{%
\varphi \in \mathcal{C}_{b}^{1,2}(\left[ 0,T\right] \times \mathbb{\Lambda }%
):\exists \,\tau _{0}\in \mathcal{T}_{+}^{t}\,,\medskip \\
\displaystyle\quad \quad \quad \quad \quad \quad \quad \quad \quad \quad
\quad \quad \varphi \left( t,\phi \right) -u\left( t,\phi \right)
=\max\nolimits_{\tau \in \mathcal{T}^{t}}\overline{\mathcal{E}}_{t}^{L}\big[%
(\varphi -u)\big(\tau \wedge \tau _{0},X^{t,\phi }\big)\big]\Big\},%
\end{array}%
\end{equation*}%
where $\mathcal{T}_{+}^{t}:=\left\{ \tau \in \mathcal{T}^{t}:\tau >t\right\}
$, if $t<T$ and $\mathcal{T}_{+}^{T}:=\left\{ T\right\} $. Also, for any $%
\xi \in L^{2}\left( \mathcal{F}_{T}^{t};\mathbb{P}\right) $,%
\begin{equation*}
\underline{\mathcal{E}}_{t}^{L}\left( \xi \right) :=\inf\nolimits_{\lambda
\in \mathcal{U}_{t}^{L}}\mathbb{E}^{\mathbb{P}^{t,\lambda }}\left( \xi
\right) \quad \text{and}\quad \overline{\mathcal{E}}_{t}^{L}\left( \xi
\right) :=\sup\nolimits_{\lambda \in \mathcal{U}_{t}^{L}}\mathbb{E}^{\mathbb{%
P}^{t,\lambda }}\left( \xi \right)
\end{equation*}%
are nonlinear expectations.$\medskip $

We are now able to give the definition of a viscosity solution of the
functional PDE (\ref{PDKE}), see, e.g. \cite[Definition 3.3]{ek-ke-to-zh/14}.

\begin{definition}
\label{definition 1}Let $u\in \mathcal{C}_{b}(\left[ 0,T\right] \times
\mathbb{\Lambda })$ be such that $u(T,\phi )=h\left( \phi \right) $, for all
$\phi \in \mathbb{\Lambda }$.

\noindent $\left( a\right) $ For any $L\geq 0$, we say that $u$ is a
viscosity $L$--subsolution of (\ref{PDKE}) if at any point $\left( t,\phi
\right) \in \left[ 0,T\right] \times \mathbb{\Lambda }$, for any $\varphi
\in \underline{\mathcal{A}}^{L}u\left( t,\phi \right) $, it holds
\begin{equation*}
-\partial _{t}\varphi (t,\phi )-\mathcal{L}\varphi (t,\phi )-f(t,\phi
,u(t,\phi ),\partial _{x}\varphi \left( t,\phi \right) \sigma (t,\phi
),\left( u(\cdot ,\phi )\right) _{t})\leq 0.
\end{equation*}%
\noindent $\left( b\right) $ For any $L\geq 0$, we say that $u$ is a
viscosity $L$--supersolution of (\ref{PDKE}) if at any point $\left( t,\phi
\right) \in \left[ 0,T\right] \times \mathbb{\Lambda }$, for any $\varphi
\in \overline{\mathcal{A}}^{L}u\left( t,\phi \right) $, we have%
\begin{equation*}
-\partial _{t}\varphi (t,\phi )-\mathcal{L}\varphi (t,\phi )-f(t,\phi
,u(t,\phi ),\partial _{x}\varphi \left( t,\phi \right) \sigma (t,\phi
),\left( u(\cdot ,\phi )\right) _{t})\geq 0.
\end{equation*}%
\noindent $\left( c\right) $ We say that $u$ is a viscosity subsolution
(respectively, supersolution) of (\ref{PDKE}) if $u$ is a viscosity $L$%
--subsolution (respectively, $L$--supersolution) of (\ref{PDKE}) for some $%
L\geq 0.\medskip $

\noindent$\left( d\right) $ We say that $u$ is a viscosity solution of (\ref%
{PDKE}) if $u$ is a viscosity subsolution and supersolution of (\ref{PDKE}).
\end{definition}

\begin{remark}
It is easy to see that this definition is equivalent to the classical one in
the Markovian framework, see, e.g. \cite{ek-ke-to-zh/14}.
\end{remark}

Let us stress that if $u$ is a function from $\mathcal{C}_{b}^{1,2}(\left[
0,T\right] \times \mathbb{\Lambda })$, then it is easily seen that $u$ is a
viscosity solution of (\ref{PDKE}) if and only if $u$ is a classical
solution for (\ref{PDKE}). Indeed, if $u$ is a viscosity solution then $u\in
\underline{\mathcal{A}}^{L}u\left( t,\phi \right) \cap \overline{\mathcal{A}}%
^{L}u\left( t,\phi \right) $ and therefore $u$ satisfies (\ref{PDKE}). For
the reverse statement one can use the nonlinear Feynman--Kac formula proved
in this new framework (see Theorem \ref{Feynman-Kac formula} below),
together with functional It\^{o}'s formula in order to compute $u\left(
s,X^{t,\phi }\right) $ and the existence and uniqueness result for the
corresponding stochastic system (\ref{FBSDE}).

Let us also mention that, in accordance with the standard theory of
viscosity solutions, the viscosity property introduced above is a local
property, i.e. to check that $u$ is a viscosity solution in $\left( t,\phi
\right) $ it is sufficient to know the value of $u$ on the interval $\left[
t,\tau _{\epsilon }\right] $, where $\epsilon >0$ is arbitrarily fixed and $%
\tau _{\epsilon }\in \mathcal{T}_{+}^{t}$ is given by $\tau _{\epsilon
}:=\inf \left\{ s>t:\left\vert \phi \left( s\right) \right\vert \geq
\epsilon \right\} \wedge \left( t+\epsilon \right) $.

Eventually, let us notice that since $b$ and $\sigma$ are Lipschitz, we have
uniqueness in law for $X^{t,\phi}$; also, since the filtration on $(\Omega,%
\mathcal{G},\mathbb{P}) $ is generated by $W$, every progressively
measurable processes $\lambda$ is a functional of $W$. Therefore, the spaces
of test functions and the above definition are independent on the choice of $%
(\Omega,\mathcal{G},\mathbb{P})$ and $W$.

\section{The forward--backward delayed system}

\label{Section 2}

We are now able to state the existence and uniqueness results for a delayed
forward-backward system, where both the forward and the backward component
exhibit a delayed behaviour, that is we will assume that the generator of
the backward equation may depend on past values assumed by its solution $%
(Y,Z)$. In complete generality, since we will need these results in next
sections, we will allow the solution to depend on a general initial time and
initial values. Also we remark that in order to ensure the existence and
uniqueness result, we need to equip the backward equation with a suitable
condition in the time interval $[0,t)$, $t$ being the initial time, fact
that implies a different proof than the one provided in \cite{de-im/10}.

The main goal is to find a family $\left( X^{t,\phi },Y^{t,\phi },Z^{t,\phi
}\right) _{(t,\phi )\in \lbrack 0,T]\times \mathbb{\Lambda }}$ of stochastic
processes such that the following decoupled forward--backward system holds $%
\mathbb{P}$--a.s.%
\begin{equation}
\left\{
\begin{array}{l}
\displaystyle X^{t,\phi }\left( s\right) =\phi \left( t\right)
+\int_{t}^{s}b(r,X^{t,\phi })dr+\int_{t}^{s}\sigma (r,X^{t,\phi })dW\left(
r\right) ,\;s\in \left[ t,T\right] ,\medskip \\
\displaystyle X^{t,\phi }\left( s\right) =\phi \left( s\right) ,\quad s\in
\lbrack 0,t),\medskip \\
\multicolumn{1}{r}{\displaystyle Y^{t,\phi }\left( s\right) =h(X^{t,\phi
})+\int_{s}^{T}F(r,X^{t,\phi },Y^{t,\phi }\left( r\right) ,Z^{t,\phi }\left(
r\right) ,Y_{r}^{t,\phi },Z_{r}^{t,\phi })dr\quad \medskip} \\
\multicolumn{1}{r}{\displaystyle-\int_{s}^{T}Z^{t,\phi }\left( r\right)
dW\left( r\right) ,\quad s\in \left[ t,T\right] ,\medskip} \\
\displaystyle Y^{t,\phi }\left( s\right) =Y^{s,\phi }\left( s\right) ,\quad
Z^{t,\phi }\left( s\right) =0,\quad s\in \lbrack 0,t).%
\end{array}%
\right.  \label{FBSDE 2}
\end{equation}

Let us stress once more that in both the forward and backward equation, the
values of $X^{t,\phi }$ and $\left( Y^{t,\phi },Z^{t,\phi }\right) $ need to
be known in the time interval $[0,t]$ and respectively $\left[ t-\delta ,t%
\right] ;$ this is one reason for which we have to impose such initial
conditions. The above initial condition for $Y$ is absolutely necessary in
view of the Feynman--Kac formula, which will be proven later. We also
prolong, by convention, $Y^{t,\phi }$ by $Y^{t,\phi }(0)$ on the negative
real axis (this is needed in the case that $t<\delta $). For the sake of
simplicity, we will take $Z^{t,\phi }\left( s\right) :=0\ $and $F\left(
s,\cdot ,\cdot ,\cdot ,\cdot ,\cdot \right) :=0$ whenever $s$ becomes
negative.

\subsection{The forward path-dependent SDE}

Let us first focus on the forward component $X$ appearing in system %
\eqref{FBSDE 2}; the next theorem states the existence and the uniqueness,
as well as estimates, for the process $\left( X^{t,\phi }\left( r\right)
\right) _{r\in \left[ 0,T\right] }$.

The existence result is a classical one (see, e.g. \cite{mo/96} or \cite%
{mo/84}) and the estimates can be obtained by applying It\^{o}'s formula
together with assumptions $\mathrm{(A}_{1}\mathrm{)}$--$\mathrm{(A}_{2}%
\mathrm{)}$, see, e.g. \cite{za/12}, and for these reasons we will not state
the proof.

In what follows we will assume the following to hold.

Let us consider two non-anticipative functionals $b:\left[ 0,T\right] \times%
\mathbb{\Lambda}\rightarrow\mathbb{R}^{d}$ and $\sigma:\left[ 0,T\right]
\times\mathbb{\Lambda}\rightarrow\mathbb{R}^{d\times d^{\prime}}$ such that

\begin{description}
\item[$\mathrm{(A}_{1}\mathrm{)}$] $b$ and $\sigma$ are continuous;

\item[$\mathrm{(A}_{2}\mathrm{)}$] there exists $\ell>0$ such that for any $%
t\in\left[ 0,T\right] $, $\phi,\phi^{\prime}\in\mathbb{\Lambda},$%
\begin{equation*}
|b(t,\phi)-b(t,\phi^{\prime})|+|\sigma(t,\phi)-\sigma(t,\phi^{\prime})|\leq%
\ell||\phi-\phi^{\prime}||_{T}\,.
\end{equation*}
\end{description}

\begin{theorem}
\label{theorem 0}Let $b,\sigma$ satisfying assumptions $\mathrm{(A}_{1}%
\mathrm{)}$--$\mathrm{(A}_{2}\mathrm{)}$. Let $\left( t,\phi\right)
,(t^{\prime},\phi^{\prime})\in\left[ 0,T\right] \times\mathbb{\Lambda}$ be
given. Then there exists a unique continuous and adapted stochastic process $%
\left( X^{t,\phi}\left( s\right) \right) _{s\in\left[ 0,T\right] }$ such that%
\begin{equation}
\left\{
\begin{array}{l}
\displaystyle X^{t,\phi}\left( s\right) =\phi\left( t\right)
+\int_{t}^{s}b(r,X^{t,\phi})dr+\int_{t}^{s}\sigma(r,X^{t,\phi})dW\left(
r\right) ,\;s\in\left[ t,T\right] ,\medskip \\
\displaystyle X^{t,\phi}\left( s\right) =\phi\left( s\right)
,\;s\in\lbrack0,t).%
\end{array}
\right.  \label{FSDE}
\end{equation}

Moreover, for any $q\geq 1$, there exists $C=C\left( q,T,\ell \right) >0$
such that%
\begin{equation*}
\begin{array}{l}
\displaystyle\mathbb{E}\big(||X^{t,\phi }||_{T}^{2q}\big)\leq C(1+||\phi
||_{T}^{2q}),\medskip \\
\multicolumn{1}{r}{\displaystyle\mathbb{E}\big(||X^{t,\phi }-X^{t^{\prime
},\phi ^{\prime }}||_{T}^{2q}\big)\leq C\Big(||\phi -\phi ^{\prime
}||_{T}^{2q}+(1+||\phi ||_{T}^{2q}+||\phi ^{\prime }||_{T}^{2q})\cdot
|t-t^{\prime }|^{q}\medskip} \\
\multicolumn{1}{r}{\displaystyle+\sup\nolimits_{r\in \left[ t\wedge
t^{\prime },t\vee t^{\prime }\right] }|\phi \left( t\right) -\phi \left(
r\right) |^{2q}\Big),\medskip} \\
\displaystyle\mathbb{E}\big(\sup\nolimits_{\substack{ s,r\in \left[ t,T%
\right]  \\ \left\vert s-r\right\vert \leq \epsilon }}|X^{t,\phi }\left(
s\right) -X^{t,\phi }\left( r\right) |^{2q}\big)\leq C(1+||\phi
||_{T}^{2q})\epsilon ^{q-1},\quad \text{for all }\epsilon >0.%
\end{array}%
\end{equation*}
\end{theorem}

\subsection{The backward delayed SDE}

Let us now consider the delayed backward SDE appearing in \eqref{FBSDE 2}.
In what follows, $d$ and $d^{\prime }$ are previously fixed constants,
whereas $m\in \mathbb{N}^{\ast }$ is a new fixed constant. Let us then
introduce the main reference spaces we will consider.

\begin{definition}
\begin{description}
\item

\item[(i)] let $\mathcal{H}_{t}^{2,m\times d^{\prime}}$ denote the space of
(equivalence classes of) $\mathbb{G}^{t}$--progressively measurable
processes $Z:\Omega\times\left[ 0,T\right] \rightarrow\mathbb{R}^{m\times
d^{\prime}}$ such that $\mathbb{E}\left[ \displaystyle%
\int_{0}^{T}|Z(s)|^{2}ds\right] <\infty\,;$

\item[(ii)] let $\mathcal{S}_{t}^{2,m}$ the space of (equivalence classes
of) continuous $\mathbb{G}^{t}$--progressively measurable processes $%
Y:\Omega \times\left[ 0,T\right] \rightarrow\mathbb{R}^{m}$ such that $%
\mathbb{E}\left[ \displaystyle\sup_{0\leq s\leq T}|Y(s)|^{2}\right] <\infty
\,.\medskip$
\end{description}

Also we will equip the spaces $\mathcal{H}_{t}^{2,m\times d^{\prime}}$ and $%
\mathcal{S}_{t}^{2,m}$ with the following norms
\begin{equation*}
\Vert Z\Vert_{\mathcal{H}_{t}^{2,m\times d^{\prime}}}^{2}=\mathbb{E}\left[
\int_{0}^{T}e^{\beta s}|Z(s)|^{2}ds\right] \,,\quad\quad\Vert Y\Vert _{%
\mathcal{S}_{t}^{2,m}}^{2}=\mathbb{E}\left[ \sup_{0\leq s\leq T}e^{\beta
s}|Y(s)|^{2}\right] \,,
\end{equation*}
for a given constant $\beta>0$.
\end{definition}

Concerning the delayed backward SDE in \eqref{FBSDE 2}, we will assume the
following to hold.

Let $F:[0,T]\times\mathbb{\Lambda}\times\mathbb{R}^{m}\times\mathbb{R}%
^{m\times d^{\prime}}\times L^{2}\left( [-\delta,0];\mathbb{R}^{m}\right)
\times L^{2}([-\delta,0];\mathbb{R}^{m\times d^{\prime}})\rightarrow \mathbb{%
R}^{m}$ and $h:\mathbb{\Lambda}\rightarrow\mathbb{R}^{m}$ such that the
following holds:

\begin{description}
\item[$\mathrm{(A}_{3}\mathrm{)}$] There exist $L,K,M>0$, $p\geq 1$ and a
probability measure $\alpha $ on $\left( [-\delta ,0],\mathcal{B}\left( %
\left[ -\delta ,0\right] \right) \right) $ such that, for any $t\in \lbrack
0,T]$, $\phi \in \mathbb{\Lambda }$, $\left( y,z\right) ,(y^{\prime
},z^{\prime })\in \mathbb{R}^{m}\times \mathbb{R}^{m\times d^{\prime }}$, $%
\hat{y},\hat{y}^{\prime }\in L^{2}\left( [-\delta ,0];\mathbb{R}^{m}\right) $
and $\hat{z},\hat{z}^{\prime }\in L^{2}([-\delta ,0];\mathbb{R}^{m\times
d^{\prime }})$, we have%
\begin{equation*}
\begin{array}{rl}
\left( i\right) & \displaystyle\phi \mapsto F(t,\phi ,y,z,\hat{y},\hat{z})%
\text{ is continuous,}\medskip \\
\left( ii\right) & \displaystyle|F(t,\phi ,y,z,\hat{y},\hat{z})-F(t,\phi
,y^{\prime },z^{\prime },\hat{y},\hat{z})|\leq L(|y-y^{\prime
}|+|z-z^{\prime }|),\medskip \\
\multicolumn{1}{l}{\left( iii\right)} & \displaystyle|F(t,\phi ,y,z,\hat{y},%
\hat{z})-F(t,\phi ,y,z,\hat{y}^{\prime },\hat{z}^{\prime })|^{2}\medskip \\
\multicolumn{1}{l}{} & \displaystyle\quad \leq K\int_{-\delta }^{0}\left(
\left\vert \hat{y}(\theta )-\hat{y}^{\prime }(\theta )\right\vert
^{2}+\left\vert \hat{z}(\theta )-\hat{z}^{\prime }(\theta )\right\vert
^{2}\right) \alpha (d\theta ),\medskip \\
\left( iv\right) & \displaystyle\left\vert F\left( t,\phi ,0,0,0,0\right)
\right\vert <M(1+\left\Vert \phi \right\Vert _{T}^{p}).%
\end{array}%
\end{equation*}

\item[$\mathrm{(A}_{4}\mathrm{)}$] The function $F\left( \cdot ,\cdot ,y,z,%
\hat{y},\hat{z}\right) $ is $\mathbb{F}$--progressively measurable, for any $%
\left( y,z,\hat{y},\hat{z}\right) \in \mathbb{R}^{m}\times \mathbb{R}%
^{m\times d^{\prime }}\times L^{2}\left( [-\delta ,0];\mathbb{R}^{m}\right)
\times L^{2}([-\delta ,0];\mathbb{R}^{m\times d^{\prime }})$.

\item[$\mathrm{(A}_{5}\mathrm{)}$] The function $h$ is continuous and $%
|h(\phi )|\leq M(1+\left\Vert \phi \right\Vert _{T}^{p}),$ for all $\phi \in
\mathbb{\Lambda }.$
\end{description}

\begin{remark}
\label{constraint}In order to show the existence and uniqueness of a
solution to the backward part of system (\ref{FBSDE 2}) and to obtain the
continuity of $Y^{t,\phi }$ with respect to $\phi $ we are forced to impose
that $K$ or $\delta $ are small enough.

More precisely, we will assume that there exists a constant $\gamma \in
(0,1) $ such that%
\begin{equation}
K\,\frac{\gamma e^{\big(\gamma +\frac{6L^{2}}{\gamma }\big)\delta }}{%
(1-\gamma )L^{2}}\,\max \left\{ 1,T\right\} <\frac{1}{290}\,.
\label{condition_KT}
\end{equation}
\end{remark}

We are now ready to state the first result of the this section.

\begin{theorem}
\label{theorem 3}

Let us assume that assumptions $\mathrm{(A}_{3}\mathrm{)}$--$\mathrm{(A}_{5}%
\mathrm{)}$ hold true. If condition (\ref{condition_KT}) is satisfied, then
there exists a unique solution $\left( Y^{t,\phi },Z^{t,\phi }\right)
_{(t,\phi )\in \lbrack 0,T]\times \mathbb{\Lambda }}$ for the backward
stochastic differential system from (\ref{FBSDE 2}) such that $\left(
Y^{t,\phi },Z^{t,\phi }\right) \in \mathcal{S}_{t}^{2,m}\times \mathcal{H}%
_{t}^{2,m\times d^{\prime }},$ for all $t\in \lbrack 0,T]$ and the
application $t\mapsto \left( Y^{t,\phi },Z^{t,\phi }\right) $ is continuous
from $[0,T]$ into $\mathcal{S}_{0}^{2,m}\times \mathcal{H}_{0}^{2,m\times
d^{\prime }}$.
\end{theorem}

\begin{remark}
The previous theorem provides only the regularity of $t\mapsto \left(
Y^{t,\phi },Z^{t,\phi }\right) .$ In what concerns the continuity of $\left(
Y^{t,\phi },Z^{t,\phi }\right) $ with respect to $\phi $, see the proof of
Theorem \ref{theorem 2}.
\end{remark}

\begin{remark}
Theorem \ref{theorem 3} states that if $K$ or $\delta $ are small enough
such that condition (\ref{condition_KT}) is satisfied, then there exists a
unique solution.

More precisely, if $\delta $ is fixed, then we can find $K$ sufficiently
small such that restriction (\ref{condition_KT}) is satisfied.

On the other hand, if $K$ is fixed it seems, at first sight, that
restriction (\ref{condition_KT}) cannot be made true by letting $\delta $ to
$0$; but in this case it is sufficient to take $\tilde{L}\geq L$ large
enough for which there exists a small enough $\delta >0$ such that condition
(\ref{condition_KT}) with $L$ replaced by $\tilde{L}$ is still satisfied. Of
course, assumption $\mathrm{(A}_{3}-ii\mathrm{)}$ is still verified with $%
\tilde{L}$ instead $L\,.$
\end{remark}

\begin{remark}
The proof is mainly based on Banach fixed point theorem; we provide it in
the Appendix section.$\smallskip $

The main difference between our result and Theorem 2.1 from \cite{de-im/10}
is due to the supplementary structure condition $Y^{t,\phi}\left( s\right)
=Y^{s,\phi}\left( s\right) $, for $s\in\lbrack0,t)$ which should be
satisfied by the unknown process $Y^{t,\phi}$.

We also allow $T$\ to be arbitrary and we consider that the time horizon is
different from the delay $\delta \in \left[ 0,T\right] $; moreover, we make
the difference between the Lipschitz constant $L$\ with respect to $\left(
y,z\right) $\ and the Lipschitz constant $K$\ with respect to $\hat{y}$ and $%
\hat{z}$, and hence in restriction (\ref{condition_KT}) we can play also
with the constant $K$.
\end{remark}

\begin{remark}
Using It\^{o}'s formula and proceeding as in the proof of Theorem \ref%
{theorem 3}, we can easily show that the solution $\left( Y^{t,\phi
},Z^{t,\phi}\right) $ to equation (\ref{FBSDE 2}) satisfies the following
inequality. For any $q\geq1$, there exists $C>0$ such that%
\begin{equation}
\begin{array}{l}
\displaystyle\mathbb{E}\big(\sup_{r\in\left[ 0,T\right] }|Y^{t,\phi}\left(
r\right) |^{2q}\big)+\mathbb{E}\Big(\int_{0}^{T}|Z^{t,\phi}\left( r\right)
|^{2}dr\Big)^{q}\medskip \\
\displaystyle\quad\leq C\Big[\mathbb{E}|h(X^{t,\phi})|^{2q}+\mathbb{E}\Big(%
\int_{0}^{T}|F(r,X^{t,\phi},0,0,0,0)|dr\Big)^{2q}\Big]\leq C(1+||\phi
||_{T}^{2q}).%
\end{array}
\label{a priori estimates}
\end{equation}
\end{remark}

\section{Path--dependent PDE -- proof of the existence theorem}

\label{Section 3}

The current section is devoted to the study of viscosity solution to the
path-dependent equation \eqref{PDKE}. In particular, in order to obtain
existence of a viscosity solution, we will impose some additional
assumptions on the generator $f$ and on the terminal condition $h$ in
equation \eqref{FBSDE 2}, in particular we will assume that the generator $f$
depends only on past values assumed by $Y$ and not by past values of $Z$. In
what follows we will assume the following to hold.

Let $f:\left[ 0,T\right] \times \mathbb{\Lambda }\times \mathbb{R}\times
\mathbb{R}^{d}\times \mathcal{C}\left( \left[ -\delta ,0\right] ;\mathbb{R}%
\right) \rightarrow \mathbb{R}$ and $h:\mathbb{\Lambda }\rightarrow \mathbb{R%
}$ such that the following hold.

\begin{description}
\item[$\mathrm{(A}_{6}\mathrm{)}$] the functions $f$ and $h$ are continuous;
also $f\left( \cdot,\cdot,y,z,\hat{y}\right) $ is non-anticipative;

\item[$\mathrm{(A}_{7}\mathrm{)}$] there exist $L,K,M>0$ and $p\geq 1$ such
that for any $\left( t,\phi \right) \in \left[ 0,T\right] \times \mathbb{%
\Lambda }$, $y,y^{\prime }\in \mathbb{R}$, $z,z^{\prime }\in \mathbb{R}^{d}$
and $\hat{y},\hat{y}^{\prime }\in \mathcal{C}\left( \left[ -\delta ,0\right]
;\mathbb{R}\right) :$%
\begin{equation*}
\begin{array}{rl}
\left( i\right) & \displaystyle|f(t,\phi ,y,z,\hat{y})-f(t,\phi ,y^{\prime
},z^{\prime },\hat{y})|\leq L(|y-y^{\prime }|+|z-z^{\prime }|),\medskip \\
\left( ii\right) & \displaystyle|f(t,\phi ,y,z,\hat{y})-f(t,\phi ,y,z,\hat{y}%
^{\prime })|^{2}\leq K\int_{-\delta }^{0}|\hat{y}\left( \theta \right) -\hat{%
y}^{\prime }\left( \theta \right) |^{2}\alpha \left( d\theta \right)
,\medskip \\
\left( iii\right) & \displaystyle|f(t,\phi ,0,0,0)|\leq M(1+\left\Vert \phi
\right\Vert _{T}^{p}),\medskip \\
\left( iv\right) & \displaystyle\left\vert h\left( \phi \right) \right\vert
\leq M(1+\left\Vert \phi \right\Vert _{T}^{p}),%
\end{array}%
\end{equation*}%
with $\alpha $ a probability measure on $\left( \left[ -\delta ,0\right] ,%
\mathcal{B}\left( \left[ -\delta ,0\right] \right) \right) $.
\end{description}

\begin{remark}
\label{REM:DelGen}For example, the following generators satisfy assumptions $%
\mathrm{(A}_{6}\mathrm{)}$,$\mathrm{(A}_{7}\mathrm{)}$:%
\begin{equation*}
f_{1}\left( t,\phi ,y,z,\hat{y}\right) :=K\int_{-\delta }^{0}\hat{y}\left(
s\right) ds,\quad \quad f_{2}\left( t,\phi ,y,z,\hat{y}\right) :=K\,\hat{y}%
\left( t-\delta \right) \,.
\end{equation*}%
In general, if $g:\left[ 0,T\right] \rightarrow \mathbb{R}$ is a measurable,
bounded function with $g\left( t\right) =0$ for $t<0$, then the following
linear time--delayed generator%
\begin{equation*}
f\left( t,\phi ,y,z,\hat{y}\right) =\int_{-\delta }^{0}g\left( t+\theta
\right) \hat{y}\left( \theta \right) \alpha \left( d\theta \right) ,
\end{equation*}%
satisfies assumptions $\mathrm{(A}_{6}\mathrm{)}$,$\mathrm{(A}_{7}\mathrm{)}$%
.
\end{remark}

We state now the main result of the present paper.

\begin{theorem}[Existence]
\label{theorem 1} Let us assume that assumptions $\mathrm{(A}_{1}\mathrm{)}$,%
$\mathrm{(A}_{2}\mathrm{)}$,$\mathrm{(A}_{6}\mathrm{)}$,$\mathrm{(A}_{7}%
\mathrm{)}$ hold. If the delay $\delta $ or the Lipschitz constant $K$ are
sufficiently small, i.e. condition (\ref{condition_KT}) is verified, then
the path--dependent PDE (\ref{PDKE}) admits at least one viscosity solution.
\end{theorem}

\begin{remark}
The new and the essential part of the proof of Theorem \ref{theorem 1} is
the nonlinear representation Feynman--Kac type formula, which links the
functional SDE (\ref{FSDE}) to a suitable BSDE with time--delayed
generators. We prove Feynman--Kac type formula first in the case of BSDE
independent of the past of $Y,$ see Theorem \ref{Feynman-Kac
formula_preliminary}, and after that in the general case of BSDE (\ref%
{BSDE_delayed 1}), see Theorem \ref{Feynman-Kac formula}.$\smallskip $

We also mention that this theorem is a generalization of the Theorem 4.3
from \cite{ek-ke-to-zh/14}.$\medskip $
\end{remark}

\begin{remark}[Uniqueness]
The problem of uniqueness is solved\footnote[1]{%
We are thankful to the Reviewer for his/her solution in what concerns the
uniqueness problem.} by reducing it to the problem of uniqueness of the
viscosity solution from the case of $f$ being independent of the term $%
\left( u(\cdot ,\phi )\right) _{t}\,.$ This particular case is presented in
Theorem 4.6 from \cite{ek-ke-to-zh/14}.

Let us take two viscosity solutions $u^{1}$ and $u^{2}$ of the
path--dependent PDE (\ref{PDKE}). We define%
\begin{equation*}
f^{i}(t,\phi ,y,z):=f(t,\phi ,y,z,\left( u^{i}\left( \cdot ,\phi \right)
\right) _{t})\,,\quad i=\overline{1,2}\,.
\end{equation*}%
Using these drivers we associate the following BSDEs:%
\begin{equation}
Y^{t,\phi }\left( s\right) =h(X^{t,\phi })+\int_{s}^{T}f^{i}(r,X^{t,\phi
},Y^{t,\phi }\left( r\right) ,Z^{t,\phi }\left( r\right)
)dr-\int_{s}^{T}Z^{t,\phi }\left( r\right) dW\left( r\right) \,,\quad i=%
\overline{1,2}\,,  \label{BSDE_delayed 3}
\end{equation}%
for which there exist unique solutions $\left( Y^{i,t,\phi },Z^{i,t,\phi
}\right) \in \mathcal{S}_{t}^{2,1}\times \mathcal{H}_{t}^{2,1\times
d^{\prime }}\,,\quad i=\overline{1,2}\,.$

Using Theorem \ref{Feynman-Kac formula_preliminary} we see that%
\begin{equation*}
Y^{i,t,\phi }\left( s\right) =v^{i}(s,X^{t,\phi }),\quad \text{for all }s\in %
\left[ 0,T\right] ,\quad \text{a.s.,}
\end{equation*}%
for any $\left( t,\phi \right) \in \left[ 0,T\right] \times \mathbb{\Lambda }%
,$ where $v^{i}:\left[ 0,T\right] \times \mathbb{\Lambda }\rightarrow
\mathbb{R}$ are defined by (\ref{def_u}), i.e.%
\begin{equation*}
v^{i}(t,\phi ):=Y^{i,t,\phi }\left( t\right) ,\;(t,\phi )\in \lbrack
0,T]\times \mathbb{\Lambda }\,.
\end{equation*}%
Hence, using Theorem \ref{theorem 1}, we obtain that the functions $v^{i}$
are solutions of the PDE of type (\ref{PDKE}), but without the delayed terms
$\left( v^{i}\left( \cdot ,\phi \right) \right) _{t}:$%
\begin{equation}
\left\{
\begin{array}{r}
-\partial _{t}v^{i}(t,\phi )-\mathcal{L}v^{i}(t,\phi )-f^{i}(t,\phi
,v^{i}(t,\phi ),\partial _{x}v^{i}\left( t,\phi \right) \sigma (t,\phi
))=0,\medskip \\
\multicolumn{1}{l}{v^{i}(T,\phi )=h(\phi ),\quad i=\overline{1,2}\,.}%
\end{array}%
\right.  \label{PDKE 3}
\end{equation}%
Since $u^{i}$ is also solution to equation (\ref{PDKE 3}), by using Theorem
4.6 from \cite{ek-ke-to-zh/14} we obtain%
\begin{equation*}
u^{i}(t,\phi )=v^{i}(t,\phi ),\quad \left( t,\phi \right) \in \left[ 0,T%
\right] \times \mathbb{\Lambda }\,,\quad i=\overline{1,2}\,.
\end{equation*}%
We mention that Theorem 4.6 from \cite{ek-ke-to-zh/14} is proved in the
particular case $b=0,\sigma =I$, but we can adapt the proof to our case
without difficulty. However, some stronger assumptions on $f$ and $h$
(boundedness and uniform continuity w.r.t. $\phi $) are required and we will
have to adopt them in order to use this result.

In this case, since
\begin{equation*}
Y^{i,t,\phi }\left( s\right) =v^{i}(s,X^{t,\phi })=u^{i}(s,X^{t,\phi }),\ i=%
\overline{1,2}\,,
\end{equation*}
BSDEs (\ref{BSDE_delayed 3}) become a single equation,%
\begin{equation}
Y^{i,t,\phi }\left( s\right) =h(X^{t,\phi })+\int_{s}^{T}f(r,X^{t,\phi
},Y^{i,t,\phi }\left( r\right) ,Z^{i,t,\phi }\left( r\right)
,Y_{r}^{i,t,\phi })dr-\int_{s}^{T}Z^{i,t,\phi }\left( r\right) dW\left(
r\right) \,,  \label{BSDE_delayed 4}
\end{equation}%
with $i=\overline{1,2}\,,$ for which we have uniqueness (see Theorem \ref%
{theorem 3}).

Hence $Y^{1,t,\phi }=Y^{2,t,\phi }$ and therefore%
\begin{equation*}
u^{1}(t,\phi )=Y^{1,t,\phi }\left( t\right) =Y^{2,t,\phi }\left( t\right)
=u^{2}(t,\phi ).
\end{equation*}
\end{remark}

Under assumptions $\mathrm{(A}_{1}\mathrm{)}$,$\mathrm{(A}_{2}\mathrm{)}$,$%
\mathrm{(A}_{6}\mathrm{)}$,$\mathrm{(A}_{7}\mathrm{)}$, it follows from
Theorem \ref{theorem 3} (in the case $m=1$) that for each $\left( t,\phi
\right) \in \left[ 0,T\right] \times \mathbb{\Lambda }$ there exists a
unique triple $\left( X^{t,\phi },Y^{t,\phi },Z^{t,\phi }\right) $ of $%
\mathbb{G}^{t}$--progressively measurable processes such that $X^{t,\phi }$
satisfies equation (\ref{FSDE}) and $\left( Y^{t,\phi },Z^{t,\phi }\right)
\in \mathcal{S}_{t}^{2,1}\times \mathcal{H}_{t}^{2,1\times d^{\prime }}$,
with $Y^{t,\phi }\left( s\right) =Y^{s,\phi }\left( s\right) $, for any $%
s\in \lbrack 0,t)$, is a solution of the following BSDE on $\left[ t,T\right]
:$%
\begin{equation}
Y^{t,\phi }\left( s\right) =h(X^{t,\phi })+\int_{s}^{T}f(r,X^{t,\phi
},Y^{t,\phi }\left( r\right) ,Z^{t,\phi }\left( r\right) ,Y_{r}^{t,\phi
})dr-\int_{s}^{T}Z^{t,\phi }\left( r\right) dW\left( r\right) .
\label{BSDE_delayed 1}
\end{equation}%
Let us further observe that the generator $f$ depends on $\omega $ only via
the the process $X^{t,x}$.

Before proving Theorem \ref{theorem 1}, we need to show some results. For
that, let us first define the function $u:[0,T]\times \mathbb{\Lambda }%
\rightarrow \mathbb{R}$ by%
\begin{equation}
u(t,\phi ):=Y^{t,\phi }\left( t\right) ,\;(t,\phi )\in \lbrack 0,T]\times
\mathbb{\Lambda }\,;  \label{def_u}
\end{equation}%
notice that $u(t,\phi )$ is a deterministic function since $Y^{t,\phi
}\left( t\right) $ is $\mathcal{G}_{t}^{t}\equiv \mathcal{N}$--measurable.

\begin{theorem}
\label{theorem 2} Under the assumptions of Theorem \ref{theorem 1}, the
function $u$ is continuous.
\end{theorem}

\begin{remark}
As in the case of Theorem \ref{theorem 3}, the proof can be found in the
Appendix section.
\end{remark}

In order to prove the generalized Feynman--Kac formula suitable to our
framework we first consider the particular case when the generator $f$ is
independent of the past values of $Y$ and $Z$, namely $\left( Y^{t,\phi
},Z^{t,\phi }\right) $ is the solution of the standard BSDE with Lipschitz
coefficients%
\begin{equation}
Y^{t,\phi }\left( s\right) =h(X^{t,\phi })+\int_{s}^{T}f(r,X^{t,\phi
},Y^{t,\phi }\left( r\right) ,Z^{t,\phi }\left( r\right)
)dr-\int_{s}^{T}Z^{t,\phi }\left( r\right) dW\left( r\right) ,\;s\in \left[
t,T\right] .  \label{BSDE_delayed 2}
\end{equation}

\begin{theorem}
\label{Feynman-Kac formula_preliminary} Let us assume that $\mathrm{(A}_{1}%
\mathrm{)}$,$\mathrm{(A}_{2}\mathrm{)}$,$\mathrm{(A}_{6}\mathrm{)}$,$\mathrm{%
(A}_{7}\mathrm{)}$ hold. Then%
\begin{equation*}
Y^{t,\phi }\left( s\right) =u(s,X^{t,\phi }),\quad \text{for all }s\in \left[
0,T\right] ,\quad \text{a.s.,}
\end{equation*}%
for any $\left( t,\phi \right) \in \left[ 0,T\right] \times \mathbb{\Lambda }%
,$ where $Y^{t,\phi }$ is the solution of BSDE (\ref{BSDE_delayed 2}) and $u:%
\left[ 0,T\right] \times \mathbb{\Lambda }\rightarrow \mathbb{R}$ is defined
by (\ref{def_u}).
\end{theorem}

\begin{remark}
We mention that this kind of result was previously proven only in a
particular case by Peng, Wang \cite{pe-wa/11} who considered the case of
smooth coefficients $h$ and $f$ (they also took $b=0,\sigma =I).$
\end{remark}

\begin{remark}
We also mention that Theorem 4.3 from \cite{ek-ke-to-zh/14} uses the
essential Feynman--Kac formula%
\begin{equation*}
Y_{s}^{0}=u\left( s,B^{t}\right)
\end{equation*}%
but without proving it; a similar situation can be found in \cite%
{ek-to-zh/14}. Our result fills this gap.
\end{remark}

\begin{proof}
Again, for the sake of readability, we split the proof into several steps.

\noindent\textbf{\textit{Step I.}}

Let $0=t_{0}<t_{1}<\dots <t_{n}=T$ and suppose that%
\begin{equation*}
\begin{array}{l}
b(t,\phi )=b_{1}(t,\phi (t))\mathbb{1}_{[0,t_{1})}(t)+b_{2}(t,\phi
(t_{1}),\phi (t)-\phi (t_{1}))\mathbb{1}_{[t_{1},t_{2})}(t)+\dots \medskip
\\
\quad +b_{n}(t,\phi (t_{1}),\phi (t_{2})-\phi (t_{1}),\dots ,\phi (t)-\phi
(t_{n-1}))\mathbb{1}_{[t_{n-1},T]}(t),\bigskip \\
\sigma (t,\phi )=\sigma _{1}(t,\phi (t))\mathbb{1}_{[0,t_{1})}(t)+\sigma
_{2}(t,\phi (t_{1}),\phi (t)-\phi (t_{1}))\mathbb{1}_{[t_{1},t_{2})}(t)+%
\dots \medskip \\
\quad +\sigma _{n}(t,\phi (t_{1}),\phi (t_{2})-\phi (t_{1}),\dots ,\phi
(t)-\phi (t_{n-1}))\mathbb{1}_{[t_{n-1},T]}(t)%
\end{array}%
\end{equation*}%
and%
\begin{equation*}
\begin{array}{l}
f(t,\phi ,y,z)=f_{1}(t,\phi (t),y,z)\mathbb{1}_{[0,t_{1})}(t)+f_{2}(t,\phi
(t_{1}),\phi (t)-\phi (t_{1}),y,z)\mathbb{1}_{[t_{1},t_{2})}(t)+\dots
\medskip \\
\quad +f_{n}(t,\phi (t_{1}),\phi (t_{2})-\phi (t_{1}),\dots ,\phi (t)-\phi
(t_{n-1}))\mathbb{1}_{[t_{n-1},T]}(t),\bigskip \\
h(\phi )=\varphi (\phi (t_{1}),\phi (t_{2})-\phi (t_{1}),\dots ,\phi
(T)-\phi (t_{n-1})),%
\end{array}%
\end{equation*}%
for every $\phi \in \mathbb{\Lambda }$.

Let us first show that the terms $(X^{t,\phi }(t_{1}),\dots ,X^{t,\phi
}(r)-X^{t,\phi }(t_{k}))$, $0\leq k\leq n-1$ are related to the solution of
a SDE of It\^{o} type in $\mathbb{R}^{n\times d}$. Let%
\begin{equation*}
\tilde{b}(t,x_{1},\dots ,x_{n}):=\left(
\begin{array}{l}
b_{1}(t,x_{1})\mathbb{1}_{[0,t_{1})}(t)\medskip \\
b_{2}(t,x_{1},x_{2})\mathbb{1}_{[t_{1},t_{2})}(t)\medskip \\
\vdots \\
b_{n}(t,x_{1},\dots ,x_{n})\mathbb{1}_{[t_{n-1},t_{n})}(t)%
\end{array}%
\right) ,
\end{equation*}%
and%
\begin{equation*}
\tilde{\sigma}(t,x_{1},\dots ,x_{n}):=\left(
\begin{array}{l}
\sigma _{1}(t,x_{1})\mathbb{1}_{[0,t_{1})}(t)\medskip \\
\sigma _{2}(t,x_{1},x_{2})\mathbb{1}_{[t_{1},t_{2})}(t)\medskip \\
\vdots \\
\sigma _{n}(t,x_{1},\dots ,x_{n})\mathbb{1}_{[t_{n-1},t_{n})}(t)%
\end{array}%
\right) .
\end{equation*}%
Let then $\tilde{X}^{t,\boldsymbol{x}}$, with $t\in \lbrack 0,T]$ and $%
\boldsymbol{x}=(x_{1},\dots ,x_{n})\in \mathbb{R}^{d\times n}$ be the unique
solution of the following stochastic differential equation:%
\begin{equation*}
\tilde{X}^{t,\boldsymbol{x}}(s)=\boldsymbol{x}+\int_{t}^{s}\tilde{b}(r,%
\tilde{X}^{t,\boldsymbol{x}})dr+\int_{t}^{s}\tilde{\sigma}(r,\tilde{X}^{t,%
\boldsymbol{x}})dW(r),\ s\in \lbrack t,T].
\end{equation*}%
We assert that, for $t\in \lbrack t_{k_{0}},t_{k_{0}+1}),\ s\in \lbrack
t_{k},t_{k+1}]$, $s\geq t$, with $0\leq k_{0}\leq k\leq n-1$, we have%
\begin{equation*}
(X^{t,\phi }(t_{1}),\dots ,X^{t,\phi }(s)-X^{t,\phi }(t_{k}))=\Big(\tilde{X}%
^{i,t,(\phi (t_{1}),\dots ,\phi (t)-\phi (t_{k_{0}}),0,\dots ,0)}(s)\Big)_{i=%
\overline{1,k+1}}\,.
\end{equation*}%
Let us stress that for $k=0$ this reads $X^{t,\phi }(s)=\tilde{X}^{1,t,\phi
(t)}(s)$; so that for $k>k_{0}=0$, it is interpreted as $(X^{t,\phi
}(t_{1}),\dots ,X^{t,\phi }(s)-X^{t,\phi }(t_{k}))=\big(\tilde{X}^{i,t,(\phi
(t),0,\dots ,0)}(s)\big)_{i=\overline{1,k+1}}\,.$

We will prove this statement by induction on $k$. If $k=0$, then $k_{0}=0$
and we obviously have%
\begin{equation*}
X^{t,\phi }(s)=\tilde{X}^{1,t,\phi (t)}(s).
\end{equation*}%
Let us suppose that the statement holds true for $k-1$; for the sake of
brevity, in what follows we will denote $\boldsymbol{x}:=(\phi (t_{1}),\dots
,\phi (t)-\phi (t_{k_{0}}),0,\dots ,0)$.

If $k_{0}\leq k-1$ then, from the induction hypothesis, we have%
\begin{equation*}
(X^{t,\phi }(t_{1}),\dots ,X^{t,\phi }(r)-X^{t,\phi }(t_{k-1}))=\left(
\tilde{X}^{i,t,(\phi (t_{1}),\dots ,\phi (t)-\phi (t_{k_{0}}),0,\dots
,0)}(r)\right) _{i=\overline{1,k}}
\end{equation*}%
for every $r\in \lbrack t_{k-1},t_{k}]$, so that, for $s\in \lbrack
t_{k},t_{k+1}]$ we have,%
\begin{equation*}
\tilde{X}^{j,t,\boldsymbol{x}}(s)=\tilde{X}^{j,t,\boldsymbol{x}%
}(t_{k})=X^{t,\phi }(t_{j})-X^{t,\phi }(t_{j-1}),\ 1\leq j\leq k\,,
\end{equation*}%
with the convention $X^{t,\phi }(t_{0})=0$.

In the case $k_{0}=k$, for $s\in\lbrack t,t_{k+1}]$ we also have:
\begin{equation*}
\tilde{X}^{j,t,\boldsymbol{x}}(s)=\boldsymbol{x}^{j}=\phi(t_{j})-\phi
(t_{j-1})=X^{t,\phi}(t_{j})-X^{t,\phi}(t_{j-1}),\ 1\leq j\leq k\,,
\end{equation*}
again with the convention $\phi(t_{0})=0$.

Consequently, on $[t\vee t_{k},t_{k+1}]$ it holds%
\begin{equation*}
\begin{array}{l}
\displaystyle\tilde{X}^{k+1,t,\boldsymbol{x}}(s)=\boldsymbol{x}%
^{k+1}+\int_{t\vee t_{k}}^{s}b_{k+1}(r,\tilde{X}^{1,t,\boldsymbol{x}%
}(r),\dots ,\tilde{X}^{k+1,t,\boldsymbol{x}}(r))dr\medskip \\
\displaystyle\quad +\int_{t\vee t_{k}}^{s}\sigma _{k+1}(r,\tilde{X}^{1,t,%
\boldsymbol{x}}(r),\dots ,\tilde{X}^{k+1,t,\boldsymbol{x}}(r))dW(r)\medskip
\\
\displaystyle=\boldsymbol{x}^{k+1}+\int_{t\vee t_{k}}^{s}b_{k+1}(r,X^{t,\phi
}(t_{1}),\dots ,X^{t,\phi }(t_{k})-X^{t,\phi }(t_{k-1}),\tilde{X}^{k+1,t,%
\boldsymbol{x}}(r))dr\medskip \\
\displaystyle\quad +\int_{t\vee t_{k}}^{s}\sigma _{k+1}(r,X^{t,\phi
}(t_{1}),\dots ,X^{t,\phi }(t_{k})-X^{t,\phi }(t_{k-1}),\tilde{X}^{k+1,t,%
\boldsymbol{x}}(r))dW(r).%
\end{array}%
\end{equation*}%
If $k_{0}\leq k-1$ then $\boldsymbol{x}^{k+1}=0$; if $k_{0}=k$, then $%
\boldsymbol{x}^{k+1}=\phi (t)-\phi (t_{k_{0}})=X^{t,\phi }(t)-X^{t,\phi
}(t_{k})$. By uniqueness, since $X^{t,\phi }$ satisfies%
\begin{equation*}
\begin{array}{l}
\displaystyle X^{t,\phi }(s)=X^{t,\phi }(t\vee t_{k})+\int_{t\vee
t_{k}}^{s}b_{k+1}(r,X^{t,\phi }(t_{1}),\dots ,X^{t,\phi }(r)-X^{t,\phi
}(t_{k}))dr\medskip \\
\displaystyle\quad +\int_{t\vee t_{k}}^{s}\sigma _{k+1}(r,X^{t,\phi
}(t_{1}),\dots ,X^{t,\phi }(r)-X^{t,\phi }(t_{k}))dW(r),\ s\in \lbrack t\vee
t_{k},t_{k+1}].%
\end{array}%
\end{equation*}%
we obtain $\tilde{X}^{k+1,t,\boldsymbol{x}}(s)=X^{t,\phi }(s)-X^{t,\phi
}(t_{k})$, for all $s\in \lbrack t\vee t_{k},t_{k+1}]$. We thus have proved
that the statement holds true for $k$.

The next step is to derive a Feynman--Kac type formula linking $X^{t,\phi }$
and $Y^{t,\phi }$. For that, let us consider the following BSDE:%
\begin{equation*}
\tilde{Y}^{t,\boldsymbol{x}}(s)=\varphi (\tilde{X}^{t,\boldsymbol{x}%
}(T))+\int_{t}^{T}\tilde{f}(r,\tilde{X}^{t,\boldsymbol{x}}(r),\tilde{Y}^{t,%
\boldsymbol{x}}(r),\tilde{Z}^{t,\boldsymbol{x}}(r))dr+\int_{t}^{s}\tilde{Z}%
^{t,\boldsymbol{x}}dW(r),\ s\in \lbrack t,T],
\end{equation*}%
where $\tilde{f}$ is defined by%
\begin{equation*}
\begin{array}{l}
\displaystyle\tilde{f}(t,x_{1},\dots ,x_{n},y,z)\medskip \\
\displaystyle=f_{1}(t,x_{1},y,z)\mathbb{1}%
_{[0,t_{1})}(t)+f_{2}(t,x_{1},x_{2},y,z)\mathbb{1}_{[t_{1},t_{2})}(t)+\dots
+f_{n}(t,x_{1},x_{2},\dots ,x_{n})\mathbb{1}_{[t_{n-1},T]}(t).%
\end{array}%
\end{equation*}%
From \cite{ka-pe-qu/97}, there exist some measurable functions $\tilde{u},%
\tilde{d}$ such that for all $\left( t,\boldsymbol{x}\right) \in \lbrack
0,T]\times \mathbb{R}^{d\times n}$ it holds%
\begin{align*}
\tilde{Y}^{t,\boldsymbol{x}}(s)& =\tilde{u}(s,\tilde{X}^{t,\boldsymbol{x}%
}(s)),\quad \text{for all }s\in \lbrack t,T]; \\
\tilde{Z}^{t,\boldsymbol{x}}(s)& =\tilde{d}(s,\tilde{X}^{t,\boldsymbol{x}%
}(s))\tilde{\sigma}(s,\tilde{X}^{t,\boldsymbol{x}}),\ ds\text{-a.e. on }%
[t,T].
\end{align*}%
On the other hand, let $t\in \lbrack t_{k_{0}},t_{k_{0}+1})$, with $0\leq
k_{0}\leq n-1$ and denote, for simplicity, $\boldsymbol{x}=(\phi
(t_{1}),\dots ,\phi (t)-\phi (t_{k_{0}}),0,\dots ,0)$.

If $s\in \lbrack t_{k},t_{k+1}]$, $s\geq t$, and therefore $k\geq k_{0}$, we
have%
\begin{equation*}
(X^{t,\phi }(t_{1}),\dots ,X^{t,\phi }(s)-X^{t,\phi }(t_{k}))=\left( \tilde{X%
}^{i,t,\boldsymbol{x}}(s)\right) _{i=\overline{1,k+1}}
\end{equation*}%
Thus, on $[t\vee t_{k},t_{k+1}]$%
\begin{align*}
\tilde{f}(s,\tilde{X}^{t,\boldsymbol{x}}(s),\tilde{Y}^{t,\boldsymbol{x}}(s),%
\tilde{Z}^{t,\boldsymbol{x}}(s))& =f_{k+1}(s,X^{t,\phi }(t_{1}),\dots
,X^{t,\phi }(s)-X^{t,\phi }(t_{k}),\tilde{Y}^{t,\boldsymbol{x}}(s),\tilde{Z}%
^{t,\boldsymbol{x}}(s)) \\
& =f(s,X^{t,\phi },\tilde{Y}^{t,\boldsymbol{x}}(s),\tilde{Z}^{t,\boldsymbol{x%
}}(s)).
\end{align*}%
Allowing $k$ to vary, we obtain the equality%
\begin{equation*}
\tilde{f}(s,\tilde{X}^{t,\boldsymbol{x}}(s),\tilde{Y}^{t,\boldsymbol{x}}(s),%
\tilde{Z}^{t,\boldsymbol{x}}(s))=f(s,X^{t,\phi },\tilde{Y}^{t,\boldsymbol{x}%
}(s),\tilde{Z}^{t,\boldsymbol{x}}(s)),\quad \text{for all }s\in \lbrack t,T].
\end{equation*}%
Also,%
\begin{equation*}
\varphi (\tilde{X}^{t,\boldsymbol{x}}(T))=\varphi (X^{t,\phi }(t_{1}),\dots
,X^{t,\phi }(T)-X^{t,\phi }(t_{n-1}))=h(X^{t,\phi })\,,
\end{equation*}%
and by the uniqueness of the solution of the BSDE, we get that%
\begin{equation*}
(\tilde{Y}^{t,\boldsymbol{x}},\tilde{Z}^{t,\boldsymbol{x}})=(Y^{t,\phi
},Z^{t,\phi })
\end{equation*}%
and, consequently,%
\begin{equation*}
\begin{array}{l}
Y^{t,\phi }(s)=\tilde{u}(s,X^{t,\phi }(t_{1}),\dots ,X^{t,\phi
}(s)-X^{t,\phi }(t_{k}),0,\dots ,0),\quad \text{for all }s\in \lbrack t\vee
t_{k},t_{k+1}],\medskip \\
Z^{t,\phi }(s)=\tilde{d}_{k+1}(s,X^{t,\phi }(t_{1}),\dots ,X^{t,\phi
}(s)-X^{t,\phi }(t_{k}),0,\dots ,0)\cdot \medskip \\
\quad \quad \quad \quad \quad \cdot \sigma _{k+1}(s,X^{t,\phi }(t_{1}),\dots
,X^{t,\phi }(s)-X^{t,\phi }(t_{k})),\ ds\text{-a.e. on }[t\vee
t_{k},t_{k+1}].%
\end{array}%
\end{equation*}%
By setting, for $0\leq k\leq n-1$, $t\in \lbrack t_{k},t_{k+1})$ and $\phi
\in \mathbb{\Lambda }$,%
\begin{align*}
u(t,\phi )& :=\tilde{u}(t,\phi (t_{1}),\dots ,\phi (t)-\phi (t_{k}),0,\dots
,0); \\
d(t,\phi )& :=\tilde{d}_{k+1}(t,\phi (t_{1}),\dots ,\phi (t)-\phi
(t_{k}),0,\dots ,0),
\end{align*}%
we get that, for every $t\in \lbrack 0,T]$ and $\phi \in \mathbb{\Lambda }$,%
\begin{align*}
Y^{t,\phi }(s)& =u(s,X^{t,\phi }),\quad \text{for all }s\in \lbrack t,T]; \\
Z^{t,\phi }(s)& =d(s,X^{t,\phi })\sigma (t,X^{t,\phi }),\ ds\text{-a.e. on }%
[t,T].
\end{align*}%
\noindent \textbf{\textit{Step II.}}

Let us notice that the same conclusion holds for $b$, $\sigma $ and $f$ of
the form%
\begin{equation*}
\begin{array}{l}
\displaystyle b(t,\phi )=b_{1}(t,\phi (t))\mathbf{1}_{[0,t_{1})}(t)+b_{2}(t,%
\phi (t_{1}),\phi (t))\mathbf{1}_{[t_{1},t_{2})}(t)+\dots \medskip \\
\displaystyle\quad +b_{n}(t,\phi (t_{1}),\dots ,\phi (t_{n-1}),\phi (t))%
\mathbf{1}_{[t_{n-1},T]}(t),\bigskip \\
\displaystyle\sigma (t,\phi )=\sigma _{1}(t,\phi (t))\mathbf{1}%
_{[0,t_{1})}(t)+\sigma _{2}(t,\phi (t_{1}),\phi (t))\mathbf{1}%
_{[t_{1},t_{2})}(t)+\dots \medskip \\
\displaystyle\quad +\sigma _{n}(t,\phi (t_{1}),\dots ,\phi (t_{n-1}),\phi
(t))\mathbf{1}_{[t_{n-1},T]}(t)%
\end{array}%
\end{equation*}%
and%
\begin{equation*}
\begin{array}{l}
\displaystyle f(t,\phi ,y,z)=f_{1}(t,\phi (t),y,z)\mathbf{1}%
_{[0,t_{1})}(t)+f_{2}(t,\phi (t_{1}),\phi (t),y,z)\mathbf{1}%
_{[t_{1},t_{2})}(t)+\dots \medskip \\
\displaystyle\quad +f_{n}(t,\phi (t_{1}),\dots ,\phi (t_{n-1}),\phi (t))%
\mathbf{1}_{[t_{n-1},T]}(t),\bigskip \\
\displaystyle h(\phi )=\varphi (\phi (t_{1}),\dots ,\phi (t_{n-1}),\phi (T)),%
\end{array}%
\end{equation*}%
for every $\phi \in \mathbb{\Lambda }$.$\medskip $

\noindent\textbf{\textit{Step III.}}

For $0=t_{0}\leq t_{1}<\dots<t_{k}\leq T$ and $x_{1},\dots,x_{k}\in \mathbb{R%
}^{d}$, let $\Phi_{t_{1},\dots,t_{k}}^{x_{1},\dots,x_{k}}:[0,T]\rightarrow%
\mathbb{R}^{d}$ be such that
\begin{equation*}
\Phi_{t_{1},\dots,t_{k}}^{x_{1},\dots,x_{k}}(t_{i})=x_{i},\ i=\overline {1,k}%
;\quad\Phi_{t_{1},\dots,t_{k}}^{x_{1},\dots,x_{k}}(T)=x_{k},\quad
\Phi_{t_{1},\dots,t_{k}}^{x_{1},\dots,x_{k}}(0)=x_{1}
\end{equation*}
and is prolonged to $\left[ 0,T\right] $ by linear interpolation.

Let us consider partitions of $[0,T]$, $0=t_{0}^{n}<t_{1}^{n}<%
\dots<t_{n}^{n}=T$, $t_{k}^{n}:=\frac{kT}{n}$. For $k\in\{1,\dots,n\}$, $%
t\in\lbrack0,T]$ and $x_{1},\dots,x_{k}\in\mathbb{R}^{d}$, we define%
\begin{equation*}
b_{k}^{n}(t,x_{1},\dots,x_{k}):=b\big(t,\Phi_{t_{1}^{n},%
\dots,t_{k}^{n}}^{x_{1},\dots,x_{k}}\big).
\end{equation*}
Notice that $b_{k}^{n}(t,\cdot)$ are continuous functions.

Finally, for $t\in \lbrack 0,T]$ and $x_{1},\dots ,x_{n}\in \mathbb{R}^{d}$
we set%
\begin{equation*}
\bar{b}_{n}(t,x_{1},\dots ,x_{n}):=b_{1}^{n}(t,x_{1})\mathbb{1}%
_{[0,t_{1}^{n})}+\dots +b_{n}^{n}(t,x_{1},\dots ,x_{n})\mathbb{1}%
_{[t_{n-1}^{n},t_{n}^{n}]}
\end{equation*}%
and, for $\left( t,\phi \right) \in \lbrack 0,T]\times \mathbb{\Lambda }$,%
\begin{equation*}
b_{n}(t,\phi ):=\bar{b}_{n}(t,\phi (t\wedge t_{1}^{n}),\dots ,\phi (t\wedge
t_{n}^{n})).
\end{equation*}%
If $t\in \lbrack t_{k-1}^{n},t_{k}^{n})$ with $k\in \{1,\dots ,n\}$, or if $%
t=T$ and $k=n$, then%
\begin{equation*}
b_{n}(t,\phi )=b_{k}^{n}(t,\phi (t_{1}^{n}),\dots ,\phi (t_{k-1}^{n}),\phi
(t)),
\end{equation*}%
so $b_{n}$ has the form described in \textbf{\textit{Step II.} }In this case
it also holds%
\begin{align*}
\left\vert b_{n}(t,\phi )-b(t,\phi )\right\vert & =\big|b\big(t,\Phi
_{t_{1}^{n},\dots ,t_{k-1}^{n},t_{k}^{n}}^{\phi (t_{1}^{n}),\dots ,\phi
(t_{k-1}^{n}),\phi (t)}\big)-b(t,\phi )\big| \\
& \leq \sup_{\left\Vert \psi -\phi \right\Vert _{T}\leq \omega (\phi
,T/n)}\left\vert b\left( t,\psi \right) -b(t,\phi )\right\vert \leq \ell
\omega (\phi ,T/n),
\end{align*}%
where $\omega (\phi ,\epsilon ):=\sup_{\left\vert s-r\right\vert \leq
\epsilon }\left\vert \phi (s)-\phi (r)\right\vert $.

In a similar way, one introduces $\sigma _{n},f_{n},h_{n}$ and we have, for
every $(t,\phi ,y,z)\in \lbrack 0,T]\times \mathbb{\Lambda }\times \mathbb{R}%
\times \mathbb{R}^{d^{\prime }}$,%
\begin{align*}
\left\vert \sigma _{n}(t,\phi )-\sigma (t,\phi )\right\vert & \leq \ell
\omega (\phi ,T/n);\smallskip \\
\left\vert f_{n}(t,\phi ,y,z)-f(t,\phi ,y,z)\right\vert & \leq
\sup_{\left\Vert \psi -\phi \right\Vert _{T}\leq \omega (\phi
,T/n)}\left\vert f\left( t,\psi ,y,z\right) -f(t,\phi ,y,z)\right\vert
;\smallskip \\
\left\vert h_{n}(\phi )-h(\phi )\right\vert & \leq \sup_{\left\Vert \psi
-\phi \right\Vert _{T}\leq \omega (\phi ,T/n)}\left\vert h\left( \psi
\right) -h(\phi )\right\vert .
\end{align*}%
We also have that $b_{n}$, $\sigma _{n}$, $f_{n}$ and $h_{n}$ satisfy
assumptions $\mathrm{(A}_{1}\mathrm{)}$,$\mathrm{(A}_{2}\mathrm{)}$,$\mathrm{%
(A}_{6}\mathrm{)}$,$\mathrm{(A}_{7}\mathrm{)}$ with the same constants.
Corresponding to these coefficients, we define the solution of the
associated FBSDE,
\begin{equation*}
\left( X^{n,t,\phi },Y^{n,t,\phi },Z^{n,t,\phi }\right) \,.
\end{equation*}

By the Feynman--Kac formula, already proven in this case, we have the
existence of some non-anticipative functions $u_{n}$ and $d_{n}$ such that,
for every $(t,\phi )\in \lbrack 0,T]\times \mathbb{\Lambda }$ we have:
\begin{align*}
Y^{n,t,\phi }(s)& =u_{n}(s,X^{n,t,\phi }),\quad \text{for all }s\in \lbrack
t,T]; \\
Z^{n,t,\phi }(s)& =d_{n}(s,X^{n,t,\phi })\sigma (t,X^{n,t,\phi }),\quad ds%
\text{--a.e. on }[t,T].
\end{align*}%
In order to pass to the limit in the first formula above, we need to show
that $X^{n,t,\phi }\rightarrow X^{t,\phi }$ in probability in $\mathbb{%
\Lambda }$ and that $u_{n}$ converges to $u\ $on compact subsets of $\mathbb{%
\Lambda }$.

Let $t\in \lbrack 0,T]$ and $\phi ,\phi ^{\prime }\in \mathbb{\Lambda }$. By
It\^{o}'s formula we have%
\begin{equation*}
\begin{array}{l}
\displaystyle\big|X^{n,t,\phi ^{\prime }}(s)-X^{t,\phi }(s)\big|^{2}=|\phi
^{\prime }(t)-\phi (t)|^{2}\medskip \\
\displaystyle\quad +2\int_{t}^{s}\big\langle b_{n}(r,X^{n,t,\phi ^{\prime
}})-b(r,X^{t,\phi }),X^{n,t,\phi ^{\prime }}(r)-X^{t,\phi }(r)\big\rangle %
dr+\int_{t}^{s}\big|\sigma _{n}(r,X^{n,t,\phi ^{\prime }})-\sigma
(r,X^{t,\phi })\big|^{2}dr\medskip \\
\displaystyle\quad +2\int_{t}^{s}\big\langle\sigma _{n}(r,X^{n,t,\phi
^{\prime }})-\sigma (r,X^{t,\phi }),\big(X^{n,t,\phi ^{\prime
}}(r)-X^{t,\phi }(r)\big)dW(r)\big\rangle.%
\end{array}%
\end{equation*}%
By standard computations using Gronwall's lemma, Burkholder--Davis--Gundy
inequalities and Theorem \ref{theorem 0}, we get, for some arbitrary $p>1,$%
\begin{equation}
\begin{array}{l}
\displaystyle\mathbb{E}\sup_{s\in \left[ 0,T\right] }\big|X^{n,t,\phi
^{\prime }}(s)-X^{t,\phi }(s)\big|^{2p}\leq C\left( ||\phi ^{\prime }-\phi
||_{T}^{2p}+\mathbb{E}\omega (X^{t,\phi },T/n)^{2p}\right) \\
\displaystyle\leq C\left( ||\phi ^{\prime }-\phi ||_{T}^{2p}+\frac{1+||\phi
||_{T}^{2p}}{n^{p-1}}+\omega (\phi ,T/n)^{2p}\right) ,%
\end{array}
\label{unif conv_X}
\end{equation}%
where the positive constant $C$ is depending only on $p$ and the parameters
of our FBSDE; also, as above, the constant $C$ is allowed to change value
from line to line during this proof.

Applying now It\^{o}'s formula to $e^{\beta s}\big|Y^{n,t,\phi ^{\prime
}}(s)-Y^{t,\phi }(s)\big|^{2}$, we obtain, for $\beta >0$,%
\begin{equation*}
\begin{array}{l}
\displaystyle e^{\beta s}\big|\Delta _{Y}^{n,t,\phi ,\phi ^{\prime }}(s)\big|%
^{2}+\beta \int_{s}^{T}e^{\beta r}\big|\Delta _{Z}^{n,t,\phi ,\phi ^{\prime
}}(r)\big|^{2}dr+\int_{s}^{T}e^{\beta r}\big|\Delta _{Z}^{n,t,\phi ,\phi
^{\prime }}(r)\big|^{2}dr\medskip \\
\displaystyle=e^{\beta T}\big|\Delta _{h}^{n,t,\phi ,\phi ^{\prime }}\big|%
^{2}+2\int_{s}^{T}\big\langle\Delta _{f}^{n,t,\phi ,\phi ^{\prime
}}(r),e^{\beta r}\Delta _{Y}^{n,t,\phi ,\phi ^{\prime }}(r)\big\rangle %
dr\medskip \\
\displaystyle+2\int_{s}^{T}\big\langle f_{n}(s,X^{n,t,\phi ^{\prime
}},Y^{n,t,\phi ^{\prime }}\left( s\right) ,Z^{n,t,\phi ^{\prime }}\left(
s\right) )-f_{n}(s,X^{n,t,\phi ^{\prime }},Y^{t,\phi }\left( s\right)
,Z^{t,\phi }\left( s\right) ),e^{\beta r}\Delta _{Y}^{n,t,\phi ,\phi
^{\prime }}(r)\big\rangle dr\medskip \\
\displaystyle-2\int_{s}^{T}\big\langle e^{\beta r}\Delta _{Y}^{n,t,\phi
,\phi ^{\prime }}(r),\Delta _{Z}^{n,t,\phi ,\phi ^{\prime }}(r)dW(r)%
\big\rangle,\medskip%
\end{array}%
\end{equation*}%
where, for the sake of simplicity, we have denoted%
\begin{equation*}
\begin{array}{l}
\displaystyle\Delta _{Y}^{n,t,\phi ,\phi ^{\prime }}(s):=Y^{n,t,\phi
^{\prime }}\left( s\right) -Y^{t,\phi }\left( s\right) ;\medskip \\
\displaystyle\Delta _{Z}^{n,t,\phi ,\phi ^{\prime }}(s):=Z^{n,t,\phi
^{\prime }}\left( s\right) -Z^{t,\phi }\left( s\right) ;\medskip \\
\displaystyle\Delta _{h}^{n,t,\phi ,\phi ^{\prime }}:=h_{n}(X^{n,t,\phi
^{\prime }})-h(X^{t,\phi });\medskip \\
\displaystyle\Delta _{f}^{n,t,\phi ,\phi ^{\prime }}(s):=f_{n}(s,X^{n,t,\phi
^{\prime }},Y^{t,\phi }\left( s\right) ,Z^{t,\phi }\left( s\right)
)-f(s,X^{t,\phi },Y^{t,\phi }\left( s\right) ,Z^{t,\phi }\left( s\right) ).%
\end{array}%
\end{equation*}%
Again, for $\beta $ sufficiently large, exploiting Burkholder--Davis--Gundy
inequalities and the Lipschitz property in $\left( y,z\right) \in \mathbb{R}%
\times \mathbb{R}^{d^{\prime }}$ of $f_{n}$, we infer that%
\begin{equation}
\mathbb{E}\sup_{s\in \lbrack t,T]}\big|\Delta _{Y}^{n,t,\phi ,\phi ^{\prime
}}(s)\big|^{2}+\mathbb{E}\int_{t}^{T}\big|\Delta _{Z}^{n,t,\phi ,\phi
^{\prime }}(r)\big|^{2}dr\leq C\mathbb{E}\left[ \big|\Delta _{h}^{n,t,\phi
,\phi ^{\prime }}\big|^{2}+\int_{t}^{T}\big|\Delta _{f}^{n,t,\phi ,\phi
^{\prime }}(r)\big|^{2}dr\right] .  \label{unif conv_Y}
\end{equation}

We now have%
\begin{equation}
\begin{array}{l}
\displaystyle\big|\Delta _{h}^{n,t,\phi ,\phi ^{\prime }}\big|\leq \big|%
h_{n}(X^{n,t,\phi ^{\prime }})-h(X^{n,t,\phi ^{\prime }})\big|+\big|%
h(X^{n,t,\phi ^{\prime }})-h(X^{t,\phi })\big|\medskip \\
\displaystyle\leq \sup_{||\psi -X^{n,t,\phi ^{\prime }}||_{T}\leq \omega
(X^{n,t,\phi ^{\prime }},T/n)}\big|h\left( \psi \right) -h(X^{n,t,\phi
^{\prime }})\big|+\big|h(X^{n,t,\phi ^{\prime }})-h(X^{t,\phi })\big|\medskip
\\
\displaystyle\leq \sup_{||\psi -X^{t,\phi }||_{T}\leq \omega (X^{t,\phi
},T/n)+3||X^{n,t,\phi ^{\prime }}-X^{t,\phi }||_{T}}|h\left( \psi \right)
-h(X^{t,\phi })|+2\big|h(X^{n,t,\phi ^{\prime }})-h(X^{t,\phi })\big|\medskip
\\
\displaystyle\leq 3\sup_{||\psi -X^{t,\phi }||_{T}\leq \omega (X^{t,\phi
},T/n)+3||X^{n,t,\phi ^{\prime }}-X^{t,\phi }||_{T}}|h\left( \psi \right)
-h(X^{t,\phi })|,%
\end{array}
\label{unif conv_h}
\end{equation}%
since%
\begin{equation*}
\omega (X^{n,t,\phi ^{\prime }},T/n)\leq \omega (X^{t,\phi
},T/n)+2||X^{n,t,\phi ^{\prime }}-X^{t,\phi }||_{T}.
\end{equation*}%
Similarly, we obtain%
\begin{equation}
\begin{array}{r}
\big|\Delta _{f}^{n,t,\phi ,\phi ^{\prime }}(s)\big|\leq 3\sup_{||\psi
-X^{t,\phi }||_{T}\leq \omega (X^{t,\phi },T/n)+3||X^{n,t,\phi ^{\prime
}}-X^{t,\phi }||_{T}}\big|f\left( s,\psi ,Y^{t,\phi }\left( s\right)
,Z^{t,\phi }\left( s\right) \right) \medskip \\
-f(s,X^{t,\phi },Y^{t,\phi }\left( s\right) ,Z^{t,\phi }\left( s\right) )%
\big|.%
\end{array}
\label{unif conv_f}
\end{equation}%
Let now $\left( t,\phi \right) \in \lbrack 0,T]\times \mathbb{\Lambda }$ and
let $\left( \phi _{n}\right) $ be a sequence converging to $\phi $ in $%
\mathbb{\Lambda }$. It is clear from relation (\ref{unif conv_X}) that $%
\left( X^{n,t,\phi _{n}}\right) $ converges in $L^{2p}\left( \Omega ;\mathbb{%
\Lambda }\right) $ to $X^{t,\phi }$, therefore there exists a subsequence
converging a.s. in $\mathbb{\Lambda }$ to $X^{t,\phi }$. Without restricting
the generality, we will still denote $\left( X^{n,t,\phi _{n}}\right) $ this
subsequence. Since $\omega (X^{t,\phi },T/n)+3||X^{n,t,\phi ^{\prime
}}-X^{t,\phi }||_{T}$ converges to $0$ a.s., it is clear by relations (\ref%
{unif conv_h}) and (\ref{unif conv_X}) that $\Delta _{h}^{n,t,\phi ,\phi
^{\prime }}$ and $\Delta _{f}^{n,t,\phi ,\phi ^{\prime }}(s)$ converge to $0$%
, a.s., respectively $ds\mathbb{P}$-a.e.

Then, by the dominated convergence theorem, we obtain
\begin{equation*}
\mathbb{E}\left[ \big|\Delta _{h}^{n,t,\phi ,\phi _{n}}\big|^{2}+\int_{t}^{T}%
\big|\Delta _{f}^{n,t,\phi ,\phi _{n}}(r)\big|^{2}dr\right] \rightarrow 0,
\end{equation*}%
which, combined with estimate (\ref{unif conv_Y}), gives that $\mathbb{E}%
\sup_{s\in \lbrack t,T]}|Y^{n,t,\phi _{n}}\left( s\right) -Y^{t,\phi }\left(
s\right) |^{2}\rightarrow 0.$ On the one hand, assuming $\phi _{n}\equiv
\phi $, this implies%
\begin{equation*}
\mathbb{E}\sup_{s\in \lbrack t,T]}|Y^{n,t,\phi }\left( s\right) -Y^{t,\phi
}\left( s\right) |^{2}\rightarrow 0\,;
\end{equation*}%
on the other hand, letting $s=t$, we obtain%
\begin{equation*}
u_{n}(t,\phi _{n})\rightarrow u(t,\phi ).
\end{equation*}%
Therefore we can pass to the limit in the relation%
\begin{equation*}
Y^{n,t,\phi }(s)=u_{n}(s,X^{n,t,\phi }),\quad \text{for all }s\in \lbrack
t,T],\ \text{a.s.}\,,
\end{equation*}%
and find the conclusion of the theorem.\hfill $\medskip $
\end{proof}

We are able now to prove the Feynman--Kac formula in the case where the
generator $f$ depends of the past values of $Y$.

\begin{theorem}
\label{Feynman-Kac formula}Let us assume that $\mathrm{(A}_{1}\mathrm{)}$,$%
\mathrm{(A}_{2}\mathrm{)}$,$\mathrm{(A}_{6}\mathrm{)}$,$\mathrm{(A}_{7}%
\mathrm{)}$ hold and condition (\ref{condition_KT}) is verified. Then%
\begin{equation*}
Y^{t,\phi }\left( s\right) =u(s,X^{t,\phi }),\quad \text{for all }s\in \left[
0,T\right] ,\quad \text{a.s.,}
\end{equation*}%
for any $\left( t,\phi \right) \in \left[ 0,T\right] \times \mathbb{\Lambda }%
,$ where $Y^{t,\phi }$ is the solution of BSDE (\ref{BSDE_delayed 1}) and $u:%
\left[ 0,T\right] \times \mathbb{\Lambda }\rightarrow \mathbb{R}$ is defined
by (\ref{def_u}).
\end{theorem}

\begin{proof}
The continuity of $u$ was already asserted in Theorem \ref{theorem 2}.

Let now consider BSDE with delayed generator%
\begin{equation}
Y^{t,\phi }\left( s\right) =h(X^{t,\phi })+\int_{s}^{T}f(r,X^{t,\phi
},Y^{t,\phi }(r),Z^{t,\phi }(r),Y_{r}^{t,\phi })dr-\int_{s}^{T}Z^{t,\phi
}\left( r\right) dW\left( r\right) \,,  \label{FBSDE 3}
\end{equation}%
and the corresponding iterative process, given by the equations%
\begin{equation}
\begin{array}{r}
\displaystyle Y^{n+1,t,\phi }\left( s\right) =h(X^{t,\phi
})+\int_{s}^{T}f(r,X^{t,\phi },Y^{n+1,t,\phi }(r),Z^{n+1,t,\phi
}(r),Y_{r}^{n,t,\phi })dr\medskip \\
\displaystyle-\int_{s}^{T}Z^{n+1,t,\phi }\left( r\right) dW\left( r\right)
,\;s\in \lbrack t,T],%
\end{array}
\label{BSDE_iter}
\end{equation}%
with $Y^{0,t,\phi }\equiv 0$ and $Z^{0,t,\phi }\equiv 0$.

Let us suppose that there exists a $\mathbb{F}$--progressively measurable
functional $u_{n}:[0,T]\times \mathbb{\Lambda }\rightarrow \mathbb{R}$ such
that $u_{n}$ is continuous and $Y^{n,t,\phi }\left( s\right)
=u_{n}(s,X^{t,\phi }),$ for every $t,s\in \left[ 0,T\right] $ and $\phi \in
\mathbb{\Lambda }.$

Let us now consider the term%
\begin{equation*}
Y_{r}^{n,t,\phi }=\big(Y^{n,t,\phi }(r+\theta )\big)_{\theta \in \lbrack
-\delta ,0]}\,;
\end{equation*}%
since $Y^{n,t,\phi }(r+\theta )=u_{n}(r+\theta ,X^{t,\phi })$ if $r+\theta
\geq 0$ and $Y^{n,t,\phi }(r+\theta )=Y^{n,t,\phi }(0)=u_{n}(0,X^{t,\phi })$
if $r+\theta <0$, by defining%
\begin{equation*}
\tilde{u}_{n}(t,\phi ):=\left( u_{n}(\mathbb{1}_{[0,T]}(t+\theta ),\phi
)\right) _{\theta \in \lbrack -\delta ,0]}\,,
\end{equation*}%
we have%
\begin{equation*}
Y_{r}^{n,t,\phi }=\tilde{u}_{n}(r,X^{t,\phi }).
\end{equation*}%
We can then apply Theorem \ref{Feynman-Kac formula_preliminary} in order to
infer that%
\begin{equation*}
Y^{n+1,t,\phi }\left( s\right) =u_{n+1}(s,X^{t,\phi }),
\end{equation*}%
for a continuous non-anticipative functional $u_{n+1}:[0,T]\times \mathbb{%
\Lambda }\rightarrow \mathbb{R}$.

Notice that $\left( Y^{n,t,\phi },Z^{n,t,\phi }\right) $ is the Picard
iteration sequence used for constructing the solution $\left( Y^{t,\phi
},Z^{t,\phi }\right) $:%
\begin{equation*}
\left( Y^{n+1,\cdot ,\phi },Z^{n+1,\cdot ,\phi }\right) =\Gamma (Y^{n,\cdot
,\phi },Z^{n,\cdot ,\phi })\,,
\end{equation*}%
where $\Gamma $ is the contraction defined in the proof of Theorem \ref%
{theorem 3}. We then have:%
\begin{equation*}
\lim_{n\rightarrow \infty }\mathbb{E}\big(\sup_{s\in \lbrack 0,T]}\big|%
Y^{n,t,\phi }(s)-Y^{t,\phi }(s)\big|^{2}\big)=0.
\end{equation*}%
Of course, $u_{n}(t,\phi )$ converges to $u(t,\phi ):=\mathbb{E}Y^{t,\phi
}(t),$ for every $t\in \lbrack 0,T]$ and $\phi \in \mathbb{\Lambda }$. This
implies that the nonlinear Feynman--Kac formula $Y^{t,\phi }\left( s\right)
=u(s,X^{t,\phi })$ holds.\hfill $\medskip $
\end{proof}

\begin{remark}
From the above proof it can be seen why we have renounced to the dependence
of the driver on the past of $z.$ More precisely, even if we had a
representation of the form%
\begin{equation*}
Z^{n,t,\phi }\left( s\right) =d^{n}(s,X^{t,\phi })
\end{equation*}%
in Theorem \ref{Feynman-Kac formula_preliminary}, we couldn't use it because
of the lack of continuity of $d^{n}\,.$
\end{remark}

We are now able to prove Theorem \ref{theorem 1}, which shows the existence
of a viscosity solution to equation (\ref{PDKE}).\medskip

\begin{proof}[Proof of Theorem \protect\ref{theorem 1}]

We will prove that function $u$ defined by (\ref{def_u}) is a viscosity
solution to (\ref{PDKE}). In particular we will only show that $u$ is a
viscosity subsolution, the supersolution case being similar.

Suppose by contrary that $u$ is not a viscosity subsolution. Then, for any $%
L_{0}\geq 0$, there exists $L\geq L_{0}$ such that $u$ is not a viscosity $L$%
--subsolution in the sense of Definition \ref{definition 1}. Therefore,
there exist $\left( t,\phi \right) \in \lbrack 0,T]\times \mathbb{\Lambda }$
and $\varphi \in \underline{\mathcal{A}}^{L}u\left( t,\phi \right) $ such
that, for some $c>0,$%
\begin{equation*}
\partial _{t}\varphi (t,\phi )+\mathcal{L}\varphi (t,\phi )+f(t,\phi
,u(t,\phi ),\partial _{x}\varphi \left( t,\phi \right) \sigma (t,\phi
),\left( u(\cdot ,\phi )\right) _{t})\leq -c<0
\end{equation*}%
Using the definition of $\underline{\mathcal{A}}^{L}u\left( t,\phi \right) $
we see that there exists $\tau _{0}\in \mathcal{T}_{+}^{t}$ such that
\begin{equation*}
\varphi \left( t,\phi \right) -u\left( t,\phi \right) =\min\nolimits_{\tau
\in \mathcal{T}^{t}}\underline{\mathcal{E}}_{t}^{L}\left[ \left( \varphi
-u\right) \left( \tau \wedge \tau _{0},X^{t,\phi }\right) \right] \,.
\end{equation*}%
Let us now take%
\begin{equation*}
\begin{array}{l}
\displaystyle\tilde{\tau}:=T\wedge \tau _{0}\wedge \inf \{s>t:\partial
_{t}\varphi (s,X^{t,\phi })+\mathcal{L}\varphi (s,X^{t,\phi })\medskip \\
\displaystyle\quad \quad \quad \quad \quad \quad \quad \quad +f(s,X^{t,\phi
},u(s,X^{t,\phi }),\partial _{x}\varphi (s,X^{t,\phi })\sigma (s,X^{t,\phi
}),(u(\cdot ,X^{t,\phi }))_{s})>-c/2\}.%
\end{array}%
\end{equation*}%
By the continuity of the coefficients, we deduce that $\tilde{\tau}\in
\mathcal{T}_{+}^{t}$ $.$

Let us denote%
\begin{equation*}
\begin{array}{l}
(Y^{1}\left( s\right) ,Z^{1}\left( s\right) ):=(\varphi (s,X^{t,\phi
}),\partial _{x}\varphi (s,X^{t,\phi })\sigma (s,X^{t,\phi })),\quad \text{%
for }s\in \left[ t,T\right] ,\medskip \\
(Y^{2}\left( s\right) ,Z^{2}\left( s\right) ):=(Y^{t,\phi }\left( s\right)
,Z^{t,\phi }\left( s\right) ),\quad \text{for }s\in \left[ 0,T\right]%
\end{array}%
\end{equation*}%
and%
\begin{equation*}
\Delta Y\left( s\right) :=Y^{1}\left( s\right) -Y^{2}\left( s\right) ,\quad
\Delta Z\left( s\right) :=Z^{1}\left( s\right) -Z^{2}\left( s\right) ,\text{
for }s\in \left[ t,T\right] .
\end{equation*}%
Using It\^{o}'s formula we deduce, for any $s\in \left[ t,T\right] ,$%
\begin{equation*}
\displaystyle\varphi (s,X^{t,\phi })=\varphi (t,\phi )+\int_{t}^{s}\big(%
\partial _{t}\varphi (r,X^{t,\phi })+\mathcal{L}\varphi (r,X^{t,\phi })\big)%
dr+\int_{t}^{s}\langle \partial _{x}\varphi (r,X^{t,\phi }),\sigma
(r,X^{t,\phi })dW\left( r\right) \rangle ,
\end{equation*}%
so that we obtain%
\begin{equation*}
\begin{array}{l}
\displaystyle\Delta Y\left( \tilde{\tau}\right) -\Delta Y\left( t\right)
\medskip \\
\displaystyle=\int_{t}^{\tilde{\tau}}\big(\partial _{t}\varphi (r,X^{t,\phi
})+\mathcal{L}\varphi (r,X^{t,\phi })+f(r,X^{t,\phi },Y^{2}(r),Z^{2}\left(
r\right) ,Y_{r}^{2})\big)dr+\int_{t}^{\tilde{\tau}}\Delta Z\left( r\right)
dW\left( r\right) .%
\end{array}%
\end{equation*}%
Since for any $r\in \left[ t,\tilde{\tau}\right] $ we have%
\begin{equation*}
\begin{array}{l}
\displaystyle\partial _{t}\varphi (r,X^{t,\phi })+\mathcal{L}\varphi
(r,X^{t,\phi })+f(r,X^{t,\phi },Y^{2}(r),Z^{2}\left( r\right)
,Y_{r}^{2})\medskip \\
\displaystyle\leq -\frac{c}{2}+f(r,X^{t,\phi },Y^{2}(r),Z^{2}\left( r\right)
,Y_{r}^{2})-f(r,X^{t,\phi },u(r,X^{t,\phi }),\partial _{x}\varphi
(r,X^{t,\phi })\sigma (r,X^{t,\phi }),(u(\cdot ,X^{t,\phi }))_{r}),%
\end{array}%
\end{equation*}%
we deduce, using the Feynman--Kac formula $Y^{2}\left( r\right) =u\left(
r,X^{t,\phi }\right) ,$ that%
\begin{equation*}
\begin{array}{l}
\displaystyle\Delta Y\left( \tilde{\tau}\right) -\Delta Y\left( t\right)
\leq -\frac{c}{2}\left( \tilde{\tau}-t\right) +\int_{t}^{\tilde{\tau}}\Delta
Z\left( r\right) dW\left( r\right) \medskip \\
\displaystyle\quad +\int_{t}^{\tilde{\tau}}\big(f(r,X^{t,\phi
},u(r,X^{t,\phi }),Z^{2}\left( r\right) ,(u(\cdot ,X^{t,\phi
}))_{r}))-f(r,X^{t,\phi },u(r,X^{t,\phi }),Z^{1}\left( r\right) ,(u(\cdot
,X^{t,\phi }))_{r})\big)dr.%
\end{array}%
\end{equation*}%
We have that there exists $\lambda \in \mathcal{U}_{T}^{L}$ such that%
\begin{equation*}
\begin{array}{l}
\displaystyle f(r,X^{t,\phi },u(r,X^{t,\phi }),Z^{1}\left( r\right)
,(u(\cdot ,X^{t,\phi }))_{r}))-f(r,X^{t,\phi },u(r,X^{t,\phi }),Z^{2}\left(
r\right) ,(u(\cdot ,X^{t,\phi }))_{r})\medskip \\
\displaystyle=\left\langle \lambda \left( r\right) ,\Delta Z\left( r\right)
\right\rangle%
\end{array}%
\end{equation*}%
and therefore%
\begin{equation*}
\Delta Y\left( \tilde{\tau}\right) -\Delta Y\left( t\right) \medskip \leq -%
\frac{c}{2}\left( \tilde{\tau}-t\right) +\int_{t}^{\tilde{\tau}}\Delta
Z\left( r\right) \big(dW\left( r\right) -\lambda \left( r\right) dr\big)\,.
\end{equation*}%
Noticing now that $W\left( s\right) -\int_{t}^{s}\lambda \left( r\right) dr$
is a $\mathbb{P}^{t,\lambda }-$martingale, we obtain%
\begin{equation*}
\begin{array}{l}
\displaystyle\Delta Y\left( t\right) =\mathbb{E}^{\mathbb{P}^{t,\lambda
}}(\Delta Y\left( t\right) )\medskip \\
\displaystyle\geq \mathbb{E}^{\mathbb{P}^{t,\lambda }}(\Delta Y\left( \tilde{%
\tau}\right) )+\frac{c}{2}\mathbb{E}^{\mathbb{P}^{t,\lambda }}\left( \tilde{%
\tau}-t\right) +\mathbb{E}^{\mathbb{P}^{t,\lambda }}\int_{t}^{\tilde{\tau}%
}\Delta Z\left( r\right) \big(dW\left( r\right) -\lambda \left( r\right) dr%
\big)\medskip \\
\displaystyle>\mathbb{E}^{\mathbb{P}^{t,\lambda }}(\Delta Y\left( \tilde{\tau%
}\right) )=\mathbb{E}^{\mathbb{P}^{t,\lambda }}(\varphi (\tilde{\tau}%
,X^{t,\phi })-Y^{t,\phi }\left( \tilde{\tau}\right) )=\mathbb{E}^{\mathbb{P}%
^{t,\lambda }}\left[ \left( \varphi -u\right) (\tilde{\tau},X^{t,\phi })%
\right] \medskip \\
\displaystyle\geq \underline{\mathcal{E}}_{t}^{L}\left[ \left( \varphi
-u\right) (\tilde{\tau},X^{t,\phi })\right] =\underline{\mathcal{E}}_{t}^{L}%
\left[ \left( \varphi -u\right) (\tau _{0}\wedge \tilde{\tau},X^{t,\phi })%
\right] \medskip \\
\displaystyle\geq \min\nolimits_{\tau \in \mathcal{T}^{t}}\underline{%
\mathcal{E}}_{t}^{L}\left[ \left( \varphi -u\right) (\tau _{0}\wedge \tilde{%
\tau},X^{t,\phi })\right] =\varphi (t,X^{t,\phi })-Y^{t,\phi }\left(
t\right) =\Delta Y\left( t\right) ,%
\end{array}%
\end{equation*}%
which is a contradiction.

The proof is complete now.\hfill
\end{proof}

\section{Financial applications}

\label{SEC:FA}

In what follows we shall apply theoretical results developed in previous
sections in order to analyze some particular models of great interest in
modern finance.$\medskip $

Financial literature that shows how delay naturally arises when dealing with
asset price evolution or in general with certain financial instruments, is
nowadays wide and developed, see, for instance, \cite%
{ar-hu/07,ch-yo/99,ka-sw-hu/02,ka-sw-hu/07} and references therein. On the
other hand, not much is done when the delay enters the backward component.
We aim here to give some financial applications where also the backward
equation exhibits a delayed behaviour. We remark that since the goal of the
present work is purely theoretical, the examples provided will not be stated
in complete generality. Actually in this section we will show how the study
of BSDEs with delayed generators, together with the associated
path--dependent Kolmogorov equation, may lead to the study of a completely
new class of financial problems that have not been studied before.
Nevertheless, we intend to address in a future work the study of these
problems in complete generality.

BSDEs with delay have been introduced in \cite{de-im/10} as a pure
mathematical tool with no financial application of interest. Later on, some
works showing that the delay in the backward component arises naturally in
several applications have appeared, see, e.g. \cite{de/11,de/12}.

In what follows we will provide two extensions of forward--backward models
that have been proposed in past literature where the backward components can
exhibit a short--time delay.


\subsection{The large investor problem}

Following the model studied in \cite{cv-ma/96}, see also, e.g. \cite%
{ka-pe-qu/01}, we will consider in the current example a non--standard
investor acting on a financial market. We assume that this investor, usually
referred to as \textit{the large investor} in the literature, has superior
information about the stock prices and/or he is willing to invest a large
amount of money in the stock. This fact implies that the large investor may
influence the behaviour of the stock price with his actions. It is further
natural to assume that there is a short time delay between the investor's
actions and the reaction of the market to the large investor's actions. In
particular we assume that the drift coefficient of the underlying $S$ at
time $t$ depends on how the large investor acts on the market in the
interval $(t-\delta ,t)$.

Let us then consider a risky asset $S$ and a riskless bond $B$ evolving
according to
\begin{equation}
\left\{
\begin{array}{l}
\displaystyle\frac{dB(t)}{B(t)}=r(t,X(t),\pi\left( t\right)
,X_{t})dt\;,\quad B(0)=1\;,\medskip \\
\displaystyle\frac{dS(t)}{S(t)}=\mu\left( t,X(t),\pi\left( t\right)
,X_{t}\right) dt+\sigma(t,X(t),X_{t})dW(t),\quad S(0)=s_{0}>0.%
\end{array}
\right.  \label{Large Investor_SDE}
\end{equation}
Here we have denoted by $X$ the portfolio of the large investor. Also we
used the notations introduced by (\ref{definition delayed term 1}) and (\ref%
{definition delayed term 2}). We suppose that the coefficients $r,\mu$ and $%
\sigma$ satisfy some suitable regularity assumptions.

We have that the portfolio $X$, composed at any time $t\in \lbrack 0,T]$ by $%
\pi (t),$ the amount invested in the risky asset $S$ and by $X(t)-\pi (t),$
the amount invested in the riskless bond $B$, evolves according to%
\begin{equation*}
\begin{array}{l}
\displaystyle dX(t)=\frac{\pi (t)}{S\left( t\right) }dS\left( t\right) +%
\frac{X\left( t\right) -\pi (t)}{B\left( t\right) }dB\left( t\right) \medskip
\\
\displaystyle=\pi (t)\cdot \left[ \mu \left( t,X(t),\pi \left( t\right)
,X_{t}\right) dt+\sigma (t,X(t),X_{t})dW(t)\right] +\left[ X\left( t\right)
-\pi (t)\right] \cdot r(t,X(t),\pi \left( t\right) ,X_{t})dt\,,%
\end{array}%
\end{equation*}%
with the final condition $X(T)=h\left( S\right) .$

Hence, for $t\in \left[ 0,T\right] ,$%
\begin{equation}
\displaystyle X(t)=h\left( S\right) +\int_{t}^{T}F\left( s,X\left( s\right)
,\pi \left( s\right) ,X_{s},\pi _{s}\right) ds-\int_{t}^{T}\pi (s)\sigma
(s,X(s),X_{s})dW(s),  \label{Large Investor_BSDE}
\end{equation}%
where we have denoted for short%
\begin{equation}
\displaystyle F\left( s,X(s),\pi \left( s\right) ,X_{s},\pi _{s}\right) :=-%
\left[ X(s)-\pi (s)\right] \cdot r(s,X(s),\pi \left( s\right) ,X_{s})-\pi
(s)\cdot \mu \left( s,X(s),\pi \left( s\right) ,X_{s}\right) .
\label{definition F}
\end{equation}%
Since the forward equations in (\ref{Large Investor_SDE}) can be explicitly
solved by%
\begin{equation*}
\displaystyle S\left( t\right) =s_{0}\exp \Big[\int_{0}^{t}\left( \mu \left(
s,X(s),\pi \left( s\right) ,X_{s}\right) -\frac{1}{2}\sigma ^{2}\left(
s,X(s),X_{s}\right) \right) ds+\int_{0}^{t}\sigma \left( s,X(s),X_{s}\right)
dW\left( s\right) \Big],
\end{equation*}%
we deduce that $S$ is a functional of $X,\pi $ and $W$, i.e. there exists $%
\tilde{h}$ such that the final condition becomes $X\left( T\right) =\tilde{h}%
\left( W,X,\pi \right) $.

A first remark is that we can impose some suitable assumptions on the
functions $r,\mu$ and $\sigma$ such that the function $\bar{F}:[0,T]\times
\mathbb{R}\times\mathbb{R}\times L^{2}\left( [-\delta,0];\mathbb{R}\right)
\rightarrow\mathbb{R},$ defined by%
\begin{equation*}
\bar{F}\left( s,y,z,\hat{y}\right) :=-\left( y-z\sigma^{-1}\left( s,y,\hat{y}%
\right) \right) \cdot r(s,y,z\sigma^{-1}\left( s,y,\hat {y}\right) ,\hat{y}%
)-z\cdot\mu\left( s,y,z\sigma^{-1}\left( s,y,\hat {y}\right) ,\hat{y}\right)
,
\end{equation*}
satisfies assumptions $\mathrm{(A}_{3}\mathrm{)}$--$\mathrm{(A}_{5}\mathrm{).%
}$

The second remark concerns the fact that Theorem \ref{theorem 3} is still
true, with a slight adjustment of the proof, if we consider in the backward
equation (\ref{FBSDE 2}) the final condition
\begin{equation*}
\bar{h}\left( W,X,Z\right) :=\tilde{h}(W,X,\sigma ^{-1}\left( \cdot
,X,X_{\cdot }\right) Z)\,,
\end{equation*}%
instead of a functional of $W$ only (which represents the forward part
within the theoretical framework from Section \ref{Section 2}), with $\bar{h}
$ satisfying a Lipschitz condition:%
\begin{equation*}
\left\vert \bar{h}\left( x,y,z\right) -\bar{h}\left( x,y^{\prime },z^{\prime
}\right) \right\vert ^{2}\leq L\left[ \int_{0}^{T}\left\vert y\left(
s\right) -y^{\prime }\left( s\right) \right\vert
^{2}ds+\int_{0}^{T}\left\vert z\left( s\right) -z^{\prime }\left( s\right)
\right\vert ^{2}ds\right] .
\end{equation*}%
Of course, $\bar{h}$ will satisfy $\mathrm{(A}_{5}\mathrm{)}$ and the above
condition if we impose some suitable assumptions on $\mu $ and $\sigma $
(for instance, if we consider $\mu $ bounded and $\sigma $ constant).

Therefore we can rewrite (\ref{Large Investor_BSDE}) as%
\begin{equation}
X(t)=\bar{h}\left( W,X,Z\right) +\int_{t}^{T}\bar{F}\left( s,X\left(
s\right) ,Z\left( s\right) ,X_{s}\right) ds-\int_{t}^{T}Z\left( s\right)
dW(s),\quad t\in \left[ 0,T\right]  \label{EQN:LargInvRe}
\end{equation}%
and we deduce from Theorem \ref{theorem 3} that, under proper assumptions on
the coefficients, there exists a unique solution $\left( X,Z\right) $ to
equation \eqref{EQN:LargInvRe}.

Hence equation (\ref{Large Investor_BSDE}) admits a unique solution $\left(
X,\pi\right) $, where
\begin{equation*}
\pi\left( s\right) :=Z\left( s\right) \sigma^{-1}(s,X(s),X_{s})\, .
\end{equation*}

In order to obtain the connection with the PDE we consider first the
decoupled forward--backward stochastic system:%
\begin{equation*}
\left\{
\begin{array}{l}
\displaystyle\bar{W}^{t,\phi }\left( s\right) =\phi \left( t\right)
+\int_{t}^{s}dW\left( r\right) ,\quad s\in \left[ t,T\right] ,\medskip \\
\displaystyle\bar{W}^{t,\phi }\left( s\right) =\phi \left( s\right) ,\quad
s\in \lbrack 0,t),\medskip \\
\displaystyle X^{t,\phi }\left( s\right) =\bar{h}(\bar{W}^{t,\phi
},X^{t,\phi },Z^{t,\phi })+\int_{s}^{T}\bar{F}(r,X^{t,\phi }\left( r\right)
,Z^{t,\phi }\left( r\right) ,X_{r}^{t,\phi })dr\medskip \\
\displaystyle\quad \quad \quad \quad \quad \quad \quad \quad \quad \quad
\quad \quad \quad \quad \quad -\int_{s}^{T}Z^{t,\phi }(r)dW\left( r\right)
,~s\in \left[ t,T\right] ,\medskip \\
\displaystyle X^{t,\phi }\left( s\right) =X^{s,\phi }\left( s\right) ,\quad
\pi ^{t,\phi }\left( s\right) =0,\quad s\in \lbrack 0,t).%
\end{array}%
\right.
\end{equation*}%
Using Theorem \ref{theorem 3} we see that, under suitable assumptions on the
coefficients, there exists a unique solution $\left( X^{t,\phi },\pi
^{t,\phi }\right) _{\left( t,\phi \right) \in \left[ 0,T\right] \times
\mathbb{\Lambda }}$ of the above system.

From the results of the previous sections, in particular from theorem \ref%
{Feynman-Kac formula}, we have the following representation for the solution
$(X,Z)$ of the backward component in the above system. In particular we
have, for every $\left( t,\phi \right) \in \left[ 0,T\right] \times \mathbb{%
\Lambda },$%
\begin{equation*}
X^{t,\phi }\left( s\right) =u(s,\bar{W}^{t,\phi }),\quad \text{for all }s\in %
\left[ t,T\right] ,
\end{equation*}%
where $u\left( t,\phi \right) =X^{t,\phi }\left( t\right) $ is a viscosity
solution of the following path--dependent PDE:%
\begin{equation*}
\left\{
\begin{array}{l}
\displaystyle\partial _{t}u(t,\phi )+\frac{1}{2}\partial _{xx}^{2}u(t,\phi )+%
\bar{F}(t,u(t,\phi ),\partial _{x}u\left( t,\phi \right) ,\left( u\left(
\cdot ,\phi \right) \right) _{t})=0,\quad \phi \in \mathbb{\Lambda },\;t\in
\lbrack 0,T),\medskip \\
\displaystyle u(T,\phi )=\bar{h}(\phi ,\left( u\left( \cdot ,\phi \right)
\right) ),\quad \phi \in \mathbb{\Lambda }.%
\end{array}%
\right.
\end{equation*}%
%
%
%
%
%
%
%
%
%
%
%
%
%
%
%
%
%
%
%
%
%
%
%
%
%
%
%
%
%
%

An example of this type has been developed first in \cite{cv-ma/96} and then
treated by many authors, see, e.g. \cite{ka-pe-qu/01}, where they considered
a case where the drift and the diffusion component of the price equation
depend both on the wealth process $X$ and the underlying process $S$. This
would lead to a fully coupled forward--backward system which does not fit in
our setting. On the other hand in our model we have assumed that the drift $%
\mu$ and the interest rate $r$ may be influenced also by past values of the
wealth process $X$, whereas in cited papers no delay in the backward
component is assumed. %

\subsection{Risk measures via g-expectations}

A key problem in financial mathematics is the risk management of an
investment. Such a problem has been widely studied in finance since the
introductory paper \cite{ar-de-eb/02} where the notion of \textit{risk
measure} has been first introduced. Since then, several empirical studies
concerning the key task of risk-management have been conducted, showing in
particular that the best way to quantify the risk of a given financial
position should be a \textit{dynamic risk measure}, rather than a classic
static one. Starting from this fact, the notion of \textit{g-expectation}
has been first introduced in \cite{pe/97}, as a fundamental mathematical
tool when dealing with \textit{dynamic risk measures}; we refer also to \cite%
{pe/04,ro/06} for a comprehensive and exhaustive introduction to \textit{%
dynamic risk measures}.

The main purpose of a risk measure is to quantify in a single number the
riskiness of a given financial position. The next one is the mathematical
formulation of the notion of \textit{risk measure}, see, e.g. \cite[%
Definition 13.1.1]{de/13}.

\begin{definition}
A family $\left( \rho _{t}\right) _{t\in \lbrack 0,T]}$ of mappings $\rho
_{t}:L^{2}(\Omega ,\mathcal{F}_{T})\rightarrow L^{2}(\Omega ,\mathcal{F}%
_{t}) $, such that $\rho _{T}(\xi )=-\xi $ is called a \textit{dynamic risk
measure}.
\end{definition}

From a practical point of view, if we denote by $\xi $ the terminal value of
a given financial position, $\rho _{t}(\xi )$ quantifies the risk the
investor takes in the position $\xi $ at terminal time $T$. Clearly, in
order to have a concrete financial use, a risk measure has to satisfy a set
of properties, usually referred to as \textit{axioms of risk measures}; we
refer to \cite{de/13} to a complete list of the aforementioned axioms.

From a mathematical point of view, it has been shown that BSDEs and the
related for\-ward--backward system are a perfect tool to tackle the problem
of risk management. In particular, one possible way to define a \textit{%
dynamic risk measure} is to specify the generator $g$ of the driving BSDE,
from here the name $g-$expectation, where the generator $g$ determines the
properties of the \textit{dynamic risk measure}. A direct approach to $g$%
-expectation is therefore to introduce a BSDE of the form
\begin{equation}
Y(t)=\xi +\int_{t}^{T}g\left( s,Y(s),Z(s)\right) ds-\int_{t}^{T}Z(s)dW(s)\,,
\label{BSDE 1_application}
\end{equation}%
where the generator $g$ is called the generator of a $g-$expectation. In
this sense we have
\begin{equation*}
\rho _{t}(\xi )=\mathbb{E}^{g}\left[ \left. \xi \right\vert \mathcal{F}_{t}%
\right] =Y(t)\,,
\end{equation*}%
where we have used the subscript $g$ to emphasize the role played by the
generator $g$. Heuristically speaking we have the relation
\begin{equation*}
\mathbb{E}^{g}\left[ \left. dY(t)\right\vert \mathcal{F}_{t}\right]
=-g\left( t,Y(t),Z(t)\right) dt\,,\quad 0\leq t\leq T\,,
\end{equation*}%
so that intuitively the coefficient $g$ reflects the agent's belief on the
expected change of risk.

Once we have chosen a risk measure $g$, such that it is financially
reasonable, we then solve the BSDE \eqref{BSDE 1_application} endowed with a
suitable final condition which represents the investor's wealth at terminal
time $T$, we refer to \cite{ba-ka/07, ro/06} for a detailed introduction to
the usage of BSDEs as \textit{dynamic risk measures}.

In the literature it has always been considered a generator $g$ that depends
only on the present value at time $t$ of the risk measure $Y(t)$ and its
variability $Z(t)$, but, as pointed out in \cite{de/12}, if we want to model
an investor preference we cannot leave aside the memory effect, that is it
is reasonable to assume that an investor makes his choices based also on
what happened on the past. In \cite{de/12} the author proposed to consider a
$g-$expectation which incorporates a disappointment effect through a BSDE of
moving average. In fact in the just mentioned work the author suggests that
when dealing with the investor's preferences, to consider Markovian systems
is restrictive since it is natural that an investor takes into account the
past history of a given investment when he is to make some choice. We refer
to \cite{de/12} for a short but exhaustive review of different economical
studies of how memory effect cannot be neglected when dealing with an
investor's choice. Regarding the case considered in \cite{de/12}, we make
the assumption that the investor has a \textit{short memory}, that is, in
making his choices he considers only what has happened in the recent past.

Let us take two bounded and Lipschitz functions $g_{1},g_{2}$ such that $%
g_{2}\left( 0\right) =0.$ We consider a $g-$expectation of the form $%
g(s,y,z)=\beta \,g_{1}\left( \bar{y}\right) \,g_{2}\left( z\right) $, where $%
\bar{y}$ is the time--average in a sufficiently small time interval and $%
\beta \in \mathbb{R}$ a given financial parameter (and we remark that, in
this case, the assumption $g(s,y,0)=0$ is satisfied). Hence we will deal
with a BSDE with delayed generator of the form
\begin{equation}
Y(t)=\xi +\frac{\beta }{\delta }\int_{t}^{T}g_{1}\Big(\int_{-\delta
}^{0}Y(s+r)dr\Big)\,g_{2}\left( Z\left( s\right) \right)
ds-\int_{t}^{T}Z(s)dW(s)\,,  \label{BSDE 2_application}
\end{equation}%
with $\xi $ the terminal payoff of the investment to be introduced soon and $%
\delta $ small enough such that inequality \eqref{condition_KT} is
well-posed.

Let us now assume that the financial market is composed by one risky asset $%
S $ and one riskless bond $B$. The generalization to $d$ risky assets can be
easily derived from the present case. We assume, in a complete generality,
that both the bond and the asset may exhibit delay. For the case of delay in
the forward component a more established theory exists, with existence and
uniqueness, as well as regularity results, see, e.g. \cite%
{ar-hu/07,fu-te/05, fu-te/10}.

We consider in what follows the delayed market model introduced in \cite%
{ch-yo/99}. In what concerns the stock price we assume that $S$ evolves
according to the following stochastic delay differential equation:%
\begin{equation}
\frac{dS(t)}{S(t)}=\mu (t,S)dt+\sigma (t,S)dW(t)\,,  \label{risky asset}
\end{equation}%
with $S_{0}=s_{0}\in \mathbb{R}\,,$ where $\mu ,\sigma :[0,T]\times \mathbb{%
\Lambda }\rightarrow \mathbb{R}$ are some given functions, where the
notation is introduced in Section \ref{Section 1}.

Let us assume that $\mu$ and $\sigma$ satisfy assumptions of type $\mathrm{(A%
}_{1}\mathrm{)}$--$\mathrm{(A}_{2}\mathrm{)}$, so that there exists a unique
solution of equation (\ref{risky asset}) (this is a consequence of Theorem %
\ref{theorem 0}).\medskip

Also we assume the investor subscribes a claim with terminal payoff $h:%
\mathbb{\Lambda }\rightarrow \mathbb{R}$ so that the BSDE
\eqref{BSDE
2_application} becomes
\begin{equation}
Y(t)=h\left( S_{T}\right) +\frac{\beta }{\delta }\int_{t}^{T}g_{1}\Big(%
\int_{-\delta }^{0}Y(s+r)dr\Big)\,g_{2}\left( Z\left( s\right) \right)
ds-\int_{t}^{T}Z(s)dW(s)\,,  \label{BSDE 2a_application}
\end{equation}%
where we assume $h$ to satisfy $\mathrm{(A}_{7}\mathrm{)}-(iv)$; also, let
us stress that the generator $g:L^{2}([-\delta ,0];\mathbb{R})\rightarrow
\mathbb{R}$ defined above satisfies assumptions $\mathrm{(A}_{6}\mathrm{),(A}%
_{7}\mathrm{)}$, see, e.g. Remark \ref{REM:DelGen}. We are naturally led to
consider the following forward-backward system with delay
\begin{equation}
\left\{
\begin{array}{l}
\displaystyle S^{t,\phi }(s)=\phi (t)+\int_{t}^{s}S^{t,\phi }(r)\mu
(t,S^{t,\phi })dr+\int_{t}^{s}S^{t,\phi }(r)\sigma (t,S^{t,\phi
})dW(r)\,,\quad s\in \lbrack t,T]\,,\medskip \\
\displaystyle S^{t,\phi }(s)=\phi (s)\,,\quad s\in \lbrack 0,t)\,,\medskip
\\
\displaystyle Y(s)=h(S^{t,\phi })+\frac{\beta }{\delta }\int_{s}^{T}g_{1}%
\Big(\int_{-\delta }^{0}Y^{t,\phi }(r+\theta )d\theta \Big)\,g_{2}(Z^{t,\phi
}\left( r\right) )\,dr\medskip \\
\displaystyle\quad \quad \quad \quad \quad \quad \quad \quad \quad \quad
\quad \quad \quad \quad \quad \quad \quad \quad -\int_{s}^{T}Z^{t,\phi
}(r)dW(r)\,,\quad s\in \lbrack t,T]\,,\medskip \\
\displaystyle Y^{t,\phi }(s)=Y^{s,\phi }(s)\,,\quad s\in \lbrack 0,t)\,,%
\end{array}%
\right.  \label{BSDE 3_application}
\end{equation}%
and by theorems \ref{theorem 0}--\ref{theorem 3}, the forward-backward
system \eqref{BSDE 3_application} admits a unique solution.

Let us also stress that the great majority of possible claims that can be
considered in finance satisfies the above assumptions on the terminal payoff
$h$; also we allow the option to possibly be \textit{path--dependent}, that
is its terminal value at time $T$ depends explicitly on past values assumed
by the asset $S$.

From the results of the previous sections, we have the following
characterization for the FBSDE \eqref{BSDE 3_application}. In particular we
obtain theorem \ref{Feynman-Kac formula} holds, so that, for every $\left(
t,\phi \right) \in \left[ 0,T\right] \times \mathbb{\Lambda },$%
\begin{equation*}
Y^{t,\phi }\left( s\right) =u(s,S^{t,\phi }),\quad \text{for all }s\in \left[
t,T\right] ,
\end{equation*}%
where $u\left( t,\phi \right) =Y^{t,\phi }\left( t\right) $ is a viscosity
solution of the following path--dependent PDE:%
\begin{equation*}
\left\{
\begin{array}{l}
\displaystyle\partial _{t}u(t,\phi )+\frac{1}{2}\,\phi ^{2}\,\sigma
^{2}(t,\phi )\,\partial _{xx}^{2}u(t,\phi )+\phi \,\mu (t,\phi )\,\partial
_{x}u(t,\phi )\medskip \\
\displaystyle\quad +\frac{\beta }{\delta }\,g_{1}\Big(\int_{-\delta
}^{0}u(t+r,\phi )dr\Big)\,g_{2}\big(\phi \,\sigma (t,\phi )\,\partial
_{x}u(t,\phi )\big)=0,\quad \phi \in \mathbb{\Lambda },\;t\in \lbrack
0,T),\medskip \\
\displaystyle u(T,\phi )=h\big(\phi \big),\quad \phi \in \mathbb{\Lambda }.%
\end{array}%
\right.
\end{equation*}

\section{Appendix: Proofs of Theorem \protect\ref{theorem 3} and Theorem
\protect\ref{theorem 2}}

\begin{proof}[Proof of Theorem \protect\ref{theorem 3}]
The existence and uniqueness will be obtained by the Banach fixed point
theorem. Let $\phi \in \mathbb{\Lambda }$ be arbitrarily fixed and let us
consider the map $\Gamma $ defined on $\mathcal{A}\times \mathcal{B}$, with $%
\mathcal{A}:=\mathcal{C}\big(\left[ 0,T\right] ;\mathcal{S}_{0}^{2,m}\big)$
and $\mathcal{B}:=\mathcal{C}\big(\left[ 0,T\right] ;\mathcal{H}%
_{0}^{2,m\times d^{\prime }}\big)$, in the following way: for $\left(
U,V\right) \in \mathcal{A}\times \mathcal{B}$, $\Gamma \left( U,V\right)
=\left( Y,Z\right) $, where for $t\in \lbrack 0,T]$, the couple of adapted
processes $\left( Y^{t}\left( s\right) ,Z^{t}\left( s\right) \right) _{s\in %
\left[ t,T\right] }$ is the unique solution of the BSDE%
\begin{equation}
Y^{t}\left( s\right) =h(X^{t,\phi })+\int_{s}^{T}F(r,X^{t,\phi },Y^{t}\left(
r\right) ,Z^{t}\left( r\right)
,U_{r}^{t},V_{r}^{t})dr-\int_{s}^{T}Z^{t}\left( r\right) dW\left( r\right)
\,,\quad s\in \lbrack t,T].  \label{BSDE iterative 1}
\end{equation}%
Now we prolong the solution by taking $Y^{t}\left( s\right) :=Y^{s}\left(
s\right) $ and $Z^{t}\left( s\right) :=0\,,\quad $for $s\in \lbrack 0,t)\,.$

\noindent\textbf{\textit{Step I.}}

Let us first show that $\Gamma$ takes values in the Banach space $\mathcal{A}%
\times\mathcal{B}$. For that, let us take $\left( U,V\right) \in\mathcal{A}%
\times\mathcal{B}$; we will prove that $(Y,Z):=\Gamma\left( U,V\right) \in%
\mathcal{A}\times\mathcal{B}$, i.e. for every $t\in\lbrack0,T]$ we have%
\begin{equation}
Y^{t}\in\mathcal{S}_{t}^{2,m}\subseteq\mathcal{S}_{0}^{2,m},\quad Z^{t}\in%
\mathcal{H}_{t}^{2,m\times d^{\prime}}\subseteq\mathcal{H}_{0}^{2,m\times
d^{\prime}}  \label{Gamma well defined 1}
\end{equation}
and the applications%
\begin{equation}
\begin{array}{l}
\left[ 0,T\right] \ni t\mapsto Y^{t}\in\mathcal{S}_{0}^{2,m},\medskip \\
\left[ 0,T\right] \ni t\mapsto Z^{t}\in\mathcal{H}_{0}^{2,m\times d^{\prime
}}%
\end{array}
\label{Gamma well defined 2}
\end{equation}
are continuous.

Let $t\in\left[ 0,T\right] $ be fixed and $t^{\prime}\in\lbrack0,T]$; with
no loss of generality, we will suppose that $t<t^{\prime}$ and $t^{\prime
}-t<\delta$.

We have, using (\ref{BSDE iterative 1}),%
\begin{equation*}
\begin{array}{l}
\mathbb{E}\big(\sup_{s\in \left[ 0,T\right] }|Y^{t}\left( s\right)
-Y^{t^{\prime }}\left( s\right) |^{2}\big)\medskip \\
\leq \mathbb{E}\big(\sup_{s\in \left[ 0,t^{\prime }\right] }|Y^{t}\left(
s\right) -Y^{t^{\prime }}\left( s\right) |^{2}\big)+\mathbb{E}\big(%
\sup_{s\in \left[ t^{\prime },T\right] }|Y^{t}\left( s\right) -Y^{t^{\prime
}}\left( s\right) |^{2}\big)\medskip \\
\leq 2\mathbb{E}\big(\sup_{s\in \left[ t,t^{\prime }\right] }|Y^{t}\left(
s\right) -Y^{t}\left( t\right) |^{2}\big)+2\mathbb{E}\big(\sup_{s\in \left[
t,t^{\prime }\right] }|Y^{t}\left( t\right) -Y^{s}\left( s\right) |^{2}\big)%
\medskip \\
\quad +\mathbb{E}\big(\sup_{s\in \left[ t^{\prime },T\right] }|Y^{t}\left(
s\right) -Y^{t^{\prime }}\left( s\right) |^{2}\big).%
\end{array}%
\end{equation*}

From the continuity of the solution of equation (\ref{BSDE iterative 1})
with respect to time, we have
\begin{equation*}
\mathbb{E}\big(\sup_{s\in \left[ t,t^{\prime }\right] }|Y^{t}\left( s\right)
-Y^{t}\left( t\right) |^{2}\big)\rightarrow 0\,,
\end{equation*}%
as $t^{\prime }\rightarrow t$.

Concerning the term $\mathbb{E}\big(\sup_{s\in \left[ t^{\prime },T\right]
}|Y^{t}\left( s\right) -Y^{t^{\prime }}\left( s\right) |^{2}\big)$ let us
denote for short, only throughout this step,%
\begin{equation*}
\begin{array}{l}
\Delta Y\left( r\right) :=Y^{t}\left( r\right) -Y^{t^{\prime }}\left(
r\right) ,\quad \Delta Z\left( r\right) :=Z^{t}\left( r\right) -Z^{t^{\prime
}}\left( r\right) \medskip \\
\Delta U\left( r\right) :=U^{t}\left( r\right) -U^{t^{\prime }}\left(
r\right) ,\quad \Delta V\left( r\right) :=V^{t}\left( r\right) -V^{t^{\prime
}}\left( r\right)%
\end{array}%
\end{equation*}%
and%
\begin{equation*}
\begin{array}{l}
\Delta h:=h(X^{t,\phi })-h(X^{t^{\prime },\phi }),\medskip \\
\Delta F\left( r\right) :=F(r,X^{t,\phi },Y^{t}\left( r\right) ,Z^{t}\left(
r\right) ,U_{r}^{t},V_{r}^{t})-F(r,X^{t^{\prime },\phi },Y^{t}\left(
r\right) ,Z^{t}\left( r\right) ,U_{r}^{t},V_{r}^{t}).%
\end{array}%
\end{equation*}%
Exploiting It\^{o}'s formula we have, for any $\beta >0$ and any $s\in \left[
t^{\prime },T\right] ,$%
\begin{equation*}
\begin{array}{l}
\displaystyle e^{\beta s}|\Delta Y\left( s\right) |^{2}+\beta
\int_{s}^{T}e^{\beta r}|\Delta Y\left( r\right) |^{2}dr+\int_{s}^{T}e^{\beta
r}|\Delta Z\left( r\right) |^{2}dr\medskip \\
\displaystyle=e^{\beta T}|\Delta Y\left( T\right)
|^{2}-2\int_{s}^{T}e^{\beta r}\langle \Delta Y\left( r\right) ,\Delta
Z\left( r\right) \rangle dW\left( r\right) \medskip \\
\displaystyle\;\;+2\int_{s}^{T}e^{\beta r}\langle \Delta Y\left( r\right)
,F(r,X^{t,\phi },Y^{t}\left( r\right) ,Z^{t}\left( r\right)
,U_{r}^{t},V_{r}^{t})-F(r,X^{t^{\prime },\phi },Y^{t^{\prime }}\left(
r\right) ,Z^{t^{\prime }}\left( r\right) ,U_{r}^{t^{\prime
}},V_{r}^{t^{\prime }})\rangle dr\,.%
\end{array}%
\end{equation*}%
From assumptions $\mathrm{(A}_{3}\mathrm{)}$--$\mathrm{(A}_{5}\mathrm{)}$,
and noting that it holds%
\begin{equation*}
\begin{array}{l}
\displaystyle\int_{s}^{T}e^{\beta r}\big(\int_{-\delta }^{0}\left( |\Delta
U\left( r+\theta \right) |^{2}+|\Delta V\left( r+\theta \right) |^{2}\right)
\alpha (d\theta )\big)dr\medskip \\
\displaystyle=\int_{-\delta }^{0}\left[ \int_{s}^{T}e^{\beta r}\left(
|\Delta U\left( r+\theta \right) |^{2}+|\Delta V\left( r+\theta \right)
|^{2}\right) dr\right] \alpha (d\theta )\medskip \\
\displaystyle=\int_{-\delta }^{0}\big(\int_{s+\theta }^{T+\theta }e^{\beta
(r-\theta )}\left( |\Delta U\left( r\right) |^{2}+|\Delta V\left( r\right)
|^{2}\right) dr\big)\alpha (d\theta )\medskip \\
\displaystyle\leq e^{\beta \delta }\cdot \int_{-\delta }^{0}\alpha (d\theta
)\cdot \int_{0}^{T}e^{\beta r}\big(|\Delta U\left( r\right) |^{2}+|\Delta
V\left( r\right) |^{2}\big)dr\medskip \\
\displaystyle\leq Te^{\beta \delta }\sup_{r\in \left[ 0,T\right] }\big(%
e^{\beta r}|\Delta U\left( r\right) |^{2}\big)+e^{\beta \delta
}\int_{0}^{T}e^{\beta r}|\Delta V\left( r\right) |^{2}dr\,,%
\end{array}%
\end{equation*}
we have for any $a>0,$%
\begin{equation*}
\begin{array}{l}
\displaystyle2\int_{s}^{T}e^{\beta r}\langle \Delta Y\left( r\right)
,F(r,X^{t,\phi },Y^{t}\left( r\right) ,Z^{t}\left( r\right)
,U_{r}^{t},V_{r}^{t})-F(r,X^{t^{\prime },\phi },Y^{t^{\prime }}\left(
r\right) ,Z^{t^{\prime }}\left( r\right) ,U_{r}^{t^{\prime
}},V_{r}^{t^{\prime }})\rangle dr\medskip \\
\displaystyle\leq a\int_{s}^{T}e^{\beta r}|\Delta Y\left( r\right) |^{2}dr+%
\frac{3}{a}\int_{s}^{T}e^{\beta r}|\Delta F\left( r\right) |^{2}dr+\frac{%
6L^{2}}{a}\int_{s}^{T}e^{\beta r}\left( |\Delta Y\left( r\right)
|^{2}+|\Delta Z\left( r\right) |^{2}\right) dr\medskip \\
\displaystyle\quad +\frac{3TKe^{\beta \delta }}{a}\sup_{r\in \left[ 0,T%
\right] }\big(e^{\beta r}|\Delta U\left( r\right) |^{2}\big)+\frac{%
3Ke^{\beta \delta }}{a}\int_{0}^{T}e^{\beta r}|\Delta V\left( r\right)
|^{2}dr\,.%
\end{array}%
\end{equation*}

Therefore we have
\begin{equation*}
\begin{array}{l}
\displaystyle e^{\beta s}|\Delta Y\left( s\right) |^{2}+\left( \beta -a-%
\frac{6L^{2}}{a}\right) \int_{s}^{T}e^{\beta r}|\Delta Y\left( r\right)
|^{2}dr+\left( 1-\frac{6L^{2}}{a}\right) \int_{s}^{T}e^{\beta r}|\Delta
Z\left( r\right) |^{2}dr\medskip \\
\displaystyle\leq e^{\beta T}|\Delta Y\left( T\right) |^{2}+\frac{3}{a}%
\int_{s}^{T}e^{\beta r}|\Delta F\left( r\right)
|^{2}dr-2\int_{s}^{T}e^{\beta r}\langle \Delta Y\left( r\right) ,\Delta
Z\left( r\right) \rangle dW\left( r\right) \medskip \\
\displaystyle\quad +\frac{3TKe^{\beta \delta }}{a}\sup_{r\in \left[ 0,T%
\right] }e^{\beta r}|\Delta U\left( r\right) |^{2}+\frac{3Ke^{\beta \delta }%
}{a}\int_{0}^{T}e^{\beta r}|\Delta V\left( r\right) |^{2}dr.%
\end{array}%
\end{equation*}%
We now choose $\beta ,a>0$ such that%
\begin{equation}
a+\frac{6L^{2}}{a}<\beta \quad \text{and}\quad \frac{6L^{2}}{a}<1,
\label{restriction_parameters}
\end{equation}%
hence we obtain%
\begin{equation}
\begin{array}{l}
\displaystyle\left( 1-\frac{6L^{2}}{a}\right) \mathbb{E}\int_{s}^{T}e^{\beta
r}|\Delta Z\left( r\right) |^{2}dr\leq \mathbb{E}\big(e^{\beta T}|\Delta
h|^{2}\big)+\frac{3}{a}\mathbb{E}\int_{s}^{T}e^{\beta r}|\Delta F\left(
r\right) |^{2}dr\medskip \\
\displaystyle\quad +\frac{3TKe^{\beta \delta }}{a}\mathbb{E}\big(\sup_{r\in %
\left[ 0,T\right] }e^{\beta r}|\Delta U\left( r\right) |^{2}\big)+\frac{%
3Ke^{\beta \delta }}{a}\mathbb{E}\int_{0}^{T}e^{\beta r}|\Delta V\left(
r\right) |^{2}dr.%
\end{array}
\label{technical ineq 14}
\end{equation}%
and, exploiting Burkholder--Davis--Gundy's inequality, we have%
\begin{equation*}
\begin{array}{l}
\displaystyle2\mathbb{E}\big(\sup_{s\in \left[ t^{\prime },T\right] }\Big|%
\int_{s}^{T}e^{\beta r}\langle \Delta Y\left( r\right) ,\Delta Z\left(
r\right) \rangle dW\left( r\right) \Big|\big)\medskip \\
\displaystyle\leq \frac{1}{4}\mathbb{E}\big(\sup_{s\in \left[ t^{\prime },T%
\right] }e^{\beta s}|\Delta Y\left( s\right) |^{2}\big)+144\mathbb{E}%
\int_{t^{\prime }}^{T}e^{\beta r}|\Delta Z\left( r\right) |^{2}dr.%
\end{array}%
\end{equation*}%
which immediately implies
\begin{equation*}
\begin{array}{l}
\displaystyle\frac{3}{4}\mathbb{E}\big(\sup_{s\in \left[ t^{\prime },T\right]
}e^{\beta s}|\Delta Y\left( s\right) |^{2}\big)\leq \mathbb{E}\big(e^{\beta
T}|\Delta h|^{2}\big)+\frac{3}{a}\mathbb{E}\int_{t^{\prime }}^{T}e^{\beta
r}|\Delta F\left( r\right) |^{2}dr\medskip \\
\displaystyle\quad +\frac{3TKe^{\beta \delta }}{a}\mathbb{E}\big(\sup_{r\in %
\left[ 0,T\right] }e^{\beta r}|\Delta U\left( r\right) |^{2}\big)+\frac{%
3Ke^{\beta \delta }}{a}\mathbb{E}\int_{0}^{T}e^{\beta r}|\Delta V\left(
r\right) |^{2}dr+144\mathbb{E}\int_{t^{\prime }}^{T}|\Delta Z\left( r\right)
|^{2}dr.%
\end{array}%
\end{equation*}%
Hence, we have%
\begin{equation}
\begin{array}{l}
\displaystyle\frac{3}{4}\mathbb{E}\big(\sup_{s\in \left[ t^{\prime },T\right]
}e^{\beta s}|\Delta Y\left( s\right) |^{2}\big)\leq \mathbb{E}\big(e^{\beta
T}|\Delta h|^{2}\big)+\frac{3}{a}C_{1}\mathbb{E}\int_{t^{\prime
}}^{T}e^{\beta r}|\Delta F\left( r\right) |^{2}dr\medskip \\
\displaystyle\quad +\frac{3TKe^{\beta \delta }C_{1}}{a}\,\mathbb{E}\big(%
\sup_{r\in \left[ 0,T\right] }e^{\beta r}|\Delta U\left( r\right) |^{2}\big)%
\medskip \\
\displaystyle\quad +\frac{3Ke^{\beta \delta }C_{1}}{a}\,\mathbb{E}%
\int_{0}^{T}e^{\beta r}|\Delta V\left( r\right) |^{2}dr,%
\end{array}
\label{technical ineq 15}
\end{equation}%
where%
\begin{equation*}
C_{1}:=1+\frac{144}{1-6L^{2}/a}\,.
\end{equation*}%
Exploiting thus assumptions $\mathrm{(A}_{3}\mathrm{)}$ and $\mathrm{(A}_{5}%
\mathrm{)}$ together with the fact that $X^{\cdot ,\phi }$ is continuous and
bounded, we have%
\begin{equation*}
C_{1}\mathbb{E}\big(e^{\beta T}|\Delta h|^{2}\big)+\frac{3}{a}C_{1}\mathbb{E}%
\int_{t^{\prime }}^{T}e^{\beta r}|\Delta F\left( r\right) |^{2}dr\rightarrow
0\quad \text{as }t^{\prime }\rightarrow t.
\end{equation*}%
Since $\left( U,V\right) \in \mathcal{A}\times \mathcal{B}$, and therefore
we have
\begin{equation*}
\mathbb{E}\big(\sup_{r\in \left[ 0,T\right] }e^{\beta r}|\Delta U\left(
r\right) |^{2}\big)\,\rightarrow 0\quad \text{and}\quad \mathbb{E}%
\int_{0}^{T}e^{\beta r}|\Delta V\left( r\right) |^{2}dr\,\rightarrow 0\,,
\end{equation*}%
as $t^{\prime }\rightarrow t$, we have%
\begin{equation}
\mathbb{E}\big(\sup_{s\in \left[ t^{\prime },T\right] }e^{\beta s}|\Delta
Y\left( s\right) |^{2}\big)\rightarrow 0\quad \text{and}\quad \mathbb{E}%
\int_{t^{\prime }}^{T}e^{\beta r}|\Delta Z\left( r\right) |^{2}dr\rightarrow
0,\quad \text{as }t^{\prime }\rightarrow t.  \label{technical ineq 16}
\end{equation}

We are left to show that the term $\mathbb{E}\big(\sup_{s\in\left[
t,t^{\prime}\right] }|Y^{t}\left( t\right) -Y^{s}\left( s\right) |^{2}\big)$
is also converging to $0$ as $t^{\prime}\rightarrow t$.

Since the map $t\mapsto Y^{t}\left( t\right) $ is deterministic, we have
from equation (\ref{BSDE iterative 1}),%
\begin{equation*}
\begin{array}{l}
\displaystyle Y^{t}\left( t\right) -Y^{s}\left( s\right) =\mathbb{E}\big[%
Y^{t}\left( t\right) -Y^{s}\left( s\right) \big]\medskip \\
\displaystyle=\mathbb{E}\big[h(X^{t,\phi })-h(X^{s,\phi })\big]+\mathbb{E}%
\int_{t}^{T}F(r,X^{t,\phi },Y^{t}\left( r\right) ,Z^{t}\left( r\right)
,U_{r}^{t},V_{r}^{t})dr\medskip \\
\displaystyle\quad -\mathbb{E}\int_{s}^{T}F(r,X^{s,\phi },Y^{s}\left(
r\right) ,Z^{s}\left( r\right) ,U_{r}^{s},V_{r}^{s})dr\medskip \\
\displaystyle=\mathbb{E}\big[h(X^{t,\phi })-h(X^{s,\phi })\big]+\mathbb{E}%
\int_{t}^{s}F(r,X^{t,\phi },Y^{t}\left( r\right) ,Z^{t}\left( r\right)
,U_{r}^{t},V_{r}^{t})dr\medskip \\
\displaystyle\quad +\mathbb{E}\int_{s}^{T}\big[F(r,X^{t,\phi },Y^{t}\left(
r\right) ,Z^{t}\left( r\right) ,U_{r}^{t},V_{r}^{t})-F(r,X^{s,\phi
},Y^{s}\left( r\right) ,Z^{s}\left( r\right) ,U_{r}^{s},V_{r}^{s})\big]dr.%
\end{array}%
\end{equation*}%
Using then assumption \textrm{(A}$_{3}$\textrm{)} we have%
\begin{equation*}
\begin{array}{l}
\displaystyle|Y^{t}\left( t\right) -Y^{s}\left( s\right) |\leq \mathbb{E}%
\big|h(X^{t,\phi })-h(X^{s,\phi })\big|+\mathbb{E}\int_{t}^{s}L\big(%
|Y^{t}\left( r\right) |+|Z^{t}\left( r\right) |\big)dr\medskip \\
\displaystyle+\sqrt{K\int_{t}^{s}\mathbb{E}\big(\int_{-\delta }^{0}\left(
|U^{t}\left( r+\theta \right) |^{2}+|V^{t}\left( r+\theta \right)
|^{2}\right) \alpha (d\theta )\big)dr}\cdot \sqrt{s-t}\medskip \\
\displaystyle\quad +\mathbb{E}\int_{t}^{s}\big|F(r,X^{t,\phi },0,0,0,0)\big|%
dr\medskip \\
\displaystyle+\mathbb{E}\int_{s}^{T}\big|F(r,X^{t,\phi },Y^{t}\left(
r\right) ,Z^{t}\left( r\right) ,U_{r}^{t},V_{r}^{t})-F(r,X^{s,\phi
},Y^{t}\left( r\right) ,Z^{t}\left( r\right) ,U_{r}^{t},V_{r}^{t})\big|%
dr\medskip \\
\displaystyle+\mathbb{E}\int_{s}^{T}L\big(|Y^{t}\left( r\right) -Y^{s}\left(
r\right) |+|Z^{t}\left( r\right) -Z^{s}\left( r\right) |\big)dr\medskip \\
\displaystyle+\sqrt{K(T-s)\int_{s}^{T}\mathbb{E}\big(\int_{-\delta
}^{0}\left( |U^{t}\left( r+\theta \right) -U^{s}\left( r+\theta \right)
|^{2}+|V^{t}\left( r+\theta \right) -V^{s}\left( r+\theta \right)
|^{2}\right) \alpha (d\theta )\big)dr}%
\end{array}%
\end{equation*}%
and therefore we obtain%
\begin{equation*}
\begin{array}{l}
\displaystyle|Y^{t}\left( t\right) -Y^{s}\left( s\right) |\medskip \\
\displaystyle\leq \mathbb{E}\big|h(X^{t,\phi })-h(X^{s,\phi })\big|+L\sqrt{%
s-t}\sqrt{T\mathbb{E}\sup_{r\in \left[ 0,T\right] }|Y^{t}\left( r\right)
|^{2}+\mathbb{E}\int_{0}^{T}|Z^{t}\left( r\right) |^{2}dr}\medskip \\
\displaystyle\quad +\sqrt{K}\sqrt{s-t}\sqrt{T\mathbb{E}\sup_{r\in \left[ 0,T%
\right] }|U^{t}\left( r\right) |^{2}+\mathbb{E}\int_{0}^{T}|V^{t}\left(
r\right) |^{2}dr}+\left( s-t\right) M(1+\mathbb{E}||X^{t,\phi
}||_{T}^{p})\medskip \\
\displaystyle\quad +\mathbb{E}\int_{s}^{T}\big|F(r,X^{t,\phi },Y^{t}\left(
r\right) ,Z^{t}\left( r\right) ,U_{r}^{t},V_{r}^{t})-F(r,X^{s,\phi
},Y^{t}\left( r\right) ,Z^{t}\left( r\right) ,U_{r}^{t},V_{r}^{t})\big|%
dr\medskip \\
\displaystyle\quad +L\sqrt{T-s}\sqrt{T\mathbb{E}\sup_{r\in \left[ s,T\right]
}|Y^{t}\left( r\right) -Y^{s}\left( r\right) |^{2}+\mathbb{E}%
\int_{s}^{T}|Z^{t}\left( r\right) -Z^{s}\left( r\right) |^{2}dr}\medskip \\
\displaystyle\quad +\sqrt{K}\sqrt{T-s}\sqrt{T\mathbb{E}\sup_{r\in \left[ 0,T%
\right] }|U^{t}\left( r\right) -U^{s}\left( r\right) |^{2}+\mathbb{E}%
\int_{0}^{T}|V^{t}\left( r\right) -V^{s}\left( r\right) |^{2}dr}\,.%
\end{array}%
\end{equation*}%
Taking again into account the fact that $\left( U,V\right) \in \mathcal{A}%
\times \mathcal{B}$, previous step and assumptions $\mathrm{(A}_{3}\mathrm{)}
$ and $\mathrm{(A}_{5}\mathrm{)}$, we infer that
\begin{equation}
\mathbb{E}\big(\sup_{s\in \left[ t,t^{\prime }\right] }|Y^{t}\left( t\right)
-Y^{s}\left( s\right) |\big)\rightarrow 0,\quad \text{as }t^{\prime
}\rightarrow t.  \label{technical ineq 17}
\end{equation}

Concerning the term $\mathbb{E}\int_{0}^{T}|Z^{t}\left( r\right)
-Z^{t^{\prime}}\left( r\right) |^{2}dr$, we see that%
\begin{equation*}
\begin{array}{l}
\displaystyle\mathbb{E}\int_{0}^{T}|Z^{t}\left( r\right)
-Z^{t^{\prime}}\left( r\right) |^{2}dr=\mathbb{E}\int_{0}^{t^{\prime}}|Z^{t}%
\left( r\right) -Z^{t^{\prime}}\left( r\right) |^{2}dr+\mathbb{E}%
\int_{t^{\prime}}^{T}|Z^{t}\left( r\right) -Z^{t^{\prime}}\left( r\right)
|^{2}dr\medskip \\
\displaystyle=\mathbb{E}\int_{t}^{t^{\prime}}|Z^{t}\left( r\right) |^{2}dr+%
\mathbb{E}\int_{t^{\prime}}^{T}|Z^{t}\left( r\right) -Z^{t^{\prime}}\left(
r\right) |^{2}dr\,,%
\end{array}%
\end{equation*}
hence, by (\ref{technical ineq 14}),
\begin{equation}
\mathbb{E}\int_{0}^{T}|Z^{t}\left( r\right) -Z^{t^{\prime}}\left( r\right)
|^{2}dr\rightarrow0,\quad\text{as }t^{\prime}\rightarrow t.
\label{technical ineq 18}
\end{equation}

\noindent\textbf{\textit{Step II.}}

We are now to prove that $\Gamma$ is a contraction on the space $\mathcal{A}%
\times\mathcal{B}$ with respect to the norms%
\begin{equation*}
|||\left( Y,Z\right) |||_{\mathcal{A}\times\mathcal{B}}:=\left(
|||Y|||_{1}^{2}+|||Z|||_{2}^{2}\right) ^{1/2}\,,
\end{equation*}
where%
\begin{equation*}
\begin{array}{l}
\displaystyle|||Y|||_{1}^{2}:=\sup_{t\in\left[ 0,T\right] }\mathbb{E}\big(%
\sup_{r\in\left[ 0,T\right] }e^{\beta r}|Y^{t}\left( r\right) |^{2}\big)%
,\medskip \\
\displaystyle|||Z|||_{2}^{2}:=\sup_{t\in\left[ 0,T\right] }\mathbb{E}%
\int_{0}^{T}e^{\beta r}|Z^{t}\left( r\right) |^{2}dr.%
\end{array}%
\end{equation*}

Let us recall that $\Gamma:\mathcal{A}\times\mathcal{B}\rightarrow \mathcal{A%
}\times\mathcal{B}$ is defined by $\Gamma\left( U,V\right) =\left(
Y,Z\right) $, where $\left( Y,Z\right) $ is the solution of the BSDE (\ref%
{BSDE iterative 1}).

Let us consider $\left( U^{1},V^{1}\right) ,\left( U^{2},V^{2}\right) \in%
\mathcal{A}\times\mathcal{B}$ and $\left( Y^{1},Z^{1}\right) :=\Gamma\left(
U^{1},V^{1}\right) $, $\left( Y^{2},Z^{2}\right) :=\Gamma\left(
U^{2},V^{2}\right) $. For the sake of brevity, we will denote in what follows%
\begin{equation*}
\begin{array}{l}
\Delta F^{t}\left( r\right) :=F(r,X^{t,\phi},Y^{1,t}\left( r\right)
,Z^{1,t}\left( r\right) ,U_{r}^{1,t},V_{r}^{1,t})\medskip \\
\quad\quad\quad\quad-F(r,X^{t,\phi},Y^{2,t}\left( r\right) ,Z^{2,t}\left(
r\right) ,U_{r}^{2,t},V_{r}^{2,t}),\medskip \\
\Delta U^{t}\left( r\right) :=U^{1,t}\left( r\right) -U^{2,t}\left( r\right)
,\quad\Delta V^{t}\left( r\right) :=V^{1,t}\left( r\right) -V^{2,t}\left(
r\right) ,\medskip \\
\Delta Y^{t}\left( r\right) :=Y^{1,t}\left( r\right) -Y^{2,t}\left( r\right)
,\quad\Delta Z^{t}\left( r\right) :=Z^{1,t}\left( r\right) -Z^{2,t}\left(
r\right) .%
\end{array}%
\end{equation*}

Proceeding as in \textbf{\textit{Step I}}, we have from It\^{o}'s formula,
for any $s\in \left[ t,T\right] $ and $\beta >0$,%
\begin{equation}
\begin{array}{l}
\displaystyle e^{\beta s}|\Delta Y^{t}\left( s\right) |^{2}+\beta
\int_{s}^{T}e^{\beta r}|\Delta Y^{t}\left( r\right)
|^{2}dr+\int_{s}^{T}e^{\beta r}|\Delta Z^{t}\left( r\right) |^{2}dr\medskip
\\
\displaystyle=2\int_{s}^{T}e^{\beta r}\langle \Delta Y^{t}\left( r\right)
,\Delta F^{t}\left( r\right) \rangle dr-2\int_{s}^{T}e^{\beta r}\langle
\Delta Y^{t}\left( r\right) ,\Delta Z^{t}\left( r\right) \rangle dW\left(
r\right) .%
\end{array}
\label{technical ineq 1}
\end{equation}

Noticing that it holds
\begin{equation*}
\begin{array}{l}
\displaystyle\frac{2K}{a}\int_{s}^{T}e^{\beta r}\Big(\int_{-\delta}^{0}\big(%
\left\vert \Delta U^{t}(r+\theta)\right\vert ^{2}+\left\vert \Delta
V^{t}(r+\theta)\right\vert ^{2}\big)\alpha(d\theta)\Big)dr\medskip \\
\displaystyle\leq\frac{2K}{a}\int_{-\delta}^{0}\Big(\int_{s}^{T}e^{\beta r}%
\big(\left\vert \Delta U^{t}(r+\theta)\right\vert ^{2}+\left\vert \Delta
V^{t}(r+\theta)\right\vert ^{2}\big)dr\Big)\alpha(d\theta)\medskip \\
\displaystyle\leq\frac{2K}{a}\int_{-\delta}^{0}\Big(\int_{s+r}^{T+r}e^{\beta%
\left( r^{\prime}-\theta\right) }\big(\left\vert \Delta
U^{t}(r^{\prime})\right\vert ^{2}+\left\vert \Delta V^{t}(r^{\prime
})\right\vert ^{2}\big)dr^{\prime}\Big)\alpha(d\theta)\medskip \\
\displaystyle\leq\frac{2K}{a}\int_{-\delta}^{0}e^{-\beta\theta}\alpha
(d\theta)\cdot\int_{s-\delta}^{T}e^{\beta r}\big(\left\vert \Delta
U^{t}(r)\right\vert ^{2}+\left\vert \Delta V^{t}(r)\right\vert ^{2}\big)%
dr\medskip \\
\displaystyle\leq\frac{2Ke^{\beta\delta}}{a}\int_{s-\delta}^{T}e^{\beta r}%
\big(\left\vert \Delta U^{t}(r)\right\vert ^{2}+\left\vert \Delta
V^{t}(r)\right\vert ^{2}\big)dr.%
\end{array}%
\end{equation*}
we immediately have, from assumptions $\mathrm{(A}_{3}\mathrm{)}$--$\mathrm{%
(A}_{5}\mathrm{)}$, that for any $a>0$,
\begin{equation}
\begin{array}{l}
\displaystyle2\Big|\int_{s}^{T}e^{\beta r}\langle\Delta Y^{t}\left( r\right)
,\Delta F^{t}\left( r\right) \rangle dr\Big|\leq2\int_{s}^{T}e^{\beta
r}|\langle\Delta Y^{t}\left( r\right) ,\Delta F^{t}\left( r\right)
\rangle|dr\medskip \\
\displaystyle\leq a\int_{s}^{T}e^{\beta r}|\Delta Y^{t}\left( r\right) |^{2}+%
\frac{1}{a}\int_{s}^{T}e^{\beta r}|\Delta F^{t}\left( r\right)
|^{2}dr\medskip \\
\displaystyle\leq a\int_{s}^{T}e^{\beta r}|\Delta Y^{t}\left( r\right) |^{2}+%
\frac{2}{a}\int_{s}^{T}e^{\beta r}L^{2}\left( |\Delta Y^{t}\left( r\right)
|+|\Delta Z^{t}\left( r\right) |\right) ^{2}dr\medskip \\
\displaystyle\quad+\frac{2}{a}\int_{s}^{T}e^{\beta r}\Big(K\int_{-\delta}^{0}%
\big(\left\vert \Delta U^{t}(r+\theta)\right\vert ^{2}+\left\vert \Delta
V^{t}(r+\theta)\right\vert ^{2}\big)\alpha(d\theta)\Big)dr\medskip \\
\displaystyle\leq a\int_{s}^{T}e^{\beta r}|\Delta Y^{t}\left( r\right) |^{2}+%
\frac{4L^{2}}{a}\int_{s}^{T}e^{\beta r}\big(|\Delta Y^{t}\left( r\right)
|^{2}+|\Delta Z^{t}\left( r\right) |^{2}\big)dr\medskip \\
\displaystyle\quad+\frac{2Ke^{\beta\delta}}{a}\int_{s-\delta}^{T}e^{\beta r}%
\big(\left\vert \Delta U^{t}(r)\right\vert ^{2}+\left\vert \Delta
V^{t}(r)\right\vert ^{2}\big)dr\,.%
\end{array}
\label{technical ineq 2}
\end{equation}

Therefore equation (\ref{technical ineq 1}) yields%
\begin{equation}
\begin{array}{l}
\displaystyle e^{\beta s}|\Delta Y^{t}\left( s\right) |^{2}+\left( \beta -a-%
\frac{4L^{2}}{a}\right) \int_{s}^{T}e^{\beta r}|\Delta Y^{t}\left( r\right)
|^{2}dr+\left( 1-\frac{4L^{2}}{a}\right) \int_{s}^{T}e^{\beta r}|\Delta
Z^{t}\left( r\right) |^{2}dr\medskip \\
\displaystyle\leq \frac{2Ke^{\beta \delta }}{a}T\sup_{r\in \left[ 0,T\right]
}e^{\beta r}|\Delta U^{t}\left( r\right) |^{2}+\frac{2Ke^{\beta \delta }}{a}%
\int_{0}^{T}e^{\beta r}|\Delta V^{t}\left( r\right) |^{2}dr\medskip \\
\displaystyle\quad -2\int_{s}^{T}e^{\beta r}\langle \Delta Y^{t}\left(
r\right) ,\Delta Z^{t}\left( r\right) \rangle dW\left( r\right) .%
\end{array}
\label{technical ineq 19}
\end{equation}%
Let now $\beta ,a>0$ satisfying%
\begin{equation}
\beta >a+\frac{4L^{2}}{a}\quad \text{and}\quad 1>\frac{4L^{2}}{a}\,,
\label{restriction_parameters2}
\end{equation}%
we have%
\begin{equation}
\begin{array}{l}
\displaystyle\left( 1-\frac{4L^{2}}{a}\right) \mathbb{E}\int_{s}^{T}e^{\beta
r}|\Delta Z^{t}\left( r\right) |^{2}dr\medskip \\
\displaystyle\leq \frac{2TKe^{\beta \delta }}{a}\mathbb{E}\big(\sup_{r\in %
\left[ 0,T\right] }e^{\beta r}|\Delta U^{t}\left( r\right) |^{2}\big)+\frac{%
2Ke^{\beta \delta }}{a}\mathbb{E}\int_{0}^{T}e^{\beta r}|\Delta V^{t}\left(
r\right) |^{2}dr\,.%
\end{array}
\label{technical ineq 3}
\end{equation}

Exploiting now Burkholder--Davis--Gundy's inequality, we have%
\begin{equation*}
\begin{array}{l}
\displaystyle2\mathbb{E}\Big[\sup_{s\in \left[ t,T\right] }\Big|%
\int_{s}^{T}e^{\beta r}\langle \Delta Y^{t}\left( r\right) ,\Delta
Z^{t}\left( r\right) \rangle dW\left( r\right) \Big|\Big]\medskip \\
\displaystyle\leq 4\mathbb{E}\Big[\sup_{s\in \left[ t,T\right] }\Big|%
\int_{t}^{s}e^{\beta r}\langle \Delta Y^{t}\left( r\right) ,\Delta
Z^{t}\left( r\right) \rangle dW\left( r\right) \Big|\Big]\medskip \\
\displaystyle\leq \frac{1}{2}\mathbb{E}\big(\sup_{s\in \left[ t,T\right]
}e^{\beta s}|\Delta Y^{t}\left( s\right) |^{2}\big)+72\mathbb{E}%
\int_{t}^{T}e^{\beta r}|\Delta Z^{t}\left( r\right) |^{2}dr\,,%
\end{array}%
\end{equation*}%
which implies
\begin{equation*}
\begin{array}{l}
\displaystyle\mathbb{E}\big(\sup_{s\in \left[ t,T\right] }e^{\beta s}|\Delta
Y^{t}\left( s\right) |^{2}\big)\leq \frac{2Ke^{\beta \delta }}{a}T\mathbb{E}%
\big(\sup_{s\in \left[ 0,T\right] }e^{\beta s}|\Delta U^{t}\left( s\right)
|^{2}\big)+\frac{2Ke^{\beta \delta }}{a}\mathbb{E}\int_{0}^{T}e^{\beta
r}|\Delta V^{t}\left( r\right) |^{2}dr\medskip \\
\displaystyle\quad +2\mathbb{E}\Big[\sup_{s\in \left[ t,T\right] }\Big|%
\int_{s}^{T}e^{\beta r}\langle \Delta Y^{t}\left( r\right) ,\Delta
Z^{t}\left( r\right) \rangle dW\left( r\right) \Big|\Big]\medskip \\
\displaystyle\leq \frac{2Ke^{\beta \delta }}{a}T\mathbb{E}\big(\sup_{s\in %
\left[ 0,T\right] }e^{\beta s}|\Delta U^{t}\left( s\right) |^{2}\big)+\frac{%
2Ke^{\beta \delta }}{a}\mathbb{E}\int_{0}^{T}e^{\beta r}|\Delta V^{t}\left(
r\right) |^{2}dr\medskip \\
\displaystyle\quad +\frac{1}{2}\mathbb{E}\big(\sup_{s\in \left[ t,T\right]
}e^{\beta s}|\Delta Y^{t}\left( s\right) |^{2}\big)+72\mathbb{E}%
\int_{t}^{T}e^{\beta r}|\Delta Z^{t}\left( r\right) |^{2}dr\,.%
\end{array}%
\end{equation*}%
Hence, we have%
\begin{equation}
\begin{array}{l}
\displaystyle\mathbb{E}\big(\sup_{s\in \left[ t,T\right] }e^{\beta s}|\Delta
Y^{t}\left( s\right) |^{2}\big)\medskip \\
\displaystyle\leq \frac{4TKe^{\beta \delta }}{a}C_{1}\mathbb{E}\big(%
\sup_{s\in \left[ 0,T\right] }e^{\beta s}|\Delta U^{t}\left( s\right) |^{2}%
\big)+\frac{4Ke^{\beta \delta }}{a}C_{1}\mathbb{E}\int_{0}^{T}e^{\beta
r}|\Delta V^{t}\left( r\right) |^{2}dr,%
\end{array}
\label{technical ineq 5}
\end{equation}%
where we have denoted by $C_{1}:=1+\frac{72}{1-4L^{2}/a}\,.\medskip $

Let us now consider the term $\mathbb{E}\big(\sup_{s\in \left[ 0,t\right]
}e^{\beta s}|\Delta Y\left( s\right) |^{2}\big)$. From equation (\ref{BSDE
iterative 1}), we see that,
\begin{equation}
\begin{array}{l}
\displaystyle\mathbb{E}\big(\sup_{s\in \left[ 0,t\right] }e^{\beta s}|\Delta
Y^{t}\left( s\right) |^{2}\big)=\mathbb{E}\big(\sup_{s\in \left[ 0,t\right]
}e^{\beta s}|Y^{1,t}(s)-Y^{2,t}(s)|^{2}\big)\medskip \\
\displaystyle=\mathbb{E}\big(\sup_{s\in \left[ 0,t\right] }e^{\beta
s}|Y^{1,s}(s)-Y^{2,s}(s)|^{2}\big)=\sup_{s\in \left[ 0,t\right] }e^{\beta
s}|\Delta Y^{s}\left( s\right) |^{2}=\sup_{s\in \left[ 0,t\right] }\mathbb{E}%
\big(e^{\beta s}|\Delta Y^{s}\left( s\right) |^{2}\big)%
\end{array}
\label{technical ineq 6}
\end{equation}%
so that, exploiting It\^{o}'s formula and proceeding as above, we obtain%
\begin{equation}
\displaystyle\mathbb{E}\big(e^{\beta s}|\Delta Y^{s}\left( s\right) |^{2}%
\big)\leq \frac{2TKe^{\beta \delta }}{a}\mathbb{E}\big(\sup_{r\in \left[ 0,T%
\right] }e^{\beta r}|\Delta U^{s}\left( r\right) |^{2}\big)+\frac{2Ke^{\beta
\delta }}{a}\mathbb{E}\int_{0}^{T}e^{\beta r}|\Delta V^{s}\left( r\right)
|^{2}dr.  \label{technical ineq 6'}
\end{equation}%
Thus from inequalities (\ref{technical ineq 3}--\ref{technical ineq 6'}) we
obtain%
\begin{equation*}
\begin{array}{l}
\displaystyle\mathbb{E}\big(\sup_{s\in \left[ 0,T\right] }e^{\beta r}|\Delta
Y^{t}\left( s\right) |^{2}\big)+\mathbb{E}\int_{0}^{T}e^{\beta r}|\Delta
Z^{t}\left( r\right) |^{2}dr\medskip \\
\displaystyle\leq \frac{4TKe^{\beta \delta }}{a}C_{1}\mathbb{E}\big(%
\sup_{s\in \left[ 0,T\right] }e^{\beta s}|\Delta U^{t}\left( s\right) |^{2}%
\big)+\frac{4Ke^{\beta \delta }}{a}C_{1}\mathbb{E}\int_{0}^{T}e^{\beta
r}|\Delta V^{t}\left( r\right) |^{2}dr\medskip \\
\displaystyle\quad +\frac{2TKe^{\beta \delta }}{a\left( 1-4L^{2}/a\right) }%
\mathbb{E}\big(\sup_{r\in \left[ 0,T\right] }e^{\beta r}|\Delta U^{t}\left(
r\right) |^{2}\big)+\frac{2Ke^{\beta \delta }}{a\left( 1-4L^{2}/a\right) }%
\mathbb{E}\int_{0}^{T}e^{\beta r}|\Delta V^{t}\left( r\right) |^{2}dr\medskip
\\
\displaystyle\quad +\frac{2TKe^{\beta \delta }}{a}\cdot \sup_{s\in \left[ 0,t%
\right] }\mathbb{E}\big(\sup_{r\in \left[ 0,T\right] }e^{\beta r}|\Delta
U^{s}\left( r\right) |^{2}\big)+\frac{2Ke^{\beta \delta }}{a}\cdot
\sup_{s\in \left[ 0,t\right] }\mathbb{E}\int_{0}^{T}e^{\beta r}|\Delta
V^{s}\left( r\right) |^{2}dr.%
\end{array}%
\end{equation*}%
Passing then to the supremum for $t\in \left[ 0,T\right] $ we get%
\begin{equation*}
|||Y^{1}-Y^{2}|||_{1}^{2}+|||Z^{1}-Z^{2}|||_{2}^{2}\medskip \leq \frac{%
2Ke^{\beta \delta }}{a}\Big(3+\frac{145}{1-4L^{2}/a}\Big)\max \left\{
1,T\right\} \Big[|||U^{1}-U^{2}|||_{1}^{2}+|||V^{1}-V^{2}|||_{2}^{2}\Big].
\end{equation*}%
By choosing now $a:=\frac{4L^{2}}{\gamma }$ and $\beta $ slightly bigger
than $\gamma +\frac{4L^{2}}{\gamma }$, condition (\ref%
{restriction_parameters2}) is satisfied and, by restriction $\mathrm{(C)}$
we have%
\begin{equation}
\frac{2Ke^{\beta \delta }}{a}\Big(3+\frac{145}{1-4L^{2}/a}\Big)\max \left\{
1,T\right\} <1\,.  \label{restriction 1}
\end{equation}

Eventually, since $U$ and $V$ were chosen arbitrarily, it follows that the
application $\Gamma$ is a contraction on the space $\mathcal{A}\times
\mathcal{B}$. Therefore there exists a unique fixed point $\Gamma
(Y,Z)=(Y,Z)\in\mathcal{A}\times\mathcal{B}$ and this finishes the proof of
the existence and a uniqueness of a solution to our BSDE with delay.\hfill
\end{proof}

\bigskip

\begin{proof}[Proof of Theorem \protect\ref{theorem 2}]
Let us first prove the continuity of $\mathbb{\Lambda }\ni \phi \mapsto
u\left( t,\phi \right) $, uniformly with respect to $t\in \left[ 0,T\right] $%
.

Let us thus take $t\in \left[ 0,T\right] $, $\phi ,\phi ^{\prime }\in
\mathbb{\Lambda }$, and let us denote%
\begin{equation*}
\Delta Y\left( r\right) :=Y^{t,\phi }\left( r\right) -Y^{t,\phi ^{\prime
}}\left( r\right) ,\quad \Delta Z\left( r\right) :=Z^{t,\phi }\left(
r\right) -Z^{t,\phi ^{\prime }}\left( r\right)
\end{equation*}%
and%
\begin{equation*}
\Delta f\left( r\right) :=f(r,X^{t,\phi },Y^{t,\phi }\left( r\right)
,Z^{t,\phi }\left( r\right) ,Y_{r}^{t,\phi })-f(r,X^{t,\phi ^{\prime
}},Y^{t,\phi }\left( r\right) ,Z^{t,\phi }\left( r\right) ,Y_{r}^{t,\phi }).
\end{equation*}%
Using similar calculus as in the proof of Theorem 7, for $\beta ,a>0$ such
that%
\begin{equation}
\beta >a+\frac{6L^{2}}{a}\quad \text{and}\quad 1>\frac{6L^{2}}{a}\,,
\label{technical ineq 9}
\end{equation}%
we obtain%
\begin{equation}
\begin{array}{l}
\displaystyle\left( 1-\frac{6L^{2}}{a}\right) \mathbb{E}\int_{t}^{T}e^{\beta
r}|\Delta Z\left( r\right) |^{2}dr\medskip \\
\displaystyle\leq \mathbb{E}\big(e^{\beta T}|\Delta Y\left( T\right) |^{2}%
\big)+\frac{3}{a}\mathbb{E}\int_{t}^{T}e^{\beta r}|\Delta f\left( r\right)
|^{2}dr+\frac{3Ke^{\beta \delta }}{a}\mathbb{E}\int_{t-\delta }^{T}e^{\beta
r}\left\vert \Delta Y(r)\right\vert ^{2}dr.%
\end{array}
\label{technical ineq 10}
\end{equation}%
Denoting%
\begin{equation}
C_{1}:=1+\frac{144}{1-6L^{2}/a}  \label{technical ineq 11}
\end{equation}%
and choosing $a:=\frac{6L^{2}}{\gamma }$ and $\beta $ bigger than $\gamma +%
\frac{6L^{2}}{\gamma }$, condition (\ref{technical ineq 9}) is satisfied and%
\begin{equation}
T\frac{3Ke^{\beta \delta }}{a}C_{1}<\frac{1}{4}\,,  \label{restriction 2}
\end{equation}%
by restriction $\mathrm{(C)}$. We then have%
\begin{equation}
\begin{array}{l}
\displaystyle\frac{1}{2}\mathbb{E}\big(\sup_{s\in \left[ t,T\right]
}e^{\beta s}|\Delta Y\left( s\right) |^{2}\big)\medskip \\
\displaystyle\leq C_{1}\mathbb{E}\big(e^{\beta T}|\Delta Y\left( T\right)
|^{2}\big)+\frac{3C_{1}}{a}\mathbb{E}\int_{t}^{T}e^{\beta r}|\Delta f\left(
r\right) |^{2}dr+\frac{3K\delta e^{\beta \delta }}{a}C_{1}\mathbb{E}\big(%
\sup_{s\in \left[ t-\delta ,t\right] }e^{\beta s}|\Delta Y\left( s\right)
|^{2}\big).%
\end{array}
\label{technical ineq 13}
\end{equation}%
Exploiting the initial conditions satisfied by the $Y^{t,\phi }$, we can
rewrite equation \eqref{technical ineq 13} as%
\begin{equation*}
\begin{array}{l}
\displaystyle\mathbb{E}\big(\sup_{s\in \left[ t,T\right] }e^{\beta s}|\Delta
Y\left( s\right) |^{2}\big)\leq 2C_{1}\mathbb{E}\big(e^{\beta T}|\Delta
Y\left( T\right) |^{2}\big)+\frac{6C_{1}}{a}\mathbb{E}\int_{t}^{T}e^{\beta
r}|\Delta f\left( r\right) |^{2}dr\medskip \\
\displaystyle\quad +\frac{6K\delta e^{\beta \delta }}{a}C_{1}\sup_{s\in %
\left[ t-\delta ,t\right] }\mathbb{E}\big(\sup_{r\in \left[ s,T\right]
}e^{\beta \left( s-r\right) }e^{\beta r}|Y^{s,\phi }\left( r\right)
-Y^{s,\phi ^{\prime }}\left( r\right) |^{2}\big).%
\end{array}%
\end{equation*}%
Passing to the supremum for $t\in \left[ 0,T\right] $ we have%
\begin{equation}
\begin{array}{l}
\displaystyle\sup_{t\in \left[ 0,T\right] }\mathbb{E}\big(\sup_{s\in \left[
t,T\right] }e^{\beta s}|Y^{t,\phi }\left( s\right) -Y^{t,\phi ^{\prime
}}\left( s\right) |^{2}\big)\medskip \\
\displaystyle\leq 2C_{1}\sup_{t\in \left[ 0,T\right] }\mathbb{E}\big(%
e^{\beta T}|h(X^{t,\phi })-h(X^{t,\phi ^{\prime }})|^{2}\big)+\frac{6C_{1}}{a%
}\sup_{t\in \left[ 0,T\right] }\mathbb{E}\int_{t}^{T}e^{\beta r}|\Delta
f\left( r\right) |^{2}dr\medskip \\
\displaystyle\quad +\frac{6K\delta e^{\beta \delta }}{a}C_{1}\sup_{s\in %
\left[ 0,T\right] }\mathbb{E}\big(\sup_{r\in \left[ s,T\right] }e^{\beta
r}|Y^{s,\phi }\left( r\right) -Y^{s,\phi ^{\prime }}\left( r\right) |^{2}%
\big).%
\end{array}
\label{technical ineq 7}
\end{equation}%
We can now see that
\begin{equation}
\begin{array}{l}
\displaystyle\mathbb{E}\big(\sup_{s\in \left[ 0,t\right] }e^{\beta
s}|Y^{t,\phi }\left( s\right) -Y^{t,\phi ^{\prime }}\left( s\right) |^{2}%
\big)=\sup_{s\in \left[ 0,t\right] }e^{\beta s}|Y^{s,\phi }\left( s\right)
-Y^{s,\phi ^{\prime }}\left( s\right) |^{2}\medskip \\
\displaystyle\leq \sup_{s\in \left[ 0,T\right] }\mathbb{E}\big(\sup_{r\in %
\left[ s,T\right] }e^{\beta r}|Y^{s,\phi }\left( r\right) -Y^{s,\phi
^{\prime }}\left( r\right) |^{2}\big)\,,%
\end{array}
\label{technical ineq 8}
\end{equation}%
so that we can apply again inequality (\ref{technical ineq 7}).

From inequalities (\ref{technical ineq 7}) and (\ref{technical ineq 8}) we
can conclude that
\begin{equation*}
\begin{array}{l}
\displaystyle\sup_{t\in\left[ 0,T\right] }\mathbb{E}\big(\sup_{s\in\left[ 0,T%
\right] }e^{\beta s}|Y^{t,\phi}\left( s\right) -Y^{t,\phi^{\prime}}\left(
s\right) |^{2}\big)\medskip \\
\displaystyle\leq4C_{1}\sup_{t\in\left[ 0,T\right] }\mathbb{E}\big(e^{\beta
T}|h(X^{t,\phi})-h(X^{t,\phi^{\prime}})|^{2}\big)+\frac{12C_{1}}{a}\sup
_{t\in\left[ 0,T\right] }\mathbb{E}\int_{t}^{T}e^{\beta r}|\Delta f\left(
r\right) |^{2}dr\medskip \\
\displaystyle\quad+\frac{12K\delta e^{\beta\delta}}{a}C_{1}\sup_{s\in\left[
0,T\right] }\mathbb{E}\big(\sup_{r\in\left[ s,T\right] }e^{\beta
r}|Y^{s,\phi}\left( r\right) -Y^{s,\phi^{\prime}}\left( r\right) |^{2}\big).%
\end{array}%
\end{equation*}
Since $\delta\leq T$, by (\ref{restriction 2}) we also have $\frac{12K\delta
e^{\beta\delta}}{a}C_{1}<1$ and so%
\begin{equation*}
\begin{array}{l}
\displaystyle\left( 1-\frac{12K\delta e^{\beta\delta}}{a}C_{1}\right)
\sup_{t\in\left[ 0,T\right] }\mathbb{E}\big(\sup_{s\in\left[ 0,T\right]
}e^{\beta s}|Y^{t,\phi}\left( s\right) -Y^{t,\phi^{\prime}}\left( s\right)
|^{2}\big)\medskip \\
\displaystyle\leq4C_{1}\sup_{t\in\left[ 0,T\right] }\mathbb{E}\big(e^{\beta
T}|h(X^{t,\phi})-h(X^{t,\phi^{\prime}})|^{2}\big)\medskip \\
\displaystyle\quad+\frac{12C_{1}}{a}\sup_{t\in\left[ 0,T\right] }\mathbb{E}%
\int_{t}^{T}e^{\beta r}|f(r,X^{t,\phi},Y^{t,\phi}\left( r\right)
,Z^{t,\phi}\left( r\right) ,Y_{r}^{t,\phi})+\medskip \\
\qquad\qquad\qquad\qquad\qquad\,\,\displaystyle-f(r,X^{t,\phi^{%
\prime}},Y^{t,\phi}\left( r\right) ,Z^{t,\phi}\left( r\right) ,Y_{r}^{t,\phi
})|^{2}dr\,.%
\end{array}%
\end{equation*}

Let us now fix $\phi\in\mathbb{\Lambda}$. In order to prove that $u$ is
continuous in $\phi$, uniformly with respect to $t\in\lbrack0,T]$, it is
enough to show that%
\begin{align*}
& \mathbb{E}\big(|h(X^{t,\phi})-h(X^{t,\phi^{\prime}})|^{2}\big) \\
& +\mathbb{E}\int_{0}^{T}|f(r,X^{t,\phi},Y^{t,\phi}\left( r\right)
,Z^{t,\phi}\left( r\right)
,Y_{r}^{t,\phi})-f(r,X^{t,\phi^{\prime}},Y^{t,\phi}\left( r\right)
,Z^{t,\phi}\left( r\right) ,Y_{r}^{t,\phi })|^{2}dr
\end{align*}
converge to $0$ as $\phi^{\prime}\rightarrow\phi$, uniformly in $t\in
\lbrack0,T]$.

Since we have no guarantee that the family $\big\{\big|Z^{t,\phi }\big|^{2}%
\big\}_{t\in \lbrack 0,T]}$ is uniformly integrable, we will use the
Lipschitz property of $f$ in the argument $\left( y,z,u\right) $ in order to
replace $[0,T]$ with a finite subset. By Theorem 7, the mapping $t\mapsto
\left( Y^{t,\phi },Z^{t,\phi }\right) $ is continuous from $[0,T]$ into $%
\mathcal{S}_{0}^{2,1}\times \mathcal{H}_{0}^{2,d^{\prime }}$ and therefore
uniformly continuous. Consequently, as $n\rightarrow \infty $, we have%
\begin{equation*}
\sup_{\left\vert t-t^{\prime }\right\vert \leq \frac{1}{n}}\mathbb{E}\left[
\sup_{s\in \lbrack 0,T]}\big(Y^{t,\phi }(s)-Y^{t^{\prime },\phi }(s)\big)%
^{2}+\int_{0}^{T}\big(Z^{t,\phi }(s)-Z^{t^{\prime },\phi }(s)\big)^{2}ds%
\right] \rightarrow 0.
\end{equation*}

Let, for $n\in\mathbb{N}^{\ast}$, $\pi_{n}:=\{0,\frac{T}{n},\dots ,\frac{%
(n-1)T}{n},T\}$, then, by $\mathrm{(A}_{6}\mathrm{)}$, we see that
\begin{equation*}
\displaystyle\sup_{t\in\lbrack0,T]}\sup_{t^{\prime}\in\pi_{n}}\mathbb{E}%
\int_{0}^{T}|f(r,X^{t,\phi},Y^{t^{\prime},\phi}\left( r\right)
,Z^{t^{\prime},\phi}\left( r\right) ,Y_{r}^{t^{\prime},\phi})-f(r,X^{t,\phi
^{\prime}},Y^{t^{\prime},\phi}\left( r\right) ,Z^{t^{\prime},\phi}\left(
r\right) ,Y_{r}^{t^{\prime},\phi})|^{2}dr\,,
\end{equation*}
converges to%
\begin{equation*}
\sup_{t\in\lbrack0,T]}\mathbb{E}\int_{0}^{T}|f(r,X^{t,\phi},Y^{t,\phi}\left(
r\right) ,Z^{t,\phi}\left( r\right) ,Y_{r}^{t,\phi})-f(r,X^{t,\phi^{\prime
}},Y^{t,\phi}\left( r\right) ,Z^{t,\phi}\left( r\right) ,Y_{r}^{t,\phi
})|^{2}dr,
\end{equation*}
uniformly in $\phi^{\prime}$. We are thus left to prove that%
\begin{align*}
& \mathbb{E}\big(|h(X^{t,\phi})-h(X^{t,\phi^{\prime}})|^{2}\big) \\
& +\mathbb{E}\int_{0}^{T}|f(r,X^{t,\phi},Y^{t^{\prime},\phi}\left( r\right)
,Z^{t^{\prime},\phi}\left( r\right) ,Y_{r}^{t^{\prime},\phi})-f(r,X^{t,\phi
^{\prime}},Y^{t^{\prime},\phi}\left( r\right) ,Z^{t^{\prime},\phi}\left(
r\right) ,Y_{r}^{t^{\prime},\phi})|^{2}dr
\end{align*}
converge to $0$ as $\phi^{\prime}\rightarrow\phi$, uniformly in $t\in
\lbrack0,T]$, for fixed $n\in\mathbb{N}^{\ast}$ and $t^{\prime}\in\pi_{n}$.

Let us thus introduce the modulus of continuity of the functions $h$ and $f$:

\begin{equation*}
m_{h,f}(\epsilon,\mathcal{K},\mathcal{U},\mathcal{\kappa}):=\sup _{\underset{%
|z|\leq\kappa,~||\phi-\phi^{\prime}||_{T}\leq\epsilon}{\phi^{\prime},\phi^{%
\prime\prime}\in\mathcal{K},\;t\in\lbrack0,T],\ (y,u)\in \mathcal{U}}%
}(|h(\phi^{\prime})-h(\phi^{\prime\prime})|+|f(t,\phi^{\prime
},y,z,u)-f(t,\phi^{\prime\prime},y,z,u)|)\,,
\end{equation*}
where $\epsilon>0$, $\mathcal{K}$\ is a compact in $\mathbb{\Lambda}$, $%
\mathcal{U}$ is a compact in $\mathbb{R}\times L^{2}\left( \left[ -\delta,0%
\right] ;\mathbb{R}\right) $ and $\kappa\in\mathbb{R}_{+}$.

Let $\epsilon>0$ be fixed, but arbitrary; with no loss generality, we can
suppose that the function $\phi^{\prime}$ lies in a compact $\mathcal{%
K\subseteq}\mathbb{\Lambda}$.

Using the estimates satisfied by the solution, we deduce that the family $%
\big\{\big(X^{t,\phi ^{\prime }},Y^{t^{\prime },\phi }\big)\big\}_{(t,\phi
^{\prime })\in \lbrack 0,T]\times \mathcal{K}}$ is tight with respect to the
product topology on $\mathbb{\Lambda }\times \mathcal{C}([0,T])$, and
therefore, for every $\epsilon >0$, there exist compact subsets $\mathcal{K}%
_{\epsilon }\subseteq \mathbb{\Lambda }$ and $\mathcal{K}_{\epsilon
}^{\prime }\subseteq \mathcal{C}([0,T])$ such that%
\begin{equation*}
\mathbb{P}\left( X^{t,\phi ^{\prime }}\in \mathcal{K}_{\epsilon },\
Y^{t^{\prime },\phi }\in \mathcal{K}_{\epsilon }^{\prime }\right) \geq
1-\epsilon ,\quad \text{for all }\left( t,\phi ^{\prime }\right) \in \lbrack
0,T]\times \mathcal{K}.
\end{equation*}%
For ease of the notation, let us define $\Phi :\mathbb{\Lambda }\times
\mathbb{\Lambda }\times \mathcal{C}([0,T])\times \mathbb{R}^{d}\rightarrow
\mathbb{R}$ by%
\begin{equation*}
\Phi \left( r,\phi ^{\prime },\phi ^{\prime \prime },y,z\right) :=\frac{1}{T}%
\left\vert h\left( \phi ^{\prime }\right) -h\left( \phi ^{\prime \prime
}\right) \right\vert ^{2}+|f(r,\phi ^{\prime },y\left( r\right)
,z,y_{r})-f(r,\phi ^{\prime \prime },y\left( r\right) ,z,y_{r})|^{2}.
\end{equation*}%
We can see by $\mathrm{(A}_{6}\mathrm{)}$, that it holds
\begin{equation*}
\Phi \left( \phi ^{\prime },\phi ^{\prime \prime },y,z\right) \leq C\left(
1+||\phi ^{\prime }||_{T}^{2p}+||\phi ^{\prime \prime
}||_{T}^{2p}+||y||_{T}^{2}+\left\vert z\right\vert ^{2}\right) ,
\end{equation*}%
where in what follows we will denote by $C$ several possibly different
constants depending only on $K$, $L$, $M$ and $T$. Then, for all $t\in
\lbrack 0,T]$, $\phi ^{\prime },\phi ^{\prime \prime }\in \mathbb{\Lambda },$
we have, from the a priori estimates on the processes $Y^{t,\phi }$ and $%
Z^{t,\phi }$,%
\begin{equation*}
\mathbb{E}\left[ \int_{0}^{T}\Phi \left( r,X^{t,\phi ^{\prime }},X^{t,\phi
^{\prime \prime }},Y^{t^{\prime },\phi },Z^{t^{\prime },\phi }(r)\right) dr%
\right] ^{p^{\prime }}\leq C\left( 1+||\phi ^{\prime }||_{T}^{pp^{\prime
}}+||\phi ^{\prime \prime }||_{T}^{pp^{\prime }}\right) \,.
\end{equation*}

Let now $\mathcal{U}_{\epsilon}$ be the image of $[0,T]\times\mathcal{K}%
_{\epsilon}^{\prime}$ through the continuous application $\left( r,y\right)
\mapsto(y(r),y_{r})\,,$ and we also have that $\mathcal{U}$ is compact in $%
\mathbb{R}\times L^{2}\left( \left[ -\delta,0\right] ;\mathbb{R}\right) $.

For arbitrary $\epsilon ^{\prime },\kappa >0$, we see that,%
\begin{equation*}
\begin{array}{l}
\displaystyle\mathbb{E}\int_{0}^{T}\Phi \left( r,X^{t,\phi },X^{t,\phi
^{\prime }},Y^{t^{\prime },\phi },Z^{t^{\prime },\phi }(r)\right) dr\medskip
\\
\displaystyle\leq \mathbb{E}\int_{0}^{T}\Phi \left( X^{t,\phi },X^{t,\phi
^{\prime }},Y^{t^{\prime },\phi },Z^{t^{\prime },\phi }(r)\right) \medskip
\\
\displaystyle\quad \quad \quad \cdot \mathbb{1}_{\left\{ (X^{t,\phi
},Y^{t^{\prime },\phi }),(X^{t,\phi ^{\prime }},Y^{t^{\prime },\phi })\in
\mathcal{K}_{\epsilon }\mathcal{\times K}_{\epsilon }^{\prime
},|Z^{t^{\prime },\phi }(r)|\leq \kappa ,||X^{t,\phi }-X^{t,\phi ^{\prime
}}||_{T}\leq \epsilon ^{\prime }\right\} }dr\medskip \\
\displaystyle\quad +\mathbb{E}\int_{0}^{T}\Phi \left( X^{t,\phi },X^{t,\phi
^{\prime }},Y^{t^{\prime },\phi },Z^{t^{\prime },\phi }(r)\right) \mathbb{1}%
_{\left\{ (X^{t,\phi },Y^{t^{\prime },\phi })\not\in \mathcal{K}_{\epsilon }%
\mathcal{\times K}_{\epsilon }^{\prime }\right\} }dr\medskip \\
\displaystyle\quad +\mathbb{E}\int_{0}^{T}\Phi \left( X^{t,\phi },X^{t,\phi
^{\prime }},Y^{t^{\prime },\phi },Z^{t^{\prime },\phi }(r)\right) \mathbb{1}%
_{\left\{ (X^{t,\phi },Y^{t^{\prime },\phi })\not\in \mathcal{K}_{\epsilon }%
\mathcal{\times K}_{\epsilon }^{\prime }\right\} }dr\medskip \\
\displaystyle\quad +\mathbb{E}\int_{0}^{T}\Phi \left( X^{t,\phi },X^{t,\phi
^{\prime }},Y^{t^{\prime },\phi },Z^{t^{\prime },\phi }(r)\right) \mathbb{1}%
_{\left\{ |Z^{t^{\prime },\phi }(r)|>\kappa \right\} }dr\medskip \\
\displaystyle\quad +\mathbb{E}\int_{0}^{T}\Phi \left( X^{t,\phi },X^{t,\phi
^{\prime }},Y^{t^{\prime },\phi },Z^{t^{\prime },\phi }(r)\right) \mathbb{1}%
_{\left\{ ||X^{t,\phi }-X^{t,\phi ^{\prime }}||_{T}>\epsilon ^{\prime
}\right\} }dr%
\end{array}%
\end{equation*}%
and therefore%
\begin{equation*}
\begin{array}{l}
\displaystyle\mathbb{E}\int_{0}^{T}\Phi \left( r,X^{t,\phi },X^{t,\phi
^{\prime }},Y^{t^{\prime },\phi },Z^{t^{\prime },\phi }(r)\right) dr\medskip
\\
\displaystyle\leq Tm_{h,f}\left( \epsilon ^{\prime },\mathcal{K}_{\epsilon },%
\mathcal{U}_{\epsilon },\kappa \right) +2\left\{ \mathbb{E}\left[
\int_{0}^{T}\Phi \left( r,X^{t,\phi },X^{t,\phi ^{\prime }},Y^{t^{\prime
},\phi },Z^{t^{\prime },\phi }(r)\right) dr\right] ^{p^{\prime }}\right\}
^{1/p^{\prime }}\epsilon ^{1-\frac{1}{p^{\prime }}}\medskip \\
\displaystyle\quad +C\mathbb{E}\left[ \left( 1+||X^{t,\phi
}||_{T}^{2p}+||X^{t,\phi ^{\prime }}||_{T}^{2p}+||Y^{t^{\prime },\phi
}||_{T}^{2}\right) \right] \int_{0}^{T}\mathbb{1}_{\left\{ |Z^{t^{\prime
},\phi }(r)|>\kappa \right\} }dr\medskip \\
\displaystyle\quad +C\mathbb{E}\int_{0}^{T}\left\vert Z^{t^{\prime },\phi
}(r)\right\vert ^{2}\mathbb{1}_{\left\{ |Z^{t^{\prime },\phi }(r)|>\kappa
\right\} }dr\medskip \\
\displaystyle\quad +\left\{ \mathbb{E}\left[ \int_{0}^{T}\Phi \left(
r,X^{t,\phi ^{\prime }},X^{t,\phi ^{\prime \prime }},Y^{t^{\prime },\phi
},Z^{t^{\prime },\phi }(r)\right) dr\right] ^{p^{\prime }}\right\}
^{1/p^{\prime }}~\left[ \mathbb{P}\left( ||X^{t,\phi }-X^{t,\phi ^{\prime
}}||_{T}>\epsilon ^{\prime }\right) \right] ^{1-\frac{1}{p^{\prime }}%
}\medskip \\
\displaystyle\leq Tm_{h,f}\left( \epsilon ^{\prime },\mathcal{K}_{\epsilon },%
\mathcal{U},\kappa \right) +C\,\mathbb{E}\int_{0}^{T}|Z^{t^{\prime },\phi
}(r)|^{2}\mathbb{1}_{\left\{ |Z^{t^{\prime },\phi }(r)|>\kappa \right\}
}dr\medskip \\
\displaystyle\quad +C\left( 1+||\phi ||_{T}^{p}+||\phi ^{\prime
}||_{T}^{p}\right) \Big[\epsilon ^{1-\frac{1}{p^{\prime }}}+\frac{\mathbb{E}%
||X^{t,\phi }-X^{t,\phi ^{\prime }}||_{T}^{p^{\prime }-1}}{\left( \epsilon
^{\prime }\right) ^{p^{\prime }-1}}+\frac{\big(\mathbb{E}\int_{0}^{T}|Z^{t^{%
\prime },\phi }(r)|^{2}dr\big)^{1/2}}{\kappa ^{2}}\Big].%
\end{array}%
\end{equation*}%
Therefore we have%
\begin{equation*}
\begin{array}{l}
\displaystyle\sup_{t\in \lbrack 0,T]}\mathbb{E}\int_{0}^{T}\Phi \left(
r,X^{t,\phi },X^{t,\phi ^{\prime }},Y^{t^{\prime },\phi },Z^{t^{\prime
},\phi }(r)\right) dr\medskip \\
\displaystyle\leq Tm_{h,f}\left( \epsilon ^{\prime },\mathcal{K}_{\epsilon },%
\mathcal{U}_{\epsilon },\kappa \right) +C\mathbb{E}\int_{0}^{T}|Z^{t^{\prime
},\phi }(r)|^{2}\mathbb{1}_{\left\{ |Z^{t^{\prime },\phi }(r)|>\kappa
\right\} }dr\medskip \\
\displaystyle\quad +C\left( 1+\left\Vert \phi \right\Vert
_{T}^{p}+\left\Vert \phi ^{\prime }\right\Vert _{T}^{p}\right) \Big[\epsilon
^{1-\frac{1}{p^{\prime }}}+\frac{\left( 1+||\phi ||_{T}^{p-1}+||\phi
^{\prime }||_{T}^{p-1}\right) \mathbb{E}||\phi -\phi ^{\prime
}||_{T}^{p^{\prime }-1}}{\left( \epsilon ^{\prime }\right) ^{p^{\prime }-1}}+%
\frac{1}{\kappa ^{2}}\Big].%
\end{array}%
\end{equation*}%
Passing now to the limit as $\phi ^{\prime }\rightarrow \phi $, $\epsilon
^{\prime }\rightarrow 0$, $\left( \epsilon ,\kappa \right) \rightarrow
(0,+\infty )$, we obtain the claim.

Concerning the continuity of $\left[ 0,T\right] \ni t\rightarrow u\left(
t,\phi \right) $, this is an immediate consequence of the continuity of the
stochastic process $Y^{t,\phi }$, together with the continuity of the
mapping $t\mapsto Y^{t,\phi }$ from $[0,T]$ into $\mathcal{S}_{0}^{2,1}$%
.\hfill \bigskip
\end{proof}

\section*{Acknowledgement}

The author L. Maticiuc thanks the Department of Computer Science, University
of Verona, for its hospitality. The work of this author was supported by
grant \textquotedblleft Deterministic and stochastic systems with state
constraints\textquotedblright , code $241/05.10.2011.\medskip $

The fourth named author gratefully acknowledges the support of the 2010 PRIN
project: \textquotedblleft Equazioni di evoluzione stocastiche con controllo
e rumore al bordo\textquotedblright .

\addcontentsline{toc}{chapter}{References}


\bigskip

\addcontentsline{toc}{chapter}{Acknowledgement}

\begin{thebibliography}{99}
\bibitem{ar-hu/07} M. Arriojas, Y. Hu, S.-E. Mohammed, G. Pap, \textit{A
delayed Black and Scholes formula}, Stochastic Analysis and Applications 25,
471--492 (2007).\vspace{-0.15cm}

\bibitem{ar-de-eb/02} P. Artzner, F. Delbaen, J.-M. Eber, D. Heath, \textit{%
Coherent Measures of Risk}. Risk management: value at risk and beyond
(Cambridge, 1998), 145--175, Cambridge Univ. Press, Cambridge (2002).\vspace{%
-0.15cm}

\bibitem{ba-ka/07} P. Barrieu, N. El Karoui, \textit{Pricing, hedging and
optimally designing derivatives via minimization of risk measures}, preprint
(2007), \href{http://papers.ssrn.com/sol3/papers.cfm?abstract_id=1005308}{%
papers.ssrn.com}.\vspace{-0.15cm}

\bibitem{bi/73} J.M. Bismut, \textit{Conjugate convex functions in optimal
stochastic control}, Journal of Mathematical Analysis and Applications 44,
384--404 (1973).\vspace{-0.15cm}

\bibitem{ch-yo/99} M.--H. Chang, R.K. Youree, \textit{The European option
with hereditary price structures: Basic theory}, Applied Mathematics and
Computation 102, 279--296 (1999).\vspace{-0.15cm}

\bibitem{ch-na/17} P. Cheridito, K. Nam, \textit{BSE's, BSDE's and fixed
point problems}, \href{https://arxiv.org/abs/1410.1247}{%
https://arxiv.org/abs/1410.1247} (2017).\vspace{-0.15cm}

\bibitem{co-fo/13} R. Cont, D.A. Fourni\'{e}, \textit{Functional It\^{o}
calculus and stochastic integral representation of martingales}, Ann.
Probab. 41, 109--133 (2013).\vspace{-0.15cm}

\bibitem{cv-ma/96} J. Cvitanic, Jin Ma, \textit{Hedging options for a large
investor and forward-backward SDE's}, The annals of applied probability 6,
370--398 (1996).\vspace{-0.15cm}

\bibitem{da-pa/97} R.W.R. Darling, E. Pardoux, \textit{Backward SDE with
random terminal time and applications to semilinear elliptic PDE}, Ann.
Probab. 25, 1135--1159 (1997).\vspace{-0.15cm}

\bibitem{de/11} \L . Delong, \textit{Applications of time--delayed backward
stochastic differential equations to pricing, hedging and portfolio
management in insurance and finance}, Appl. Math. (Warsaw) 39, 463--488
(2012).\vspace{-0.15cm}

\bibitem{de/12} \L . Delong, \textit{BSDEs with time-delayed generators of a
moving average type with applications to non-monotone preferences},
Stochastic Models 28, 281--315 (2012).\vspace{-0.15cm}

\bibitem{de/13} \L . Delong, \textit{Backward stochastic differential
equations with jumps and their actuarial and financial applications},
London, Springer, 2013.\vspace{-0.15cm}

\bibitem{de-im/10} \L . Delong, P. Imkeller, \textit{Backward stochastic
differential equations with time delayed generators -- results and
counterexamples}, The Annals of Applied Probability 20, 1512--1536 (2010).%
\vspace{-0.15cm}

\bibitem{de-im/10x} \L . Delong, P. Imkeller, \textit{On Malliavin's
differentiability of BSDE with time delayed generators driven by Brownian
motions and Poisson random measures}, Stochastic Process. Appl. 120,
1748--1775 (2010).\vspace{-0.15cm}

\bibitem{du/09} B. Dupire, \textit{Functional It\^{o} calculus}, preprint
(2009), \href{http://papers.ssrn.com/sol3/papers.cfm?abstract_id=1435551}{%
papers.ssrn.com}.\vspace{-0.15cm}

\bibitem{ek-ke-to-zh/14} I. Ekren, C. Keller, N. Touzi, J. Zhang, \textit{On
viscosity solutions of path dependent PDEs}, Ann. Probab. 42, 204--236
(2014).\vspace{-0.15cm}

\bibitem{ek-to-zh/14} I. Ekren, C. Keller, N. Touzi, J. Zhang, \textit{%
Viscosity Solutions of Fully Nonlinear Parabolic Path Dependent PDEs: Part I}%
, \href{http://arxiv.org/abs/1210.0006v3}{http://arxiv.org/abs/1210.0006v3}
(2014).\vspace{-0.15cm}

\bibitem{ek-to-zh/14x} I. Ekren, C. Keller, N. Touzi, J. Zhang, \textit{%
Viscosity Solutions of Fully Nonlinear Parabolic Path Dependent PDEs: Part
II, }\href{http://arxiv.org/abs/1210.0007v3}{http://arxiv.org/abs/1210.0007v3%
} (2014).\vspace{-0.15cm}

\bibitem{ka-ka/97} N. El Karoui, C. Kapoudjian, E. Pardoux, S. Peng, M.--C.
Quenez, \textit{Reflected solutions of backward SDE's and related obstacle
problems for PDE's}, Ann.Probab. 25, 702--737 (1997).\vspace{-0.15cm}

\bibitem{ka-pe-qu/97} N. El Karoui, S. Peng, M.C. Quenez, \textit{Backward
stochastic differential equations in finance}, Mathematical Finance 7, 1--71
(1997).\vspace{-0.15cm}

\bibitem{ka-pe-qu/01} N. El Karoui, S. Peng, M. C. Quenez, \textit{A dynamic
maximum principle for the optimization of recursive utilities under
constraints}, The Annals of Applied Probability 11, 664--693 (2001).\vspace{%
-0.15cm}

\bibitem{fu-te/05} M. Fuhrman, G. Tessitore, \textit{Generalized directional
gradients, backward stochastic differential equations and mild solutions of
semilinear parabolic equations}, Applied Mathematics and Optimization 51,
279--332 (2005).\vspace{-0.15cm}

\bibitem{fu-te/10} M. Fuhrman, F. Masiero, G. Tessitore, \textit{Stochastic
Equations with Delay: Optimal Control via BSDEs and Regular Solutions of
Hamilton--Jacobi--Bellman Equations, }SIAM J. Control Optim. 48, 4624--4651
(2010).\vspace{-0.15cm}

\bibitem{ka-sw-hu/02} Y. Kazmerchuk, A. Swishchuk, J. Wu, \textit{%
Black--Scholes formula for security markets with delayed response},
Bachelier Fin. Soc., Second World Congress, Crete (2002).\vspace{-0.15cm}

\bibitem{ka-sw-hu/07} Y. Kazmerchuk, A. Swishchuk, J. Wu, \textit{The
pricing of options for securities markets with delayed response},
Mathematics and Computers in Simulation 75, 69--79 (2007).\vspace{-0.15cm}

\bibitem{ma-ra/10} L. Maticiuc, A. R\u{a}\c{s}canu, \textit{A stochastic
approach to a multivalued Dirichlet--Neumann problem}, Stochastic Processes
and their Applications 120, 777--800 (2010).\vspace{-0.15cm}

\bibitem{ma-ra/15} L. Maticiuc, A. R\u{a}\c{s}canu, \textit{Backward
Stochastic Variational Inequalities on Random Interval}, Bernoulli 21,
1166--1199 (2015).\vspace{-0.15cm}

\bibitem{mo/96} S.--E. A. Mohammed, \textit{Stochastic differential systems
with memory: Theory, examples and applications}, Laurent Decreusefond (ed.)
et al., Stochastic analysis and related topics VI. (1998), Proceedings of
the 6th Oslo--Silivri workshop, Geilo, Norway, July 29-August 6, 1996.
Boston, Birkh\"{a}user. Prog. Probab. 42, 1--77.\vspace{-0.15cm}

\bibitem{mo/84} S.--E. A. Mohammed, \textit{Stochastic functional
differential equations}, Research Notes in Mathematics 99, Pitman, Boston,
MA, 1984.\vspace{-0.15cm}

\bibitem{pa-pe/90} E. Pardoux, S. Peng, \textit{Adapted solution of a
backward stochastic differential equation}, Systems Control Lett. 14, 55--61
(1990).\vspace{-0.15cm}

\bibitem{pa-pe/92} E. Pardoux, S. Peng, \textit{Backward SDE's and
quasilinear parabolic PDE's}, Stochastic PDE and Their Applications (B.L.
Rozovskii, R.B. Sowers eds.), 200--217, LNCIS 176 (1992), Springer.\vspace{%
-0.15cm}

\bibitem{pa-zh/98} E. Pardoux, S. Zhang, \textit{Generalized BSDE and
nonlinear Neumann boundary value problems}, Probab. Theory Relat. Fields
110, 535--558 (1998).

\bibitem{pe/97} S. Peng, \textit{Backward SDE and related }$g$\textit{%
-expectation}, Backward stochastic differential equations, Pitman Research
Notes in Mathematics Series, Longman, Harlow, 141--159 (1997).\vspace{-0.15cm%
}

\bibitem{pe/10} S. Peng, \textit{Backward Stochastic Differential Equation,
Nonlinear Expectation and Their Applications}, Proceedings of the
International Congress of Mathematicians Hyderabad, India (2010).\vspace{%
-0.15cm}

\bibitem{pe/04} S. Peng, \textit{Nonlinear expectations, nonlinear
evaluations and risk measures}. Stochastic methods in finance. Springer
Berlin Heidelberg, 2004. 165-253.\vspace{-0.15cm}

\bibitem{pe/12} S. Peng, \textit{Note on Viscosity Solution of
Path-Dependent PDE and }$G$\textit{-Martingales}, \href{http://arxiv.org/abs/1106.1144v2%
}{http://arxiv.org/abs/1106.1144v2} (2012).\vspace{-0.15cm}

\bibitem{pe/91} S. Peng, \textit{Probabilistic Interpretation for Systems of
Quasilinear Parabolic Partial Differential Equation}, Stochastics 37, 61--74
(1991).\vspace{-0.15cm}

\bibitem{pe-wa/11} S. Peng, F. Wang, \textit{BSDE, Path--dependent PDE and
Nonlinear Feynman--Kac Formula}, Science China Mathematics 59, 19--36 (2016).%
\vspace{-0.15cm}


\bibitem{ro/06} E. Rosazza Gianin, \textit{Risk measures via }$g$\textit{%
-expectations}, Insurance Mathematics and Economics 39, 19--34 (2006).%
\vspace{-0.15cm}

\bibitem{uv-pl/07} K. Uwe, E. Platen, \textit{Time Delay and Noise
Explaining Cyclical Fluctuations in Prices of Commodities}, Quantitative
Finance Research Centre, Research paper 195 (2007).\vspace{-0.15cm}

\bibitem{za/12} A. Z\u{a}linescu, \textit{Second order
Hamilton-Jacobi-Bellman equations with an unbounded operator}, Nonlinear
Anal.-Theory Methods Appl. 75, 4784--4797 (2012).
\end{thebibliography}
\end{document}